\documentclass[12pt]{article}
\usepackage{authblk} 

\usepackage[T1]{fontenc}
\usepackage[utf8]{inputenc}
\usepackage[english]{babel}
\usepackage{enumerate,amsmath,amsthm}
\usepackage{amssymb}
\usepackage{booktabs}
\usepackage{bookman}
\usepackage{float} 
\usepackage{longtable} 
\usepackage{enumitem,etoolbox,hyperref}
\usepackage{needspace}
\hypersetup{
    colorlinks=true,
    linkcolor=blue,
    citecolor=magenta,
    urlcolor=teal,
    pdfborder={0 0 0},
}
\usepackage{eulervm} 
\usepackage{tabu}
\usepackage{multirow}
\usepackage{graphicx} 
\usepackage{wrapfig} 
\usepackage[font=small,labelfont=bf]{caption} 
\usepackage{mathtools}
\graphicspath{ {pictures/} } 
\usepackage{placeins} 
\allowdisplaybreaks 
\usepackage{tikz} 
\usetikzlibrary{shapes,shapes.geometric,patterns,decorations,positioning,calc}

\tikzset{
    vertex/.style={minimum size=1.5em},
    edge/.style={->,> = latex'},
  pair/.style={{Circle[length=4pt]}-{Circle[length=4pt]},
    shorten >=-3.5pt,
    shorten <=-3.5pt}
  }
\usetikzlibrary{arrows}
\usetikzlibrary{arrows.meta}
\usepackage{xcolor} 
\usepackage{colortbl} 
\usepackage{arydshln} 
\definecolor{Gray}{gray}{0.9} 
\definecolor{mycolor}{RGB}{255, 240, 210} 

\definecolor{diagramblue}{gray}{0.22}
\definecolor{diagramorange}{gray}{0.00}
\definecolor{diagramlightblue}{gray}{0.97}
\definecolor{diagramlightorange}{gray}{0.88}
\definecolor{diagramgray}{gray}{0.50}

\usepackage{hhline}

\usepackage[final]{listings}

\newcommand*{\coreT}{core(T)}
\newcommand*{\shellT}{shell(T)}

\providecommand{\Aut}{\operatorname{Aut}}
\providecommand{\fall}[2]{(#1)_{#2}}

\newcommand{\qthree}{q_3}
\newcommand{\qfour}{q_4}

\newcommand{\Dt}{D(t)}
\newcommand{\Et}{E(t)}
\newcommand{\xt}{x(t)}
\newcommand{\GaussianCore}{\mathcal N}
\newcommand{\ZeroMark}{\mathcal O}

\newcolumntype{Y}{>{\raggedright\arraybackslash}X}
\newcolumntype{C}[1]{>{\centering\arraybackslash}p{#1}}

\newlength{\RoundedBoxWidth}
\newsavebox{\GrayRoundedBox}
\newenvironment{GrayBox}[1][\dimexpr\textwidth-4.5ex]%
   {\setlength{\RoundedBoxWidth}{\dimexpr#1}
    \begin{lrbox}{\GrayRoundedBox}
       \begin{minipage}{\RoundedBoxWidth}}%
   {   \end{minipage}
    \end{lrbox}
    \begin{center}
    \begin{tikzpicture}%
       \draw node[draw=black,fill=black!10,rounded corners,%
             inner sep=2ex,text width=\RoundedBoxWidth]%
             {\usebox{\GrayRoundedBox}};
    \end{tikzpicture}
    \end{center}}

\newenvironment{implementationbox}
  {\begin{GrayBox}\textbf{Implementation invariant.}\par\smallskip}
  {\end{GrayBox}}

\newcommand{\ru}[1]{\rule{0pt}{#1 em}}

\input{mycommands.sty}

\usepackage{comment}

\newcommand{\bunderline}[1]{\mkern2mu\underline{\mkern-2mu #1\mkern-4mu}\mkern4mu }

\newlength\titlebox 

\theoremstyle{plain}
\newtheorem{theorem}{Theorem}
\newtheorem{lemma}[theorem]{Lemma}
\newtheorem{corollary}[theorem]{Corollary}
\newtheorem{proposition}[theorem]{Proposition}

\newtheorem*{no-lemma}{Lemma}

\theoremstyle{definition}
\newtheorem{definition}[theorem]{Definition}

\theoremstyle{remark}
\newtheorem{remark}[theorem]{Remark}
\newtheorem{example}[theorem]{Example}
\usepackage{cleveref}

\usepackage{tocloft}
\input{amsbibstyle.sty}
\title{The Sixth Moment of Random Determinants for Arbitrarily Distributed Random Entries}
\author[1]{Dominik Beck}
\affil[1]{\small Charles University, Faculty of Mathematics and Physics, Prague 116 36}
\affil[ ]{\textit {\href{mailto:beckd@karlin.mff.cuni.cz}{beckd@karlin.mff.cuni.cz}}}
 \author[2]{Zelin Lv}
 \author[2]{Aaron Potechin}
 \affil[2]{\small The University of Chicago, Chicago, IL 60637}
 \affil[ ]{\textit{\href{mailto:zlv@uchicago.edu}{zlv@uchicago.edu}},
 \textit{\href{mailto:potechin@uchicago.edu}{potechin@uchicago.edu}}}
\date{\today}

\begin{document}
\maketitle
\begin{abstract}
Via the method of marked permutation tables presented in this paper, we generalize the formula for the sixth moment of a random determinant to account for entries with arbitrary distribution. That is, let $f_6(n) = \mathbb{E}(\det A)^6$, where $A$ is an $n$ by $n$ random matrix with independent and identically distributed entries. We show that the exponential generating function $F_6(t) = \sum_{n=0}^\infty f_6(n)t^n/(n!)^2$ is D-finite and we present it in a closed form. Our method relies on carefully decomposing marked permutation tables into a shell, a core, and a floating component, each of which has a separate contribution to $F_6(t)$. After this decomposition, it is sufficient to enumerate over a finite number of possible shells, which we did using a highly intricate computer program. 
We verified our result up to $n = 7$ in the general case and up to $n = 9$ for random matrices whose entries only take two values by using a different method for computing $f_6(n)$ for these cases.
\end{abstract}

\tableofcontents

\section{Introduction}
Let $X_{ij}$ be i.i.d. random variables with distribution $\Omega$ and let $m_q = \expe{X_{11}^q}$ be the moments of $\Omega$. For a random matrix $A = (X_{ij})_{n \times n}$ with i.i.d. entries distributed according to $\Omega$, our objective is to study the $k$-th moment of its determinant 
\[
f_k(n) = \expe{(\det A)^k},
\]
or equivalently the generating functions $$F_k(t) = \sum_{n=0}^\infty \frac{f_k(n)}{n!^2} t^n,$$ from which $f_k(n)$ follow simply by Taylor expansion. Although the problem of deducing $f_k(n)$ is simple to pose, only a handful of results were known prior to our own investigation. For odd $k$, the anti-symmetry of $\det A$ enforces $f_k(n)$ to vanish, so only even $k$ have nontrivial contribution. It is trivial to show (c.f. \cite{stanley_fomin_1999}, \cite{fortet1951random}),
\begin{align}
    F_2(t) & = (1+m_1^2t)e^{(m_2 - m_1^2)t},\\
    f_2(n) & = n! (m_2 + m_1^2(n-1)) (m_2 - m_1^2)^{n-1}.
\end{align}
In general, the determinant moments $f_k(n)$ are polynomials in $m_1,m_2,\ldots,m_k$. As we will see, it will be convenient to express these polynomials in the variables $m_1,\mu_2,\ldots,\mu_k$ where $\mu_q = \expe{(X_{11}-m_1)^q}$ is the $q-$th \emph{central moment} of $\Omega$. Concretely, we have
\begin{equation}
\begin{split}
\mu_2 &= m_2-m_1^2,\\
\mu_3 & = m_3-3 m_1 m_2+2 m_1^3,\\
\mu_4 & = m_4-4 m_1 m_3+6 m_1^2 m_2-3 m_1^4 \\
\mu_5 & = m_5-5 m_1 m_4+10 m_1^2 m_3-10 m_1^3 m_2+4 m_1^5,\\ \mu_6 & = m_6-6 m_1 m_5+15 m_1^2 m_4-20 m_1^3 m_3+15 m_1^4 m_2-5 m_1^6. 
\end{split}    
\end{equation}
The problem of determining determinant moments starts to become highly nontrivial for higher moments. Nyquist, Rice and Riordan \cite{nyquist1954distribution} derived the special case of $k = 4$ with $m_1 = 0$. The general $k=4$ case imposing no condition on $m_1$ was derived only recently in \cite{beck2022fourth} by one of the authors.
\begin{equation*}
F_4(t) = \frac{e^{t(\mu_4-3 \mu _2^2)}}{(1-\mu_2^2 t)^3} \left((1+m_1\mu_3 t)^4+6m_1^2\mu_2 t\frac{(1+m_1 \mu_3t)^2}{1-\mu _2^2 t}+m_1^4 t\frac{1+7 \mu_2^2 t+4 \mu _2^4 t^2}{(1-\mu _2^2 t)^2}\right),
\end{equation*}
from which one can deduce, denoting $d_0(c) = 2+c$, $d_1(c) = c(2+c)$ and $d_2(c) = c^3$, that
\begin{equation*}
f_4(n) = n!^2 \sum _{w=0}^2 \sum _{s=0}^{4-2w} \sum _{c=0}^{n-s} \binom{4-2 w}{s}\frac{(1+c) m_1^{s+2 w} \mu_2^{2c-w} \mu _3^s \left(\mu _4-3\mu_2^2\right){}^{n-c-s}}{(n-c-s)!(2-w)! w!} d_w(c).
\end{equation*}
Furthermore, 
we solved the special case of $k = 6$ with $m_1 = 0$ in a previous paper \cite{BLP23}. For this special case, we showed that 
\begin{equation}\label{Eq:F6m0}
    F_6(t)|_{m_1 = 0} = (1+m_3^2t)^{10}\,\frac{e^{t \left(m_6-15m_4m_2-10m_3^2 + 30m_2^3\right)}}{\left(1+3m_2^3t-m_4m_2t\right)^{15}} N\left(\frac{m_2^3t}{\left(1+3m_2^3t-m_4m_2t\right)^{3}}\right),
\end{equation}
where
\begin{equation}\label{Eq:N}
N(t) = \frac{1}{48}\sum_{n=0}^\infty (n+1)(n+2)(n+4)! \,t^n
\end{equation}
is the exponential generating function for the sixth moment of the determinant of a random matrix with i.i.d. Gaussian entries \cite{nyquist1954distribution}\cite{dembo1989random} (denoted as $N_6(t)$ in \cite{BeckPhD}\cite{BLP25}). Defining $q_6 = m_6 - 10m_3^2- 15m_4m_2+ 30m_2^3$ and $q_4 = m_4m_2- 3m_2^3$, it follows that
\begin{equation*}
f_6(n)|_{m_1=0} = n!^2 \sum _{j=0}^n \sum _{i=0}^j \sum_{k=0}^{n-j} \frac{(1+i) (2+i) (4+i)! }{48 (n-j-k)!}\binom{10}{k}\binom{14+j+2i}{j-i} q_6^{n-j-k} q_4^{j-i}m_3^{2k} m_2^{3i}.
\end{equation*}
Both of our previous papers \cite{beck2022fourth}\cite{BLP23} use recursions as their main tool as used in prior works. 
However, this method 
ignores some of the internal structure and doesn't give good intuition for 
why the generating functions look like they do. A more effective way of 
thinking about the problem starts by reframing it 
in terms of analyzing certain combinatorial objects associated with the random determinants and by employing the algebraic rules for compositions of generating functions of the components which they are formed from. The combinatorial objects in question are the permutation tables associated to tuples of determinants and the components of these tables are various structures formed by the columns of the table. As simple as this bijection may seem, this simple change of perspective has significantly improved our understanding of the moments of the determinant of a random matrix.  
In the course of the past few years, there has been a clear crystalisation of this permutation tables technique, see our summary paper \cite{BLP25} for a quick introduction-overview or see Chapter 2 of PhD thesis \cite{BeckPhD} of one of the authors where the technique was introduced and used to derive $F_6(t)$ for arbitrary $m_1$ but assuming $\mu_3 = 0$.

\subsection{Main result}
The main result of our paper is computing $F_6(t)$ for any underlying distribution $\Omega$ of entries. Note that we only need to consider $\mu_2 = 1$, that is $F_6(t)|_{\mu_2 = 1}$. Let
\begin{equation}
F_6(t) = \Psi(t,m_1,\mu_2,\mu_3,\mu_4,\mu_5,\mu_6)
\end{equation}
for some function $\Psi$, then also (see \cite{BeckPhD}),
\begin{equation}
F_6(t) = \Psi\left(\mu_2^3 t,\frac{m_1}{\sqrt{\mu _2}},1,\frac{\mu _3}{\mu _2^{3/2}},\frac{\mu _4}{\mu _2^2},\frac{\mu_5}{\mu _2^{5/2}},\frac{\mu _6}{\mu _2^3}\right),
\end{equation}
but $\Psi(t,m_1,1,\mu_3,\mu_4,\mu_5,\mu_6) = F_6(t)|_{\mu_2 = 1}$. For the moments themselves, let
\begin{equation}
f_6(n) = \psi(n,m_1,\mu_2,\mu_3,\mu_4,\mu_5,\mu_6)
\end{equation}
for some function $\psi$, then also
\begin{equation}
f_6(n) = \mu_2^{3n} \psi\left(n,\frac{m_1}{\sqrt{\mu _2}},1,\frac{\mu _3}{\mu _2^{3/2}},\frac{\mu _4}{\mu _2^2},\frac{\mu_5}{\mu _2^{5/2}},\frac{\mu _6}{\mu _2^3}\right),
\end{equation}
but $\psi(n,m_1,1,\mu_3,\mu_4,\mu_5,\mu_6) = f_6(n)|_{\mu_2 = 1}$. Now we are ready to state our main result:
\begin{theorem}
For any distribution $\Omega$ with $\mu_2 = 1$,
\begin{equation}\label{Eq:F6Zsum}
F_6(t)|_{\mu_2 = 1} = \sum _{r=0}^6 m_1^r \left(1+\left(\mu _5-10\mu_3\right) m_1 t\right){}^{6-r} Z_r(t).
\end{equation}
Denote $x = 1/(1-\left(\mu _4-3\right) t)$ and $y = \mu_3^2/(1+\mu _3^2 t)$, then $Z_r(t) =$
\begin{equation*}
 \mu _3^{(r \bmod 2)} \left(1+t \mu _3^2\right){}^{10-(r \bmod 2)} e^{t \left(\mu_6-15\mu_4-10\mu_3^2 + 30\right)} x^{21-2 r} \sum_{s=0}^{2\lfloor r/2\rfloor} y^s \left(p_{rs}(t) N\left(t x^3\right)+q_{rs}(t) N'\left(t x^3\right)\right),
\end{equation*}
where $\quad p_{00}(t) = 1/x^6,\quad q_{00}(t) = 0,\quad p_{10}(t) = 60 t/x^3,\quad q_{10}(t) = 12 t^2$,

\begin{align*}
& p_{20}(t) = 15 (x-1)/x^2, && p_{30}(t) = 20 t \left(21 x^2-42 x+4\right),\\
& p_{21}(t) = 630 t^2, && p_{31}(t) = 480 t^2 \left(7 t x^3+3 x-5\right),\\
& p_{22}(t) = -630 t^3, && p_{32}(t) = 480 t^3 \left(2-7 t x^3\right),\\
& q_{20}(t) = 3 t x^2, && q_{30}(t) = 12 t x^2 \left(7 t (x-2) x^2+2\right),\\
& q_{21}(t) = 6 t^2 \left(21 t x^3-2 x+10\right), && q_{31}(t) = 16 t^2 \left(42 t^2 x^6\!+\!18 t x^4\!-\!11 t x^3\!-\!6 x\!+\!10\right),\\
& q_{22}(t) = -6 t^3 \left(21 t x^3+8\right), && q_{32}(t) = -32 t^3 \left(21 t^2 x^6+5 t x^3+2\right),
& &&\\[2ex]
& p_{40}(t) = 15 x^2 ((4 t-6) x+3), && q_{40}(t) = 3 x^4 \left(4 t^2 x^2-2 t (3 x+2) x+1\right),\\
& p_{41}(t) = 60 t x^2 \begin{Bmatrix}84 t x^3-147 t x^2\\+(44 t+6) x-18\end{Bmatrix}, && q_{41}(t) = 12 t x^2 \begin{Bmatrix}84 t^2 x^6-147 t^2 x^5\\+t (44 t+6) x^4+18 t x^3\\-50 t x^2-2 x+6\end{Bmatrix},\\
& p_{42}(t) = 180 t^2 \begin{Bmatrix}\!42 t^2 x^6\!-\!28 t x^5\!+\!105 t x^4\!\\-108 t x^3+24 x-20\end{Bmatrix}, && q_{42}(t) = 12 t^2 \begin{Bmatrix}126 t^3 x^9-84 t^2 x^8+315 t^2 x^7\\-282 t^2 x^6-44 t x^5+192 t x^4\\-84 t x^3-24 x+20\end{Bmatrix},\\
& p_{43}(t) = -720 t^3 \begin{Bmatrix}21 t^2 x^6+14 t x^4\\-38 t x^3+4 x-4\end{Bmatrix}, && q_{43}(t) = -48 t^3 \begin{Bmatrix}63 t^3 x^9+42 t^2 x^7-93 t^2 x^6\\+22 t x^4+2 t x^3-4 x+4\end{Bmatrix},\\
& p_{44}(t) = 2520 t^5 x^3 \left(3 t x^3-4\right), && q_{44}(t) = 168 t^5 x^3 \left(9 t^2 x^6-9 t x^3+4\right),
\end{align*}

\begin{align*}
p_{50}(t) & = -60 x^4 \left(t \left(21 x^2-43 x+4\right)+3\right),\\
p_{51}(t) & = 720 t x^2 \left(42 t^2 x^6-70 t^2 x^5+7 t (4 t+1) x^4-31 t x^3+18 t x^2+2 x-2\right),\\
p_{52}(t) & = -720 t^2 \left(42 t^2 x^8-126 t^2 x^7+98 t^2 x^6-21 t x^5+4 t x^4+8 t x^3+4 x^2-8 x+4\right),\\
p_{53}(t) & = -20160 t^4 x^3 \left(2 t x^4-4 t x^3+x-1\right),\\
p_{54}(t) & = -30240 t^6 x^6,\\
q_{50}(t) & = -12 x^4 \left(t^2 \left(21 x^3-43 x^2-8 x-4\right) x^2+t \left(x^2+10 x+4\right) x-1\right),\\
q_{51}(t) & = 48 t x^2 \begin{Bmatrix}126 t^3 x^9-210 t^3 x^8+21 t^2 (4 t+1) x^7-51 t^2 x^6\\-35 t^2 x^5+t (20 t+11) x^4+13 t x^3-18 t x^2-2 x+2\end{Bmatrix},\\
q_{52}(t) & = -48 t^2 \begin{Bmatrix} 126 t^3 x^{11}-378 t^3 x^{10}+294 t^3 x^9-21 t^2 x^8-162 t^2 x^7\\+94 t^2 x^6+57 t x^5-76 t x^4+28 t x^3-4 x^2+8 x-4\end{Bmatrix},\\
q_{53}(t) & = -1344 t^4 x^3 \left(6 t^2 x^7-12 t^2 x^6+7 t x^4-5 t x^3-x+1\right),\\
q_{54}(t) & = -2016 t^6 x^6 \left(3 t x^3-1\right),\\
&\\
p_{60}(t) & = x^6 \left(4 t^2 \left(105 x^2-255 x+64\right)+15 t (15 x-1)-15\right)/t,\\
p_{61}(t) & = -10 x^4 \left(504 t^2 x^5-777 t^2 x^4-3 t (64 t-81) x^3+4 t (46 t-27) x^2-18 (4 t+1) x+18\right),\\
p_{62}(t) & = 30 t x^2 \begin{Bmatrix} 2520 t^3 x^9-4284 t^3 x^8+56 t^2 (32 t+15) x^7-2919 t^2 x^6\\+3 t (648 t+35) x^5-8 t (28 t-27) x^4-504 t x^3+24 (9 t-1) x^2+48 x-24\end{Bmatrix},\\
p_{63}(t) & = -240 t^2 \begin{Bmatrix} 630 t^3 x^{11}-1449 t^3 x^{10}+7 t^2 (109 t+12) x^9-525 t^2 x^8+408 t^2 x^7\\+(59-48 t) t x^6-93 t x^5+36 t x^4-2 (t+2) x^3+12 x^2-12 x+4 \end{Bmatrix},\\
p_{64}(t) & = 840 t^4 x^3 \left(90 t^2 x^8-369 t^2 x^7+274 t^2 x^6-51 t x^5+12 t x^4+20 t x^3+12 x^2-24 x+12\right),\\
p_{65}(t) & = 5040 t^6 x^6 \left(18 t x^4-25 t x^3+6 x-6\right),\\
p_{66}(t) & = 25200 t^8 x^9,\\
q_{60}(t) & = x^6 \left(12 t^2 (7 x-17) x^4+t \left(45 x^2+44 x+36\right) x^2+(1/t)-\left(8 x^2+3 x+12\right) x\right),\\
q_{61}(t) & = 2 x^4 \begin{Bmatrix} -504 t^3 x^8+777 t^3 x^7+3 t^2 (64 t-81) x^6+2 t^2 (18 t-155) x^5\\+6 t (37 t+9) x^4+3 t (8 t+9) x^3+12 t (2 t-3) x^2-6 (4 t+1) x+6\end{Bmatrix},\\
q_{62}(t) & = 2 t x^2 \begin{Bmatrix}7560 t^4 x^{12}-12852 t^4 x^{11}+168 t^3 (32 t+15) x^{10}-6237 t^3 x^9\\+9 t^2 (172 t+35) x^8-1968 t^2 x^6+4 t^2 (71 t+564) x^7-3 t (72 t+107) x^5\\+8 t (34 t+27) x^4+288 t x^3-24 (9 t-1) x^2-48 x+24\end{Bmatrix},\\
q_{63}(t) & = -16 t^2 \begin{Bmatrix}1890 t^4 x^{14}-4347 t^4 x^{13}+21 t^3 (109 t+12) x^{12}-945 t^3 x^{11}-225 t^3 x^{10}\\+t^2 (389 t+351) x^9-24 t^2 x^8-372 t^2 x^7+t (126 t-95) x^6\\+201 t x^5-144 t x^4+(38 t+4) x^3-12 x^2+12 x-4\end{Bmatrix},\\
q_{64}(t) & = 56 t^4 x^3 \begin{Bmatrix}270 t^3 x^{11}-1107 t^3 x^{10}+822 t^3 x^9-63 t^2 x^8-333 t^2 x^7\\+242 t^2 x^6+159 t x^5-228 t x^4+88 t x^3-12 x^2+24 x-12\end{Bmatrix},\\
q_{65}(t) & = 336 t^6 x^6 \left(54 t^2 x^7-75 t^2 x^6+36 t x^4-29 t x^3-6 x+6\right),\\
q_{66}(t) & = 1680 t^8 x^9 \left(3 t x^3-1\right).
\end{align*}
\end{theorem}
The form of our result above is not particularly illuminating and even the functions $p_{rs}(t)$ and $q_{rs}(t)$ do not carry any specific meaning as they were introduced for the sole purpose to make our result (machine) readable. However, we will show that $Z_0(t),\ldots,Z_6(t)$ do carry specific meaning and thus can be obtained as concrete linear combinations of derivatives of $N(t)$ whose coefficients are compositions of rational and exponential functions associated with permutation tables. The fact that we only need the first two derivatives of $N(t)$ to express the full $F_6(t)$ follows from $N(t)$ satisfying the following ordinary differential equation whose correctness can be easily verified:
\begin{equation}
    t^2 N''(t)-(1-9t) N'(t)+15 N(t)=0.
\end{equation}
In the rest of our paper, we will assume $\mu_2 = 1$ (if not stated otherwise). Also, it will be convenient to 
write $O(t) = F_6(t)|_{m_1 = 0}$ as a function of $\mu_q$'s, that is, by Equation \eqref{Eq:F6m0},
\begin{equation}\label{Eq:O}
O(t) = (1+\mu_3^2t)^{10}\,\frac{e^{t (\mu_6-15\mu_4-10\mu_3^2 + 30)}}{\left(1+3t-\mu_4t\right)^{15}} N\left(\frac{t}{\left(1+3t-\mu_4t\right)^{3}}\right)
\end{equation}
or simply $O(t) = (1+\mu_3^2t)^{10}\,\frac{e^{t (\mu_6-15\mu_4-10\mu_3^2 + 30)}}{\left(1+3t-\mu_4t\right)^{15}} N(t')$ with $t'=t/(1-(\mu_4-3)t)^3$.

\subsection{Examples and validations}
Our formula enables us to evaluate the moments for certain special distribution unamenable with the previously solved special cases $m_1 = 0$ (in \cite{BLP23}) and $\mu_3 = 0$ (in \cite{BeckPhD}). 

\subsubsection{General case}
First, let us note that we tested our formula by independently calculating $f_6(n)$ up to $n = 7$ for general $m_1,m_2,\ldots,m_6$). Let us briefly discuss how those values were obtained.

The polynomial $f_6(n)$ corresponding to the sixth moment was computed by exploiting the factorization $(\det A)^6=(\det A)^4(\det A)^2$,
which avoids the direct enumeration of all ordered 6-tuples of permutations.

In the first stage, all equivalence classes of ordered 4-tuples of permutations were generated up to simultaneous conjugation by the symmetric group. The search was performed incrementally as a permutation branching tree. At every level, branches were canonically represented using conjugation, equivalent branches were merged, and stabilizer sizes were propagated to determine the correct multiplicities. This reduced the search space by several orders of magnitude. The resulting fourth-layer representatives, together with their signed multiplicities, were exported for subsequent processing.

Independently, all terms occurring in $(\det A)^2$ were generated once and stored as pairs of permutation indices with their corresponding signs. Both files containing (symmetrized) $(\det A)^4$ and $(\det A)^2$ were roughly the same size in order of 100 GB.

The final polynomial was obtained by combining every representative from the fourth layer with every determinant-squared pair. For each resulting 6-tuple of permutations, the corresponding monomial in the multiplicity variables $(m_1,\ldots,m_6)$ was computed from the column occupancy counts, and its coefficient was accumulated.

Since the fourth layer for ($n=7$) contains millions of representatives, the computation was divided into independent batches. A preliminary benchmarking stage determined the largest batch size that could reliably finish within the 12-hour wall-time limit of the computing cluster. Each batch was then evaluated independently using all CPU cores allocated by Slurm, producing a partial polynomial. A Slurm job array was used to execute all batches concurrently. Finally, all partial polynomials were merged by summing coefficients of identical monomials, yielding the complete polynomial for the sixth moment.

The implementation used Numba to JIT-compile the monomial construction routine, multiprocessing for shared-memory parallelism within each batch, and Slurm job arrays for distributed execution on the Sn\v{e}hurka cluster (former cluster of Chrales University, Prague). This combination enabled the computation of the ($n=7$) case. The computation was performed as 171 independent Slurm array jobs using 72 CPU cores each. The total computational effort corresponds to approximately $1.09\times 10^9$ CPU-hours (12.4 CPU-years) on the cluster.

\subsubsection{Exponential distribution}
Let $\Omega = \mathsf{Exp}(1)$, that is $m_k = k!$ or $(m_1,\mu_2,\mu_3,\mu_4,\mu_5,\mu_6) = (1,1,2,9,44,265)$, then we have the following matrix determinant moments (Table \ref{tab:distrExp}):
\begin{table}[H]
    \centering
    \begin{tabular}{c|l}
\hline
$n$ & $f_6(n)$\\
\hline
 0 & 1 \\
 1 & 720 \\
 2 & 907200 \\
 3 & 1559900160 \\
 4 & 3340718899200 \\
 5 & 8515130572800000 \\
 6 & 25161471058916966400 \\
 7 & 84778820397427064832000 \\
 8 & 322187011630166233055232000 \\
 9 & 1370965636084544778912399360000 \\
 10 & 6500579834950673083080376320000000 \\
 11 & 34243885869525179666648684707184640000 \\
 12 & 200085271301335275167268690894363033600000 \\
 13 & 1295983978018787557825779569628868824268800000 \\
 14 & 9307332162977084442481734803080957717905408000000 \\
 15 & 74166629026207481358952479221797600584794112000000000 \\
 16 & 656429305664736449200916262839339144268500460109824000000 \\
 17 & 6459825500870621719022312200945319525232978210639052800000000 \\
 \hline
    \end{tabular}
    \caption{Sixth determinant moments for the exponential distribution of entries}
    \label{tab:distrExp}
\end{table}

\subsubsection{Shifted Scaled Bernoulli Matrices}
Let $\Omega = \{a,b\}$, so that $X_{ij} = b$ with probability $p$ and $X_{ij} = a$ with probability $1-p$. This distribution family naturally generalizes the $\{0,1\}$ matrices. Alternatively, $\Omega = a + (b-a) \mathsf{Bern}(p)$, where $\mathsf{Bern}(p)$ is the Bernoulli distribution. We can also equivalently characterize the distribution of $\Omega$ by its moments $m_k = a^k (1-p) + b^k p$. Without loss of generality, since we can put $\mu_2 = 1$, we can assume
\begin{equation}
p = \frac{1}{2}-\frac{v}{2 \sqrt{v^2+4}},\quad a = u + \frac{1}{2} \left(v-\sqrt{v^2+4}\right),\quad b = u +\frac{1}{2} \left(\sqrt{v^2+4}+v\right),
\end{equation}
so we have $(m_1,\mu_2,\mu_3,\mu_4,\mu_5,\mu_6) = (u,1,v,1+v^2,2v+v^3,1+3v^2+v^4)$. Our formula predicts
\begin{align*}
f_6(1) &= \left(v^4+3 v^2+1\right)+6 u v \left(v^2+2\right)+15 u^2 \left(v^2+1\right)+20 u^3 v+15 u^4+u^6,\\
f_6(2) &= 2 \left(v^2+4\right) \left(v^6+2 v^4+8 v^2+4\right)+24 u v \left(v^2+4\right) \left(v^4+v^2+3\right)\\
& +60 u^2 \left(v^2+4\right) \left(2v^4+v^2+2\right)+80 u^3 \left(v^2+4\right) \left(4 v^3+v\right)+120 u^4 \left(v^2+4\right) \left(4 v^2+1\right)\\
&+360 u^5 v \left(v^2+4\right)+4 u^6 \left(v^2+4\right) \left(v^2+34\right),\\
f_6(3) & = 6 \left(v^2+4\right)^2 \left(v^8+v^6+21 v^4+46 v^2+16\right)+108 u v \left(v^2+4\right)^2 \left(v^6+12 v^2+12\right)\\
&+270 u^2 \left(v^2+4\right)^2 \left(3 v^6-v^4+20 v^2+8\right)+360 u^3 v \left(v^2+4\right)^2 \left(9 v^4-2 v^2+28\right)\\
&+270 u^4 \left(v^2+4\right)^2 \left(27 v^4+6 v^2+32\right)+2160 u^5 v \left(v^2+4\right)^2 \left(4 v^2+3\right)\\
&+18 u^6 \left(v^2+4\right)^2 \left(v^4+248 v^2+376\right),\\
f_6(4) & = 24 \left(v^2+4\right) \left(v^{14}+8 v^{12}+56 v^{10}+480 v^8+1920 v^6+5984 v^4+5792 v^2+2944\right)\\
&+576 u v \left(v^2+4\right) \left(v^{12}+7 v^{10}+35 v^8+270 v^6+880 v^4+2024 v^2+1088\right)\\
&+1440 u^2 \left(v^2+4\right) \left(4 v^{12}+27 v^{10}+96 v^8+632 v^6+1800 v^4+2832 v^2+1024\right)\\
&+1920 u^3 v \left(v^2+4\right) \left(16 v^{10}+113 v^8+316 v^6+1492 v^4+3812 v^2+3376\right)\\
&+5760 u^4 \left(v^2+4\right) \left(16 v^{10}+125 v^8+340 v^6+994 v^4+2198 v^2+1192\right)\\
&+8640 u^5 v \left(v^2+4\right) \left(17 v^8+150 v^6+456 v^4+896 v^2+1248\right)\\
&+96 u^6 \left(v^2+4\right) \left(v^{10}+1040 v^8+10480 v^6+39040 v^4+78560 v^2+86464\right)
\end{align*}
and so on. Let us briefly explain what is needed to recover $f_6(n)$ for the shifted scaled Bernoulli matrices independent from our result. Let $B_n$ be the set of all $\{0,1\}$ matrices of size $n \times n$ and let
\begin{equation}
    c_n(\alpha,\beta,\gamma) = \left\{B = (B_{ij})_{n\times n} \in B_n \,\bigg{|}\, \sum_{ij} B_{ij} = \alpha,\, \det B = \beta,\, \det(B+J) = \gamma\right\},
\end{equation}
where $J$ is the matrix of ones. Note that $\alpha$ counts the number of ones in $B$. We can write any matrix $A$ whose entries follow the shifted scaled Bernoulli distribution as $A = a J + (b-a) B$ for some (random) $B \in B_n$. By the Matrix Determinant Lemma, we know that $\det A$ must be a linear function of $\det B$ and this linearity forces
\begin{equation}
\det A = (b-a)^{n-1}((b-2a)\det B + a\det (B+J)).
\end{equation}
Assuming $B$ has $\mathsf{Bern}(p)$ entries, the probability of selecting a given matrix $B$ equals $p^\alpha(1-p)^{n^2 - \alpha}$ and thus for any $k$,
\begin{equation}\label{Eq:SSB}
f_k(n) = \Exx (\det A)^k = (b-a)^{k(n-1)} \sum_{\alpha,\beta,\gamma} c_n(\alpha,\beta,\gamma) p^\alpha(1-p)^{n^2 - \alpha} ((b-2a) \beta + a \gamma)^k,
\end{equation}
where the sum is taken over all possible values of $\alpha,\beta,\gamma$ for $B_n$ matrices. By only knowing $c_n(\alpha,\beta,\gamma)$, we recover $f_k(n)$ for the distribution class of shifted scaled Bernoulli matrices. We checked our formula for $f_6(n)$ using the above formula by obtaining $c_n(\alpha,\beta,\gamma)$ up to $n = 9$ by running Nauty and Traces \cite{McKay2014} on the new HPC cluster of Charles University, Prague. The case $n = 9$ required approximately $1.93\times 10^{15}$ arithmetic operations, $16{,}384$ parallel CPU-hours of wall time, and produced $\approx 96$\,GB of intermediate data.  The case $n = 10$ is estimated to require at least $70$ days of computation on $10{,}000$ cores together with $\approx 236$\,TB of storage, and was not attempted. The core idea is to enumerate $n\times n$ matrices \emph{implicitly}
via their canonical representatives under the action of $S_n \times S_n$
(simultaneous row and column permutations). Their counts (sequence \href{https://oeis.org/A002724}{A002724} in the OEIS) grow rapidly: $317$ for $n=5$, $251{,}610$ for $n=7$, $33{,}642{,}660$ for $n=8$, and $14{,}685{,}630{,}688$ for $n=9$. We can only store representatives of $n = 8$. Then, for each representative, we generate all possible $9\times 9$ matrices on the fly by \emph{bordering}: extending each $8\times 8$
canonical form by one new row and one new column -- the method introduced by \v{Z}ivkovi\'{c} \cite{Zivkovic2006}. This gives $2^{17} = 131{,}072$ candidate extensions per base form. Determinants of all extensions are then accumulated using the Bareiss algorithm
implemented in Numba-JIT-compiled Python, achieving a sustained throughput of approximately $3{,}775\,\mu$s per base form. This turned out to be faster than directly using \texttt{genbg} from Nauty on $9 \times 9$ matrices. 

\subsubsection{Bernoulli and Uniform Binary matrices}
Finally, let us mention a notable special case of Bernoulli matrices (which is not equivalent to the general case of shifted scaled Bernoulli matrices). Let $C_n(\alpha,\beta) = \sum_\gamma c_n(\alpha,\beta,\gamma)$ be the number of $\{0,1\}$ matrices of size $n \times n$ with $\alpha$ of ones and with $\det B = \beta$. Then for $\Omega = \mathsf{Bern}(p)$, we have by Equation \eqref{Eq:SSB} with $a=0$ and $b = 1$,
\begin{equation}
f_k(n) = \sum_{\alpha,\beta} C_n(\alpha,\beta) p^\alpha(1-p)^{n^2 - \alpha} \beta^k.
\end{equation}
The case above is still uncoverable from the known $F_6(t)|_{\mu_3 = 0}$ (see \cite{BeckPhD}) since we have $m_1 = p$ and $\mu_3 = p (1-p) (1-2 p)$. However, when, $p = 1/2$ (that is, matrices $B$ are selected uniformly), we do indeed have $\mu_3 = 0$. Let $T_n(\alpha) = \sum_\beta C(\alpha,\beta)$ be the number of $\{0,1\}$ matrices of size $n \times n$ with $\det B = \beta$, then $f_k(n) = (1/2^{n^2}) \sum_\alpha T_n(\alpha) \beta^k$, where the sum runs over all possible values of $\det B$. Note that the values $T_n(\alpha)$ are listed by OEIS sequence \href{https://oeis.org/A089478}{A089478} (only up to $n=9$). Note that, as pointed by Tao \cite{TerryTao}, there is an equivalence between $\{0,1\}$ and $\{\pm 1\}$ matrices with $p=1/2$, so essentially the above generating function can be deduced even from the special case of $F_6(t)$ with $m_1 = 0$ (see \cite{BLP23}).

\subsection{Asymptotics}
The knowledge of the full generating function enables us to deduce certain interesting asymptotics by the method of analysing factorially divergent series due to Borinsky \cite{Borinsky2018}. For any positive integer $r$, we have the asymptotic expansion
\begin{equation}
    \frac{f_6(n)}{n!^2} = \left(\sum _{j=0}^{r-1} c_j\,(n+9-j)!\right) + O((n+9-r)!),
\end{equation}
where the coefficients $c_j$ are given by applying the $\mathcal{A}^1_{10}$ operator on $F_6(t)$. Formally,
\begin{equation}
\mathcal{A}^1_{10}[F_6](t) = \sum_{j=0}^\infty c_j t^j.
\end{equation}
More concretely, let $g(t) = t/\left(1+(3-\mu_4)t\right)^3$, then $\mathcal{A}^1_{10}[F_6](t)$ equals $F_6(t)$ in which we replace $N(g(t))$ with
\begin{equation}
    \frac{t^3}{48}\left(\frac{t}{g(t)}\right)^7 \exp\left(\frac{1}{t}-\frac{1}{g(t)}\right) (1-8 g(t)+12 g(t)^2)
\end{equation}
and $N'(g(t))$ by
\begin{equation}
    \frac{t^2}{48}\left(\frac{t}{g(t)}\right)^8 \exp\left(\frac{1}{t}-\frac{1}{g(t)}\right) (1-15g(t)+60 g(t)^2-60 g(t)^3).
\end{equation}
For any $\Omega$ with $\mu_2 = 1$, we get by expanding $\mathcal{A}^1_{10}[F_6](t)$,
\begin{equation}
f_6(n) = \frac{e^{3\mu_4 - 9}}{48} (n!)^3 (m_1^6 n^9+m_1^4 \left(3+12 m_1 \mu _3+m_1^2 \left(46-9 \mu _4-3 \mu _4^2+\mu _6\right)\right) n^8 + O(n^7))
\end{equation}

\subsubsection{Exponential distribution}
A more precise asymptotic relation can be obtained by expanding $\mathcal{A}^1_{10}[F_6](t)$ into more terms. For example, when $\Omega = \mathsf{Exp}(1)$, we would get
\begin{equation}
\begin{split}
&f_6(n) = \frac{e^{18}}{48} (n!)^3 \left(n^9+14 n^8+75 n^7-2352 n^6-48429 n^5-\frac{2585282}{5} n^4\right.\\
&\left.-\frac{15625943}{5}n^3+\frac{46592436}{7}n^2+\frac{2370767804}{5}n+\frac{318577704432}{35} + O(1/n)\right).
\end{split}
\end{equation}

\subsubsection{Bernoulli matrices}
Let $\Omega = \mathsf{Bern}(p)$, then
\begin{equation}
\begin{split}
f_6(n) & = \frac{(n!)^3}{48} e^{3 \left(-3+p-4 p^2+6 p^3-3 p^4\right)} \big{(}n^9 p^6+n^8 p^4 \left(-27 p^{10}+108 p^9-185 p^8+177 p^7\right.\\
&\left.-77 p^6-15 p^5+51 p^4-44 p^3+58 p^2+3\right) + \tfrac12 n^7 p^2\left(729 p^{20}-5832 p^{19}\right.\\
&\left.+21654 p^{18}-49518 p^{17}+76401 p^{16}-80028 p^{15}+50769 p^{14}-4806 p^{13}\right.\\
&\left.-31241 p^{12}+43748 p^{11}-35543 p^{10}+17960 p^9-1565 p^8-6262 p^7\right.\\
&\left.+7114 p^6-4846 p^5+3157 p^4-180 p^3+312 p^2+6\right) + O(n^6)\big{)}.
\end{split}
\end{equation}

\subsection{Permutation tables decomposition}
Let us set the stage for the proof of the main result. For that, we need to recall some constructions on permutation tables. Let $S_n$ be the set of all permutations on $[n] = \{1,2,3,\ldots,n\}$, then by definition of determinant,
\begin{equation}
\det A = \sum_{\pi \in S_n} \sgn \pi \prod_{i=1}^n X_{i \pi(i)}.
\end{equation}
Taking the $6$-th power and applying the expectation operator, we get
\begin{equation}
f_6(n) = \sum_{T \in T_{6,n}} w(T) \sgn T,
\end{equation}
where $T_{6,n} := S_n\times S_n\times S_n\times S_n\times S_n\times S_n$ is the set of all \emph{permutation tables} with $6$ rows and $n$ columns and where
\begin{equation}
\sgn T = \sgn \pi_1 \sgn \pi_2 \cdots \sgn \pi_6
\end{equation}
is the \emph{sign} and
\begin{equation}
w(T) = \Exx \prod_{j=1}^6 \prod_{i=1}^n X_{i \pi_j(i)}
\end{equation}
is the \emph{weight} of a given permutation table $T = \pi_1 \times \pi_2 \times \pi_3 \times \pi_4 \times \pi_5 \times \pi_6 \in T_{6,n}$. Due to the independence, the weight $w(T)$ can we written as a product of weights $w(\mathbf{c})$ of individual columns $\mathbf{c}$ of $T$. An example of a table $T$ is shown below in Figure \ref{fig:exatable6x11} together with its column weights. Its seventh column corresponds to the term $X_{79}^2 X_{78}^4$ in the expansion of $(\det A)^6$ and the weight of this column is 
$w(\mathbf{c}) = \Exx X_{79}^2 X_{78}^4 = m_2 m_4$.

\begin{figure}[H]
\centering
    \begin{tabular}{|*{11}{>{\columncolor{mycolor}[\tabcolsep][\tabcolsep]\centering}p{2.6em}|}>{\centering\arraybackslash}p{0.6em}}
    \hhline{|-----------|~}
    \ru{0.9}8 & 5 & 1 & 10 & 4 & 7 & 9 & 11 & 3 & 2 & 6 & $-$ \\
        10 & 9 & 11 & 7 & 2 & 5 & 8 & 1 & 4 & 3 & 6 & $+$ \\
        10 & 9 & 11 & 7 & 3 & 5 & 8 & 4 & 2 & 1 & 6 & $-$ \\
        10 & 9 & 11 & 7 & 1 & 5 & 8 & 2 & 4 & 3 & 6 & $-$ \\
        10 & 5 & 9 & 7 & 11 & 4 & 8 & 3 & 2 & 1 & 6 & $-$ \\
        10 & 5 & 2 & 8 & 7 & 4 & 9 & 1 & 3 & 11 & 6 & $-$ \\
        \hhline{|-----------|~}
        \multicolumn{1}{c}{\ru{1.2}\makebox[0pt]{$m_1m_5$}} &
        \multicolumn{1}{c}{\makebox[0pt]{$m_3^{2}$}} &
        \multicolumn{1}{c}{\makebox[0pt]{$m_1^{3}m_3$}} &
        \multicolumn{1}{c}{\makebox[0pt]{$m_1^{2}m_4$}} &
        \multicolumn{1}{c}{\makebox[0pt]{$m_1^{6}$}} &
        \multicolumn{1}{c}{\makebox[0pt]{$m_1m_2m_3$}} &
        \multicolumn{1}{c}{\makebox[0pt]{$m_2m_4$}} &
        \multicolumn{1}{c}{\makebox[0pt]{$m_1^{4}m_2$}} &
        \multicolumn{1}{c}{\makebox[0pt]{$m_2^{3}$}} &
        \multicolumn{1}{c}{\makebox[0pt]{$m_1^{2}m_2^{2}$}} &
        \multicolumn{1}{c}{\makebox[0pt]{$m_6$}} &
        \multicolumn{1}{c}{}
    \end{tabular}
\caption{A permutation table $T \in T_{6,11}$ with $w(T)=m_1^{19}m_2^{8}m_3^{4}m_4^{2}m_5m_6$ and $\sgn T=-1$.}
\label{fig:exatable6x11}
\end{figure}

\subsubsection{Marked permutation tables}
It turns out that we can drastically reduce the number of tables by centering the random entries $X_{ij}$ in $A$. Below we state the basic results. For more details, see \cite{BeckPhD} and \cite{BLP25}. Let $Y_{ij} := X_{ij} - m_1$ and let $\mu_q = \expe{Y_{11}^q}$ be their moments (central moments of $\Omega$). Let $S_n^\times$ be the set of \emph{marked permutations} where at most one element of a permutation on $n$ elements is decorated by a mark ($\times$). Then, by the Matrix Determinant Lemma,
\begin{equation}
\det A = \sum_{\pi \in S_n^\times} \sgn \pi \prod_{i = 1}^n Y_{i \pi(i)}^\times,
\end{equation}
where $Y_{i \pi(i)}^\times = m_1$ if the element $i$ is marked and $Y_{i \pi(i)}^\times = Y_{i \pi(i)}$ otherwise. Again by rasing to the $6$ power, expanding and applying the expectation operator, we get
\begin{equation}
f_6(n) = \sum_{T \in T_{6,n}^\times} w(T) \sgn T,
\end{equation}
where $T_{6,n}^\times := S_n^\times\times S_n^\times\times S_n^\times\times S_n^\times\times S_n^\times\times S_n^\times$ is the set of all \emph{marked permutation tables} with $6$ rows and $n$ columns and where the sign $\sgn T$ corresponds to the sign of the underlying un-marked permutation table (equals the product of signs of the individual permutations) and where
\begin{equation}
w(T) = \Exx \prod_{j=1}^6 \prod_{i=1}^n Y_{i \pi_j(i)}
\end{equation}
is again the \emph{weight} of a given marked permutation table $T \in T_{6,n}^\times$. We also denote $T_6^\times = \bigcup_{n \in \mathbb{N}} T_{6,n}$ as the set of all marked permutation tables with $6$ rows irrespective of the number of columns. The weight $w(T)$ can we written as a product of weights $w(\mathbf{c})$ of individual columns $\mathbf{c}$ of $T$. Note that the column weight $w(\mathbf{c})$ is independent on what element is being covered by the mark $\times$. However, even if not shown next to a mark, this marked element is always uniquely deducible from other elements in the same row since each row contains at most one mark and uses up all elements from $[n]$ (no element repeats in a row of $T$). The types of all columns $\mathbf{c}$ with non-zero weight which can appear in a table $T$ (up to permutation of rows) are shown below (including their weights $w(\mathbf{c})$). Elements $a,b,c,\ldots$ are placeholders representing \emph{different} numbers (that is, $a\neq b, a\neq c, b\neq c,$ etc.). The column types have names (shown in the first row)
which we will use often in our paper. The notation $\times^r_s$ means a column with $r$ marks and, in case of ambiguity, with $s$ the size of the largest unmarked tuple. 
\vspace{-1em}
\bgroup
\renewcommand{\arraystretch}{0.8}
\begin{table}[H]
\centering
\setlength{\tabcolsep}{5.2pt}
\begin{GrayBox}
\vspace{-0.5em}
\centering
\begin{tabular}{cccccccccccc}
    $\mathbf{c}$: & $6$ & $4$ & $3$ & $2$ & $\times_5^1$ & $\times_3^1$ & $\times_4^2$ & $\times_2^2$ & $\times^3$ & $\times^4$ & $\times^6$ \\[0.3em]
    $T_{6,n}^\times:$& \begin{tabular}{|>{\columncolor{mycolor}}c|}
    \hline
        $a$\\
        $a$\\
        $a$\\
        $a$\\
        $a$\\
        $a$\\
    \hline
    \end{tabular}
    &
        \begin{tabular}{|>{\columncolor{mycolor}}c|}
    \hline
        $a$\\
        $a$\\
        $a$\\
        $a$\\
        $b$\\
        $b$\\
    \hline
    \end{tabular}
    &
        \begin{tabular}{|>{\columncolor{mycolor}}c|}
    \hline
        $a$\\
        $a$\\
        $a$\\
        $b$\\
        $b$\\
        $b$\\
    \hline
    \end{tabular}
    &
        \begin{tabular}{|>{\columncolor{mycolor}}c|}
    \hline
        $a$\\
        $a$\\
        $b$\\
        $b$\\
        $c$\\
        $c$\\
    \hline
    \end{tabular}
    &
        \begin{tabular}{|>{\columncolor{mycolor}}c|}
    \hline
        $\!\times\!$\\
        $a$\\
        $a$\\
        $a$\\
        $a$\\
        $a$\\
    \hline
    \end{tabular}
    &
        \begin{tabular}{|>{\columncolor{mycolor}}c|}
    \hline
        $\!\times\!$\\
        $a$\\
        $a$\\
        $a$\\
        $b$\\
        $b$\\
    \hline
    \end{tabular}
    &
        \begin{tabular}{|>{\columncolor{mycolor}}c|}
    \hline
        $\!\times\!$\\
        $\!\times\!$\\
        $a$\\
        $a$\\
        $a$\\
        $a$\\
    \hline
    \end{tabular}
    &
        \begin{tabular}{|>{\columncolor{mycolor}}c|}
    \hline
        $\!\times\!$\\
        $\!\times\!$\\
        $a$\\
        $a$\\
        $b$\\
        $b$\\
    \hline
    \end{tabular}
    &
        \begin{tabular}{|>{\columncolor{mycolor}}c|}
    \hline
        $\!\times\!$\\
        $\!\times\!$\\
        $\!\times\!$\\
        $a$\\
        $a$\\
        $a$\\
    \hline
    \end{tabular}
    &
        \begin{tabular}{|>{\columncolor{mycolor}}c|}
    \hline
        $\!\times\!$\\
        $\!\times\!$\\
        $\!\times\!$\\
        $\!\times\!$\\
        $a$\\
        $a$\\
    \hline
    \end{tabular}
    &
        \begin{tabular}{|>{\columncolor{mycolor}}c|}
    \hline
        $\!\times\!$\\
        $\!\times\!$\\
        $\!\times\!$\\
        $\!\times\!$\\
        $\!\times\!$\\
        $\!\times\!$\\
    \hline
    \end{tabular}
    \\[2.6em]
    $w(\mathbf{c})$: & $\mu_6$ & $\mu_4$ & $\mu_3^2$ & $1$ & $m_1\mu_5$ & $m_1\mu_3$ & $m_1^2\mu_4$ & $m_1^2$ & $m_1^3\mu_3$ & $m_1^4$ & $m_1^6$
\end{tabular}
\vspace{-0.6em}
\end{GrayBox}
\end{table}
\egroup
\vspace{-1em}
\noindent
An example of a marked permutation table $T$ is shown below in Figure \ref{fig:markedtable}. Its sixth column corresponds to the term $m_1 Y_{68}^2 Y_{62}^3$ in the expansion of $(\det A)^6$ 
and its weight is $w(\mathbf{c}) = \Exx m_1 Y_{68}^2 Y_{62}^3 = m_1 \mu_3$ (recall we assume $\mu_2 = 1$).

\begin{figure}[H]
\centering
    \begin{tabular}{|*{8}{>{\columncolor{mycolor}[\tabcolsep][\tabcolsep]\centering}p{3.1em}|}>{\centering\arraybackslash}p{1.5em}}
    \hhline{|--------|~}
    \ru{0.9}5 & 3 & 1 & 2 & $\!\!\!\!\!\times{4}$ & 8 & 6 & 7 & $+$ \\
        5 & 3 & 1 & 8 & 4 & $\!\!\!\!\!\times{2}$ & 6 & 7 & $-$ \\
        5 & 3 & 6 & 8 & 4 & 2 & 7 & $\!\!\!\!\!\times{1}$ & $-$ \\
        5 & 3 & 6 & 2 & 4 & 8 & 7 & $\!\!\!\!\!\times{1}$ & $+$ \\
        5 & 8 & 1 & 3 & 4 & 2 & 6 & 7 & $+$ \\
        5 & 8 & 6 & 3 & 4 & 2 & $\!\!\!\!\!\times{1}$ & 7 & $-$ \\
        \hhline{|--------|~}
        \multicolumn{1}{c}{\ru{1.2}\makebox[0pt]{$\mu_6$}} &
        \multicolumn{1}{c}{\makebox[0pt]{$\mu_4$}} &
        \multicolumn{1}{c}{\makebox[0pt]{$\mu_3^2$}} &
        \multicolumn{1}{c}{\makebox[0pt]{$1$}} &
        \multicolumn{1}{c}{\makebox[0pt]{$m_1\mu_5$}} &
        \multicolumn{1}{c}{\makebox[0pt]{$m_1\mu_3$}} &
        \multicolumn{1}{c}{\makebox[0pt]{$m_1\mu_3$}} &
        \multicolumn{1}{c}{\makebox[0pt]{$m_1^2\mu_4$}} &
        \multicolumn{1}{c}{}
    \end{tabular}
\caption{A permutation table $T\in T_{6,8}^\times$ with columns $6, 4, 3, 2, \times_5^1, \times_3^1, \times_3^1, \times_4^2$, carrying $5$ marks in distinct rows, $w(T)=m_1^5\mu_3^4\mu_5\mu_4^2\mu_6$, $\sgn\tau=-1$.}
\label{fig:markedtable}
\end{figure}

\subsubsection{Paired/Blocked tables}
From now on, we assume $\mu_ 2 = 1$. To utilize the correspondence with the known machinery for the Gaussian case (in \cite{BLP23}) for which $\mu_6 = 15, \mu_5 = 0, \mu_4 = 3$ and $\mu_3 = 0$, we introduce another set of tables.

\begin{definition}
Let the \emph{paired/blocked tables} $\underline{T}_{6,n}^\times$ be the marked permutation tables $T_{6,n}^\times$ where each time an element appears unmarked $4$ or more times in a column, we either have all of the unmarked copies of this element in this column be a block or we partition these copies into pairs (or a pair and a triple if the column has $5$ unmarked copies of the element), 
The weights of these columns are 
carefully chosen via replacements (inspired by the inclusion/exclusion in \cite{BLP23}) such that in the end we have
\begin{equation}\label{eq:InclExcl}
    f_6(n) = \sum_{T \in T_{6,n}^\times} w(T) \sgn T = \sum_{T \in \bunderline{T}_{6,n}^\times} w(T) \sgn T.
\end{equation}
Columns of $\bunderline{T}_6^\times = \bigcup_{n \in \mathbb{N}} \bunderline{T}_{6,n}^\times$ with their weights are shown below. To distiguish them from columns of tables in $T_6^\times$, we use an underline, so a $\underline{6}-$column is a column of a table in $\bunderline{T}_6^{\times}$ 
If a column with no marks has a block of size $4$ or more, we say it is a \emph{known} column. If a column has four or more copies of an element but has no blocks or marks, we say it is an \emph{unknown} column. Note that unknown columns must be $\underline{2}$ columns where some of the elements $a,b,c$ are equal to each other.

\vspace{-1em}
\bgroup
\renewcommand{\arraystretch}{0.8}
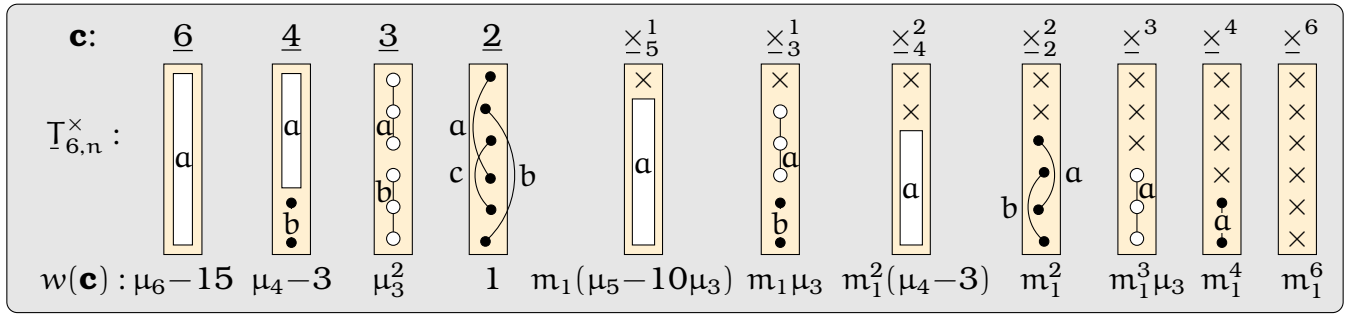
\begin{table}[H]
\centering
\setlength{\tabcolsep}{5.2pt}
\begin{GrayBox}
\vspace{-0.5em}
\centering
\begin{tabular}{cccccccccccc}
    $\mathbf{c}$: & $\underline{6}$ & $\underline{4}$ & $\underline{3}$ & $\underline{2}$ & ${\underline{\!\times\!}\,}_5^1$ & ${\underline{\!\times\!}\,}_3^1$ & ${\underline{\!\times\!}\,}_4^2$ & ${\underline{\!\times\!}\,}_2^2$ & ${\underline{\!\times\!}\,}^3$ & ${\underline{\!\times\!}\,}^4$ & ${\underline{\!\times\!}\,}^6$ \\[0.0em]
    $\!\!\bunderline{T}_{6,n}^\times:\!\!$& \begin{tikzpicture}[baseline = 8.2ex,scale = 0.42]
        \draw[fill=mycolor] (1.2-.6, 0.5) rectangle (1 *1.2+.6, 6.5);
        \pairA[1.2];
        \draw[fill=white] (1*1.2-0.3, 0.8) rectangle (1*1.2+0.3, 6.2);
        \node[vertex] (a) at (1*1.2,3.5) {$a$};
        \end{tikzpicture}
    &
        \begin{tikzpicture}[baseline = 8.2ex,scale = 0.42]
        \draw[fill=mycolor] (1.2-.6, 0.5) rectangle (1 *1.2+.6, 6.5);
        \pairA[1.2];
        \draw[fill=white] (1*1.2-0.3, 2.6) rectangle (1*1.2+0.3, 6.2);
        \node[vertex] (a) at (1*1.2,4.5) {$a$};
        \draw[pair] (A1.center) to node[fill=mycolor, inner sep=1.0pt] {$b$} (A2.center);
        \end{tikzpicture}

    &
        \begin{tikzpicture}[baseline = 8.2ex,scale = 0.42]
        \draw[fill=mycolor] (1.3-0.6, 0.5) rectangle (1*1.3+0.6, 6.5);
        \pairA[1.3];
        \draw (A1.center) -- (A3.center);
        \draw (A4.center) -- (A6.center);
        \foreach \i in {1,...,6}
            \draw[fill=white, draw=black] (A\i) circle (6pt);
        \node[vertex] (a) at (1.0,4.5) {$a$};
        \node[vertex] (a) at (1.0,2.5) {$b$};
        \end{tikzpicture}
    &
        \begin{tikzpicture}[baseline = 8.2ex,scale = 0.42]
        \draw[fill=mycolor] (1.3-0.6, 0.5) rectangle (1*1.3+0.6, 6.5);
        \pairA[1.3];
        \draw[pair] (A1.center) to [bend right=45] node[inner sep=1.0pt] {\,\,\,\,\,$b$} (A5.center);
        \draw[pair] (A3.center) to [bend left=33] node[inner sep=1.0pt] {$a$\,\,\,\,\,\,} (A6.center);
        \draw[pair] (A2.center) to [bend left=45] node[inner sep=1.0pt] {$c$\,\,\,\,\,\,\,} (A4.center);
        \end{tikzpicture}
    &
        \begin{tikzpicture}[baseline = 8.2ex,scale = 0.42]
        \draw[fill=mycolor] (1.3-0.6, 0.5) rectangle (1*1.3+0.6, 6.5);
        \pairA[1.3];
        \node at (A6) {$\times$};
        \draw[fill=white] (1*1.2-0.25, 0.8) rectangle (1*1.2+0.45, 5.4);
        \node[vertex] (a) at (1*1.2+0.07,3.3) {$a$};
        \end{tikzpicture}
    &
        \begin{tikzpicture}[baseline = 8.2ex,scale = 0.42]
        \draw[fill=mycolor] (1.3-0.6, 0.5) rectangle (1*1.3+0.6, 6.5);
        \pairA[1.3];
        \node at (A6) {$\times$};
        \draw (A3.center) -- (A5.center);
        \foreach \i in {3,...,5}
            \draw[fill=white, draw=black] (A\i) circle (6pt);
        \node[vertex] (a) at (1.6,3.5) {$a$};
        \draw[pair] (A1.center) to node[fill=mycolor, inner sep=1.0pt] {$b$} (A2.center);
        \end{tikzpicture}
    &
        \begin{tikzpicture}[baseline = 8.2ex,scale = 0.42]
        \draw[fill=mycolor] (1.3-0.6, 0.5) rectangle (1*1.3+0.6, 6.5);
        \pairA[1.3];
        \node at (A6) {$\times$};
        \node at (A5) {$\times$};
        \draw[fill=white] (1*1.2-0.25, 0.8) rectangle (1*1.2+0.45, 4.4);
        \node[vertex] (a) at (1*1.2+0.07,2.5) {$a$};
        \end{tikzpicture}
    &
        \begin{tikzpicture}[baseline = 8.2ex,scale = 0.42]
        \draw[fill=mycolor] (1.3-0.6, 0.5) rectangle (1*1.3+0.6, 6.5);
        \pairA[1.3];
        \node at (A6) {$\times$};
        \node at (A5) {$\times$};
        \draw[pair] (A1.center) to [bend left=45] node[inner sep=1.0pt] {$b$\,\,\,\,\,\,\,} (A3.center);
        \draw[pair] (A2.center) to [bend left=-45] node[inner sep=1.0pt] {\,\,\,\,\,\,\,$a$} (A4.center);
        \end{tikzpicture}
    &
        \begin{tikzpicture}[baseline = 8.2ex,scale = 0.42]
        \draw[fill=mycolor] (1.3-0.6, 0.5) rectangle (1*1.3+0.6, 6.5);
        \pairA[1.3];
        \node at (A6) {$\times$};
        \node at (A5) {$\times$};
        \node at (A4) {$\times$};
        \draw (A1.center) -- (A3.center);
        \foreach \i in {1,...,3}
            \draw[fill=white, draw=black] (A\i) circle (6pt);
        \node[vertex] (a) at (1.6,2.5) {$a$};
        \end{tikzpicture}
    &
        \begin{tikzpicture}[baseline = 8.2ex,scale = 0.42]
        \draw[fill=mycolor] (1.3-0.6, 0.5) rectangle (1*1.3+0.6, 6.5);
        \pairA[1.3];
        \node at (A6) {$\times$};
        \node at (A5) {$\times$};
        \node at (A4) {$\times$};
        \node at (A3) {$\times$};
        \draw[pair] (A1.center) to node[fill=mycolor, inner sep=1.0pt] {$a$} (A2.center);
        \end{tikzpicture}
    &
        \begin{tikzpicture}[baseline = 8.2ex,scale = 0.42]
        \draw[fill=mycolor] (1.3-0.6, 0.5) rectangle (1*1.3+0.6, 6.5);
        \pairA[1.3];
        \node at (A6) {$\times$};
        \node at (A5) {$\times$};
        \node at (A4) {$\times$};
        \node at (A3) {$\times$};
        \node at (A2) {$\times$};
        \node at (A1) {$\times$};
        \end{tikzpicture}
    \\[2.6em]
    $\!\!w(\mathbf{c}):\!\!$ & $\!\!\mu_6\!-\!15\!\!$ & $\mu_4\!-\!3$ & $\mu_3^2$ & $1$ & $\!\!\!\!\!\!\!m_1(\mu_5\!-\!10\mu_3)\!\!\!$ & $m_1\mu_3$ & $\!\!m_1^2(\mu_4\!-\!3)\!\!\!\!$ & $m_1^2$ & $\!\!m_1^3\mu_3\!\!\!\!$ & $m_1^4$ & $m_1^6$
\end{tabular}
\vspace{-0.6em}
\end{GrayBox}
\end{table}
\egroup
\vspace{-1em}

\noindent
An example of a paired/blocked table is shown below in Figure \ref{fig:pairblockT}.
\begin{figure}[H]
\centering
\begin{tikzpicture}[baseline = 8.2ex,scale = 0.70]
        \pairABCDEFGHI[1.3];
        \draw (1.3-1, 0.35) rectangle (9*1.3+1, 6.65);
        \draw[fill=mycolor,draw=none] (1.3-1, 0.5) rectangle (9*1.3+1, 6.5);
        \draw[fill=white] (2*1.3-0.3, 0.8) rectangle (2*1.3+0.3, 6.2);
        \node[vertex] (3) at (2*1.3,3.5) {3};
        \draw[fill=white] (3*1.3-0.3, 0.8) rectangle (3*1.3+0.3, 5.4);
        \node[vertex] (1) at (3*1.3,3.2) {1};
        \draw[fill=white] (7*1.3-0.3, 2.6) rectangle (7*1.3+0.3, 6.2);
        \node[vertex] (9) at (7*1.3,4.5) {9};
        \draw[pair] (A1.center) to [bend left=25] node[fill=mycolor, inner sep=1.0pt] {4} (A6.center);
        \draw[pair] (A2.center) to node[fill=mycolor, inner sep=1.0pt] {4} (A3.center);
        \draw[pair] (A4.center) to node[fill=mycolor, inner sep=1.0pt] {4} (A5.center);
        \node at (C6) {$\times 1$\!\!};
        \draw (D1.center) -- (D3.center);
        \draw (D4.center) -- (D6.center);
        \foreach \i in {1,...,6}
            \draw[fill=white, draw=black] (D\i) circle (6pt);
        \node[vertex] (a) at ($(D5) + (+0.5,0)$) {2};
        \node[vertex] (b) at ($(D2) + (+0.5,0)$) {5};
        \draw[pair] (E1.center) to node[fill=mycolor, inner sep=1.0pt] {6} (E2.center);
        \draw[pair] (E3.center) to [bend left=-45] node[fill=mycolor, inner sep=1.0pt] {6} (E6.center);
        \draw[pair] (E4.center) to node[fill=mycolor, inner sep=1.0pt] {8} (E5.center);
        \draw[pair] (F1.center) to [bend left=45] node[fill=mycolor, inner sep=1.0pt] {7} (F4.center);
        \node at (F2) {$\times 7$\!\!};
        \node at (F3) {$\times 7$\!\!};
        \draw[pair] (F5.center) to node[fill=mycolor, inner sep=1.0pt] {7} (F6.center);
        \draw[pair] (G1.center) to node[fill=mycolor, inner sep=1.0pt] {8} (G2.center);
        \draw[pair] (H1.center) to node[fill=mycolor, inner sep=1.0pt] {9} (H2.center);
        \draw[pair] (H3.center) to [bend left=45] node[fill=mycolor, inner sep=1.0pt] {8} (H6.center);
        \draw[pair] (H4.center) to node[fill=mycolor, inner sep=1.0pt] {6} (H5.center);
        \draw (I1.center) -- (I3.center);
        \draw (I4.center) -- (I6.center);
        \foreach \i in {1,...,6}
            \draw[fill=white, draw=black] (I\i) circle (6pt);
        \node[vertex] (a) at ($(I5) + (+0.5,0)$) {5};
        \node[vertex] (b) at ($(I2) + (+0.5,0)$) {2};        
        \draw (1.3-1, 0.5) rectangle (9*1.3+1, 6.5);
        \node at ($(I6)+(1.7,0)$) {$+$};
        \node at ($(I5)+(1.7,0)$) {$-$};
        \node at ($(I4)+(1.7,0)$) {$-$};
        \node at ($(I3)+(1.7,0)$) {$-$};
        \node at ($(I2)+(1.7,0)$) {$+$};
        \node at ($(I1)+(1.7,0)$) {$+$};
        \end{tikzpicture}
\caption{A paired/blocked table $T \in \bunderline{T}_{6,9}^\times$ with $w(T) = m_1^3 \mu_3^4 (\mu_4-3) (\mu_5-10\mu_3)(\mu_6-15)$ and $\sgn T = -1$.}
\label{fig:pairblockT}
\end{figure}
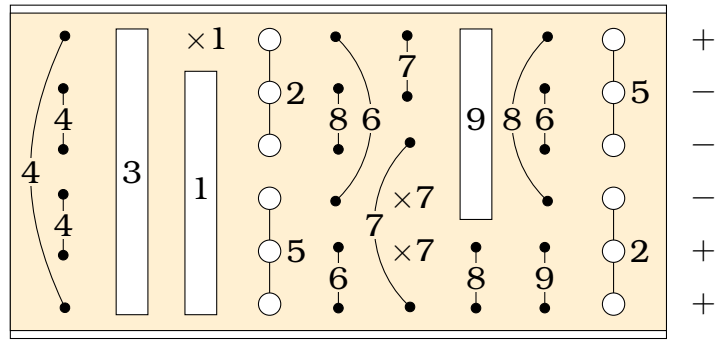

\end{definition}

\noindent
Finally, we show the replacements rules for columns. For a given table $T \in T_{6,n}^\times$, we generate a set of tables $T' \in \bunderline{T}_{6,n}^\times$ in which each column $\mathbf{c}$ in $T$ was replaced by blocked/paired columns $\mathbf{c}'$. This replacement is not a bijection since for a given column, we may replace it by many possible blocked/paired columns, so the set of $T'$ tables may be large. However, Equation \eqref{eq:InclExcl} stays preserved if
\begin{itemize}
    \item for any $T' \in \underline{T}_6^\times$, there is a unique map from $T'$ to $T$ and
    \item sum of weights of all columns $\mathbf{c}'$ equals $w(\mathbf{c})$.
\end{itemize}
Both conditions above can be satisfied by replacement rules shown in Figures \ref{fig:exclinclF6col4} -- \ref{fig:exclinclF6colx42}. The remaining correspondences of columns not shown here are simply bijectively paired up, so for example $2$-columns maps to $\underline{2}$ by pairing up the identical elements (only one pairing exist among those columns), the same goes with three-columns, that is $3$-columns match $\underline{3}$-columns etc. Those columns do not have to be called as known/unknown.

\begin{minipage}{.5\textwidth}
\begin{figure}[H]
    \centering
\begin{tabular}{ccccccccc}
    $4$ & & $\underline{4}$ & & $\underline{2}$ &  & $\underline{2}$ & & $\underline{2}$\\
    \begin{tikzpicture}[baseline = 6.2ex,scale = 0.42]
        \draw[fill=mycolor] (1.2-.6, 0.5) rectangle (1 *1.2+.6, 6.5);
        \node[vertex] (A1) at (1*1.2,1) {$b$};
        \node[vertex] (A2) at (1*1.2,2) {$b$};
        \node[vertex] (A3) at (1*1.2,3) {$a$};
        \node[vertex] (A4) at (1*1.2,4) {$a$};
        \node[vertex] (A5) at (1*1.2,5) {$a$};
        \node[vertex] (A6) at (1*1.2,6) {$a$};
        \end{tikzpicture}
    & \!\!$=$\!\!&
        \begin{tikzpicture}[baseline = 6.2ex,scale = 0.42]
        \draw[fill=mycolor] (1.2-.6, 0.5) rectangle (1 *1.2+.6, 6.5);
        \pairA[1.2];
        \draw[fill=white] (1*1.2-0.3, 2.6) rectangle (1*1.2+0.3, 6.2);
        \node[vertex] (a) at (1*1.2,4.5) {$a$};
        \draw[pair] (A1.center) to node[fill=mycolor, inner sep=1.0pt] {$b$} (A2.center);
        \end{tikzpicture}
    & \!\!\!\!\!$+$ &
        \begin{tikzpicture}[baseline = 6.2ex,scale = 0.42]
        \draw[fill=mycolor] (1.2-.5, 0.5) rectangle (1 *1.2+.5, 6.5);
        \pairA[1.2];
        \draw[pair] (A1.center) to node[fill=mycolor, inner sep=1.0pt] {$b$} (A2.center);
        \draw[pair] (A3.center) to node[fill=mycolor, inner sep=1.0pt] {$a$} (A4.center);
        \draw[pair] (A5.center) to node[fill=mycolor, inner sep=1.0pt] {$a$} (A6.center);
        \end{tikzpicture}
    & \!\!$+$\!\! &
        \begin{tikzpicture}[baseline = 6.2ex,scale = 0.42]
        \draw[fill=mycolor] (1.2-.7, 0.5) rectangle (1 *1.2+.5, 6.5);
        \pairA[1.2];
        \draw[pair] (A1.center) to node[fill=mycolor, inner sep=1.0pt] {$b$} (A2.center);
        \draw[pair] (A4.center) to node[fill=mycolor, inner sep=1.0pt] {$a$} (A5.center);
        \draw[pair] (A3.center) to [bend left=40] node[fill=mycolor, inner sep=1.0pt] {$\!a\!$} (A6.center);
        \draw (1.2-.7, 0.5) rectangle (1 *1.2+.5, 6.5);
        \end{tikzpicture}
    & \!\!$+$\!\! &
        \begin{tikzpicture}[baseline = 6.2ex,scale = 0.42]
        \draw[fill=mycolor,draw=none] (1.2-.7, 0.5) rectangle (1 *1.2+.7, 6.5);
        \pairA[1.2];
        \draw[pair] (A1.center) to node[fill=mycolor, inner sep=1.0pt] {$b$} (A2.center);
        \draw[pair] (A3.center) to [bend right=60] node[fill=mycolor, inner sep=1.0pt] {$\!a\!$} (A5.center);
        \draw[pair] (A4.center) to [bend left=60] node[fill=mycolor, inner sep=1.0pt] {$\!a\!$} (A6.center);
        \draw (1.2-.7, 0.5) rectangle (1 *1.2+.7, 6.5);
        \end{tikzpicture}
    \\[2.5em]
    $\mu_4$ & & $\!\!\!\!\!\!\!\!\mu_4\!-\!3\!\!\!\!\!\!\!\!$ & & $1$ &  & $1$ & & $1$
\end{tabular}
    \caption{Replacement of $4$-columns}
    \label{fig:exclinclF6col4}
\end{figure}
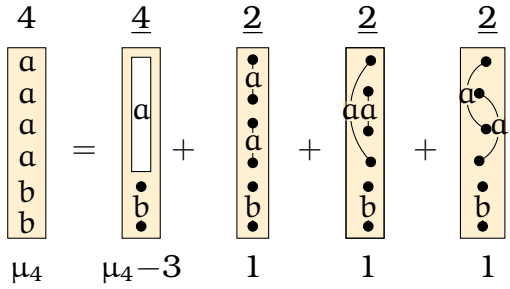
\end{minipage}
\begin{minipage}{.48\textwidth}
\begin{figure}[H]
    \centering
\begin{tabular}{ccccccccc}
    $6$ & & $\underline{6}$ & & $\underline{2}$ &  & $\underline{2}$ & & $\underline{2}$\\
    \begin{tikzpicture}[baseline = 6.2ex,scale = 0.42]
        \draw[fill=mycolor] (1.2-.6, 0.5) rectangle (1 *1.2+.6, 6.5);
        \node[vertex] (A1) at (1*1.2,1) {$a$};
        \node[vertex] (A2) at (1*1.2,2) {$a$};
        \node[vertex] (A3) at (1*1.2,3) {$a$};
        \node[vertex] (A4) at (1*1.2,4) {$a$};
        \node[vertex] (A5) at (1*1.2,5) {$a$};
        \node[vertex] (A6) at (1*1.2,6) {$a$};
        \end{tikzpicture}
    & \!\!$=$\!\! &
        \begin{tikzpicture}[baseline = 6.2ex,scale = 0.42]
        \draw[fill=mycolor] (1.2-.6, 0.5) rectangle (1 *1.2+.6, 6.5);
        \pairA[1.2];
        \draw[fill=white] (1*1.2-0.3, 0.8) rectangle (1*1.2+0.3, 6.2);
        \node[vertex] (a) at (1*1.2,3.35) {$a$};
        \end{tikzpicture}
    & \!\!\!\!\!$+$\!\! &
        \begin{tikzpicture}[baseline = 6.2ex,scale = 0.42]
        \draw[fill=mycolor] (1.2-.5, 0.5) rectangle (1 *1.2+.5, 6.5);
        \pairA[1.2];
        \draw[pair] (A1.center) to node[fill=mycolor, inner sep=1.0pt] {$a$} (A2.center);
        \draw[pair] (A3.center) to node[fill=mycolor, inner sep=1.0pt] {$a$} (A4.center);
        \draw[pair] (A5.center) to node[fill=mycolor, inner sep=1.0pt] {$a$} (A6.center);
        \end{tikzpicture}
    & \!\!\!\!$+$\!\!\!\! &
        \begin{tikzpicture}[baseline = 6.2ex,scale = 0.42]
        \draw[fill=mycolor] (1.2-.7, 0.5) rectangle (1 *1.2+.5, 6.5);
        \pairA[1.2];
        \draw[pair] (A1.center) to node[fill=mycolor, inner sep=1.0pt] {$a$} (A2.center);
        \draw[pair] (A3.center) to [bend left=40] node[fill=mycolor, inner sep=1.0pt] {$\!a\!$} (A6.center);
        \draw[pair] (A4.center) to node[fill=mycolor, inner sep=1.0pt] {$\!a\!$} (A5.center);
        \draw (1.2-.7, 0.5) rectangle (1 *1.2+.5, 6.5);
        \end{tikzpicture}
    & \!\!$+\, \cdots\, +$\!\! &
        \begin{tikzpicture}[baseline = 6.2ex,scale = 0.42]
        \draw[fill=mycolor,draw=none] (1.2-.7, 0.5) rectangle (1 *1.2+.7, 6.5);
        \pairA[1.2];
        \draw[pair] (A1.center) to node[fill=mycolor, inner sep=1.0pt] {$a$} (A2.center);
        \draw[pair] (A3.center) to [bend right=56] node[fill=mycolor, inner sep=1.0pt] {$\!a\!$} (A5.center);
        \draw[pair] (A4.center) to [bend left=56] node[fill=mycolor, inner sep=1.0pt] {$\!a\!$} (A6.center);
        \draw (1.2-.7, 0.5) rectangle (1 *1.2+.7, 6.5);
        \end{tikzpicture}
    \\[2.5em]
    $\mu_6$ & & $\!\!\!\!\!\!\!\!\mu_6\!-\!15\!\!\!\!\!\!\!\!$ & & $1$ &  & $1$ & & $1$
\end{tabular}
    \caption{Replacement of $6$-columns}
    \label{fig:exclinclF6col6}
\end{figure}
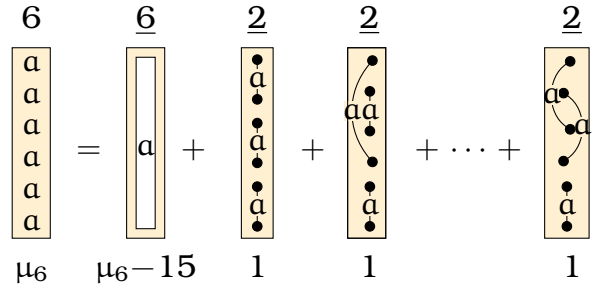
\end{minipage}

\begin{minipage}{.48\textwidth}
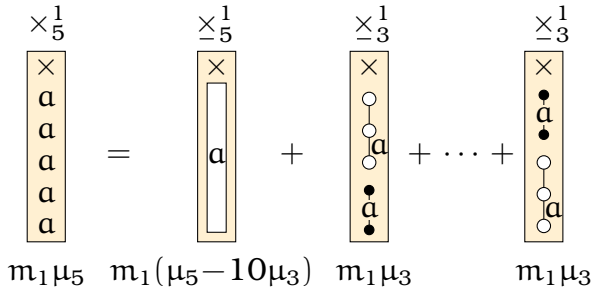
\begin{figure}[H]
    \centering
\begin{tabular}{ccccccc}
    $\times_5^1$ & & ${\underline{\!\times\!}\,}_5^1$ & & ${\underline{\!\times\!}\,}_3^1$ & & ${\underline{\!\times\!}\,}_3^1$\\
    \begin{tikzpicture}[baseline = 6.2ex,scale = 0.42]
        \draw[fill=mycolor] (1.2-.6, 0.5) rectangle (1 *1.2+.6, 6.5);
        \node[vertex] (A1) at (1*1.2,1) {$a$};
        \node[vertex] (A2) at (1*1.2,2) {$a$};
        \node[vertex] (A3) at (1*1.2,3) {$a$};
        \node[vertex] (A4) at (1*1.2,4) {$a$};
        \node[vertex] (A5) at (1*1.2,5) {$a$};
        \node[vertex] (A6) at (1*1.2,6) {$\times$};
        \end{tikzpicture}
    & \!\!$=$\!\! &
        \begin{tikzpicture}[baseline = 6.2ex,scale = 0.42]
        \draw[fill=mycolor] (1.2-.6, 0.5) rectangle (1 *1.2+.6, 6.5);
        \pairA[1.2];
        \draw[fill=white] (1*1.2-0.3, 0.8) rectangle (1*1.2+0.3, 5.5);
        \node[vertex] (a) at (1*1.2,3.2) {$a$};
        \node[vertex] (A6) at (1*1.2,6) {$\times$};
        \end{tikzpicture}
    & \!\!\!\!\!\!\!\!\!\!$+$\!\! &
       \begin{tikzpicture}[baseline = 6.2ex,scale = 0.42]
        \draw[fill=mycolor] (1.3-0.6, 0.5) rectangle (1*1.3+0.6, 6.5);
        \pairA[1.3];
        \node at (A6) {$\times$};
        \draw (A3.center) -- (A5.center);
        \foreach \i in {3,...,5}
            \draw[fill=white, draw=black] (A\i) circle (6pt);
        \node[vertex] (a) at (1.6,3.5) {$a$};
        \draw[pair] (A1.center) to node[fill=mycolor, inner sep=1.0pt] {$a$} (A2.center);
        \end{tikzpicture}
    & \!\!\!\!\!$+ \,\cdots\, +$\!\!\!\!\! &
       \begin{tikzpicture}[baseline = 6.2ex,scale = 0.42]
        \draw[fill=mycolor] (1.3-0.6, 0.5) rectangle (1*1.3+0.6, 6.5);
        \pairA[1.3];
        \node at (A6) {$\times$};
        \draw (A1.center) -- (A3.center);
        \foreach \i in {1,...,3}
            \draw[fill=white, draw=black] (A\i) circle (6pt);
        \node[vertex] (a) at (1.6,1.5) {$a$};
        \draw[pair] (A4.center) to node[fill=mycolor, inner sep=1.0pt] {$a$} (A5.center);
        \end{tikzpicture}
    \\[2.5em]
    $m_1 \mu_5$ & & $\!\!\!\!\!\!\!\!m_1(\mu_5\!-\!10\mu_3)\!\!\!\!\!\!$ & & $\!\!\!\!m_1\mu_3\!\!\!\!$ & & $\!\!\!\!m_1\mu_3\!\!\!\!$
\end{tabular}
    \caption{Replacement of $\times_5^1$-columns}
    \label{fig:exclinclF6colx51}
\end{figure}
\end{minipage}
\begin{minipage}{.48\textwidth}
\begin{figure}[H]
    \centering
\begin{tabular}{ccccccccc}
    $\times_4^2$ & & ${\underline{\!\times\!}\,}_4^2$ & & ${\underline{\!\times\!}\,}_2^2$ &  & ${\underline{\!\times\!}\,}_2^2$ & & ${\underline{\!\times\!}\,}_2^2$\\
    \begin{tikzpicture}[baseline = 6.2ex,scale = 0.42]
        \draw[fill=mycolor] (1.2-.6, 0.5) rectangle (1 *1.2+.6, 6.5);
        \node[vertex] (A1) at (1*1.2,1) {$a$};
        \node[vertex] (A2) at (1*1.2,2) {$a$};
        \node[vertex] (A3) at (1*1.2,3) {$a$};
        \node[vertex] (A4) at (1*1.2,4) {$a$};
        \node[vertex] (A5) at (1*1.2,5) {$\times$};
        \node[vertex] (A6) at (1*1.2,6) {$\times$};
        \end{tikzpicture}
    & \!\!$=$\!\! &
        \begin{tikzpicture}[baseline = 6.2ex,scale = 0.42]
        \draw[fill=mycolor] (1.2-.6, 0.5) rectangle (1 *1.2+.6, 6.5);
        \pairA[1.2];
        \draw[fill=white] (1*1.2-0.3, 0.8) rectangle (1*1.2+0.3, 4.4);
        \node[vertex] (a) at (1*1.2,2.5) {$a$};
        \node[vertex] (A5) at (1*1.2,5) {$\times$};
        \node[vertex] (A6) at (1*1.2,6) {$\times$};
        \end{tikzpicture}
    & \!\!\!\!\!$+$\!\! &
        \begin{tikzpicture}[baseline = 6.2ex,scale = 0.42]
        \draw[fill=mycolor] (1.2-.5, 0.5) rectangle (1 *1.2+.5, 6.5);
        \pairA[1.2];
        \draw[pair] (A1.center) to node[fill=mycolor, inner sep=1.0pt] {$b$} (A2.center);
        \draw[pair] (A3.center) to node[fill=mycolor, inner sep=1.0pt] {$a$} (A4.center);
        \node[vertex] (A5) at (1*1.2,5) {$\times$};
        \node[vertex] (A6) at (1*1.2,6) {$\times$};
        \end{tikzpicture}
    & \!\!$+$\!\! &
        \begin{tikzpicture}[baseline = 6.2ex,scale = 0.42]
        \draw[fill=mycolor] (1.2-.7, 0.5) rectangle (1 *1.2+.5, 6.5);
        \pairA[1.2];
        \draw[pair] (A2.center) to node[fill=mycolor, inner sep=1.0pt] {$a$} (A3.center);
        \draw[pair] (A1.center) to [bend left=40] node[fill=mycolor, inner sep=1.0pt] {$\!a\!$} (A4.center);
        \node[vertex] (A5) at (1*1.2,5) {$\!\times$};
        \node[vertex] (A6) at (1*1.2,6) {$\!\times$};
        \draw (1.2-.7, 0.5) rectangle (1 *1.2+.5, 6.5);
        \end{tikzpicture}
    & \!\!$+$\!\! &
        \begin{tikzpicture}[baseline = 6.2ex,scale = 0.42]
        \draw[fill=mycolor,draw=none] (1.2-.7, 0.5) rectangle (1 *1.2+.7, 6.5);
        \pairA[1.2];
        \draw[pair] (A1.center) to [bend right=60] node[fill=mycolor, inner sep=1.0pt] {$\!a\!$} (A3.center);
        \draw[pair] (A2.center) to [bend left=60] node[fill=mycolor, inner sep=1.0pt] {$\!a\!$} (A4.center);
        \node[vertex] (A5) at (1*1.2,5) {$\!\times$};
        \node[vertex] (A6) at (1*1.2,6) {$\!\times$};
        \draw (1.2-.7, 0.5) rectangle (1 *1.2+.7, 6.5);
        \end{tikzpicture}
    \\[2.5em]
    $\mu_4$ & & $\mu_4\!-\!3$ & & $1$ &  & $1$ & & $1$
\end{tabular}
    \caption{Replacement of $\times_4^2$-columns}
    \label{fig:exclinclF6colx42}
\end{figure}
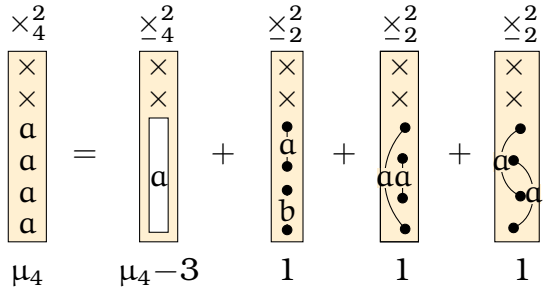
\end{minipage}

\vspace{1em}
Note that with these replacement rules, we have the restriction that for $\underline{3}$ and $\underline{4}$ columns, $a \neq b$. For the other column types, some or all of the elements with marks and/or different labels may be equal to each other.

Throughout the rest of the paper, we assume $T$ is a paired/blocked marked permutation table.

\subsubsection{Connection formula}\label{sec:connectionformula}
The ${\underline{\!\times\!}\,}_5^1$ columns are always disjoint from each other and the rest of the table. The fact that there cannot be a single uncovered element together with the rule that there is at most one mark per row forces the covered element in a ${\underline{\!\times\!}\,}_5^1$ column to always coincide with the block element. Hence, in terms of generating functions, one gets, by collecting those ${\underline{\!\times\!}\,}_5^1$ columns,
\begin{equation}
F_6(t) = \sum _{r=0}^6 m_1^r \left(1+\left(\mu _5-10\mu_3\right) m_1 t\right){}^{6-r} Z_r(t),
\end{equation}
where $Z_r(t)$ are the generating functions of pair/block tables with exactly $r$ marks which do not contain ${\underline{\!\times\!}\,}_5^1$ columns. We further set $m_1 = 1$ in those tables since $m_1^r$ has been already pulled out.

\subsubsection{Factorisation of disjoint components}\label{sec:disjointcomponentfactorisation}
We utilize the fact the exponential generating function (EGF) of permutation tables (of any kind) factorizes over disjoint columns (see \cite{BeckPhD}). Already in the case $m_1 = 0$ treated in \cite{BLP25}, we identified the 
three types of sub-components of paired/blocked tables which are disjoint from everything else, namely $\underline{6}$-columns, cycles of $\underline{4}$ columns and cycles of $\underline{3}$ columns. We call the union of these three sub-components the \emph{floating component} of the table. In terms of EGFs, they contribute a factor
\begin{equation}
Fl(t) = \underbrace{e^{(\mu_6-15)t}}_{\text{known $6$-columns}} \underbrace{\frac{e^{-15(\mu_4\!-3)t}}{(1-(\mu_4-3)t)^{15}}}_{\text{cycles of known $4$-columns}} \underbrace{(1\!+\! \mu_3^2t)^{10} e^{-10\mu_3^2t}}_{\text{cycles of $3$-columns}} = O(t)/N(t'),
\end{equation}
where the last equality comes from Equation \eqref{Eq:O}. This factor must appear also in the general case. More concretely, the generating function $F_6(t)$ corresponding to all paired/blocked tables equals to the product of $O(t)/N(t')$ with a generating function corresponding to all paired/blocked tables \textbf{lacking} $\underline{6}$-columns, cycles of $\underline{4}$ columns and cycles of $\underline{3}$ columns.

\section{Decomposition of a Table Based on its Shell and Core}
Given a table $T$, we can perform the following operations to reduce $T$ to a simpler table.
\begin{definition}
Given a table $T$, we say that an element $a$ is uncovered if it appears unmarked $6$ times in $T$ (i.e., it is never replaced by an $\times$). Otherwise, we say that the element $a$ is covered.
\end{definition}
\begin{definition}
Given a table $T$, we define the reduced table $T_{red}$ to be the table which is obtained from $T$ after applying the following steps:
\begin{enumerate}
\item If there is an uncovered element $a$ which appears as a triple or block of size $4$ 
in a column which also contains an $\times$, we can delete the entries containing $a$ and then merge the two columns which contained these entries. After doing this merge, we put the resulting column in the same place as the column which contained an $a$ but did not have an $a$ in its first row (deleting the column which had an $a$ in its first row). We repeat this step until there are no such uncovered elements $a$. 
\item If there is an uncovered element $b$ which appears as a block of size $4$ in one column 
and appears $2$ times in a different column with an $\times$, we can delete the entries containing $b$ and then merge the two columns which contained these entries. Similarly, if there is an uncovered element $b$ which appears as a triple in two different columns, we can delete the entries containing $b$ and then merge the two columns which contained these entries. After doing such a merge, we put the resulting column in the same place as the column which contained a $b$ but did not have a $b$ in its first row (deleting the column which had a $b$ in its first row). Finally, if there is an uncovered $b$ which appears as a block of size $6$ or appears $6$ times in a column  
due to a previous merge (in which case it will either appear as a block of size $4$ and a pair or appear as two triples), we can delete this column. Note that if $T$ starts with an unknown $6$-column then we do not delete this column as this was present from the start rather than being due to a merger of columns.

We repeat these steps until there are no such uncovered elements $b$.
\end{enumerate}
\end{definition}
\begin{remark}
When merging columns, it doesn't really matter which column is deleted as long as we have a consistent procedure for picking this column.
\end{remark}
For our analysis, it is very helpful to split the columns of $T_{red}$ into two parts which we call the shell and core of $T$.
\begin{definition}[Shells and Cores of a Table]
Given a table $T$, we define the shell $\shellT$ of $T$ to be the columns of $T_{red}$ such that at least one of the following is true.
\begin{enumerate}
\item The column contains at least one $\times$.
\item The column contains a triple of a covered element (in which case it must also contain at least one $\times$ or a triple of a second covered element).
\item The column contains a block of size $4$ or $5$ of a covered element.
\end{enumerate}
We define the core $\coreT$ of $T$ to be the columns of $T_{red}$ which are not in the shell.
\end{definition}
\begin{example}\label{ex:smalltabledecomp}
If $T$ is the following table
\begin{center}
\begin{tabular}{|c|c|c|c|c|c|c|}
 \hline
  $\times a$ & $b$ & $d$ & $c$ & $e$ & $f$ & $g$\\
 \hline
  $\times b$ & $d$ & $g$ & $c$ & $a$ & $e$ & $f$\\
 \hline
  $c$ & $\times a$ & $d$ & $f$ & $e$ & $b$ & $g$\\
 \hline
  $c$ & $b$ & $d$ & $a$ & $e$ & $f$ & $g$\\
 \hline
  $c$ & $b$ & $d$ & $f$ & $a$ & $e$ & $g$\\
 \hline
  $c$ & $d$ & $g$ & $a$ & $e$ & $b$ & $f$\\
 \hline
\end{tabular}
\end{center}
where the first, third, and fifth columns are all known $4$-columns, then in the first step of the decomposition, we merge the columns containing $c$ which gives us the following table:
\begin{center}
\begin{tabular}{|c|c|c|c|c|c|c|}
 \hline
  $\times a$ & $b$ & $d$ & $e$ & $f$ & $g$\\
 \hline
  $\times b$ & $d$ & $g$ & $a$ & $e$ & $f$\\
 \hline
  $f$ & $\times a$ & $d$ & $e$ & $b$ & $g$\\
 \hline
  $a$ & $b$ & $d$ & $e$ & $f$ & $g$\\
 \hline
  $f$ & $b$ & $d$ & $a$ & $e$ & $g$\\
 \hline
  $a$ & $d$ & $g$ & $e$ & $b$ & $f$\\
 \hline
\end{tabular}
\end{center}
In the second stage of the decomposition, we merge the columns containing $d$, merge the columns containing $e$, and merge the columns containing $g$. This gives the following table $T_{red}$:
\begin{center}
\begin{tabular}{|c|c|c|c|c|c|}
 \hline
  $\times a$ & $b$ & $f$\\
 \hline
  $\times b$ & $f$ & $a$\\
 \hline
  $f$ & $\times a$ & $b$\\
 \hline
  $a$ & $b$ & $f$\\
 \hline
  $f$ & $b$ & $a$\\
 \hline
  $a$ & $f$ & $b$\\
 \hline
\end{tabular}
\end{center}
$\shellT$ and $\coreT$ are as follows:
\begin{center}
\begin{tabular}{|c|c|}
 \hline
  $\times a$ & $b$\\
 \hline
  $\times b$ & $f$\\
 \hline
  $f$ & $\times a$\\
 \hline
  $a$ & $b$\\
 \hline
  $f$ & $b$\\
 \hline
  $a$ & $f$\\
 \hline
\end{tabular}
\ \ \ \ 
\begin{tabular}{|c|}
 \hline
  $f$\\
 \hline
  $a$\\
 \hline
  $b$\\
 \hline
  $f$\\
 \hline
  $a$\\
 \hline
  $b$\\
 \hline
\end{tabular}
\end{center}
\end{example}
\begin{example}\label{ex:largetabledecomp}
If $T$ is the following table
\begin{center}
\begin{tabular}{|c|c|c|c|c|c|c|c|c|c|c|c|c|c|c|c|c|c|c|}
 \hline
  $j$ & $e$ & $\times a$ & $d$ & $n$ & $r$ & $k$ & $f$ & $i$ & $l$ & $c$ & $o$ & $m$ & $s$ & $p$ & $q$ & $g$ & $h$ & $b$\\
 \hline
  $h$ & $e$ & $\times a$ & $d$ & $o$ & $r$ & $k$ & $f$ & $i$ & $l$ & $c$ & $q$ & $m$ & $p$ & $s$ & $n$ & $g$ & $j$ & $b$\\
 \hline
  $j$ & $e$ & $h$ & $d$ & $o$ & $\times b$ & $a$ & $r$ & $i$ & $l$ & $c$ & $q$ & $f$ & $s$ & $p$ & $n$ & $g$ & $m$ & \cellcolor[HTML]{FD6864}$k$\\
 \hline
  $f$ & $e$ & $j$ & $\times c$ & $n$ & $d$ & $a$ & $r$ & $i$ & $l$ & $b$ & $o$ & $g$ & $p$ & $s$ & $q$ & $m$ & $h$ & \cellcolor[HTML]{67FD9A}$k$\\
 \hline
  $f$ & $\times c$ & $h$ & $l$ & $n$ & $d$ & $a$ & $r$ & $i$ & $e$ & $b$ & $o$ & $g$ & $s$ & $p$ & $q$ & $m$ & $j$ & \cellcolor[HTML]{67FD9A}$k$\\
 \hline
  $h$ & $\times c$ & $j$ & $l$ & $o$ & $d$ & $a$ & $r$ & $i$ & $e$ & $b$ & $q$ & $f$ & $s$ & $p$ & $n$ & $g$ & $m$ & \cellcolor[HTML]{FD6864}$k$\\
 \hline
\end{tabular}
\end{center}
where the column with six $i$ is a known $6$-column, the columns with four $a$, $e$, $g$, $l$, $p$, $r$, and $s$ are known four columns, and the column with four $k$ is an unknown $4$-column with the given pairing then in the first step of the decomposition, we merge the columns containing $d$, we merge the columns containing $e$, and we merge the columns containing $l$ (this must be done after merging the columns containing $e$) which gives us the following table:
\begin{center}
\begin{tabular}{|c|c|c|c|c|c|c|c|c|c|c|c|c|c|c|c|}
 \hline
  $j$ & $\times a$ & $n$ & $r$ & $k$ & $f$ & $i$ & $c$ & $o$ & $m$ & $s$ & $p$ & $q$ & $g$ & $h$ & $b$\\
 \hline
  $h$ & $\times a$ & $o$ & $r$ & $k$ & $f$ & $i$ & $c$ & $q$ & $m$ & $p$ & $s$ & $n$ & $g$ & $j$ & $b$\\
 \hline
  $j$ & $h$ & $o$ & $\times b$ & $a$ & $r$ & $i$ & $c$ & $q$ & $f$ & $s$ & $p$ & $n$ & $g$ & $m$ & \cellcolor[HTML]{FD6864}$k$\\
 \hline
  $f$ & $j$ & $n$ & $\times c$ & $a$ & $r$ & $i$ & $b$ & $o$ & $g$ & $p$ & $s$ & $q$ & $m$ & $h$ & \cellcolor[HTML]{67FD9A}$k$\\
 \hline
  $f$ & $h$ & $n$ & $\times c$ & $a$ & $r$ & $i$ & $b$ & $o$ & $g$ & $s$ & $p$ & $q$ & $m$ & $j$ & \cellcolor[HTML]{67FD9A}$k$\\
 \hline
  $h$ & $j$ & $o$ & $\times c$ & $a$ & $r$ & $i$ & $b$ & $q$ & $f$ & $s$ & $p$ & $n$ & $g$ & $m$ & \cellcolor[HTML]{FD6864}$k$\\
 \hline
\end{tabular}
\end{center}
In the second step, we merge the columns containing $r$ and we merge the columns containing $g$. We also delete the column containing $i$ (as this is a known $6$-column), the columns containing $p$ and $s$ (as this is a cycle of known $4$-columns) and the columns containing $n$, $o$, and $q$ (as this is a cycle of columns with two unmarked triples). This gives the following table $T_{red}$: 
\begin{center}
\begin{tabular}{|c|c|c|c|c|c|c|c|}
 \hline
  $j$ & $\times a$ & $k$ & $f$ & $c$ & \cellcolor[HTML]{FD6864}$m$ & $h$ & $b$\\
 \hline
  $h$ & $\times a$ & $k$ & $f$ & $c$ & \cellcolor[HTML]{FD6864}$m$ & $j$ & $b$\\
 \hline
  $j$ & $h$ & $a$ & $\times b$ & $c$ & $f$ & $m$ & \cellcolor[HTML]{FD6864}$k$\\
 \hline
  $f$ & $j$ & $a$ & $\times c$ & $b$ & \cellcolor[HTML]{67FD9A}$m$ & $h$ & \cellcolor[HTML]{67FD9A}$k$\\
 \hline
  $f$ & $h$ & $a$ & $\times c$ & $b$ & \cellcolor[HTML]{67FD9A}$m$ & $j$ & \cellcolor[HTML]{67FD9A}$k$\\
 \hline
  $h$ & $j$ & $a$ & $\times c$ & $b$ & $f$ & $m$ & \cellcolor[HTML]{FD6864}$k$\\
 \hline
\end{tabular}
\end{center}
$\shellT$ and $\coreT$ are as follows:
\begin{center}
\begin{tabular}{|c|c|c|c|}
 \hline
  $\times a$ & $k$ & $f$ & $c$\\
 \hline
  $\times a$ & $k$ & $f$ & $c$\\
 \hline
  $h$ & $a$ & $\times b$ & $c$\\
 \hline
  $j$ & $a$ & $\times c$ & $b$\\
 \hline
  $h$ & $a$ & $\times c$ & $b$\\
 \hline
  $j$ & $a$ & $\times c$ & $b$\\
 \hline
\end{tabular}
\ \ \ \ 
\begin{tabular}{|c|c|c|c|}
 \hline
  $j$ & \cellcolor[HTML]{FD6864}$m$ & $h$ & $b$\\
 \hline
  $h$ & \cellcolor[HTML]{FD6864}$m$ & $j$ & $b$\\
 \hline
  $j$ & $f$ & $m$ & \cellcolor[HTML]{FD6864}$k$\\
 \hline
  $f$ & \cellcolor[HTML]{67FD9A}$m$ & $h$ & \cellcolor[HTML]{67FD9A}$k$\\
 \hline
  $f$ & \cellcolor[HTML]{67FD9A}$m$ & $j$ & \cellcolor[HTML]{67FD9A}$k$\\
 \hline
  $h$ & $f$ & $m$ & \cellcolor[HTML]{FD6864}$k$\\
 \hline
\end{tabular}
\end{center}
\end{example}
For our analysis, we will enumerate possible shells and cores so it is useful to note the properties which shells and cores must have.
\begin{definition}
We say that $T$ is a partial table if $T$ satisfies the following conditions:
\begin{enumerate}
\item Each row of $T$ contains at most one $\times $ and most one of each element.
\item Each element which is present in $T$ appears a total of $2$, $4$, or $6$ times in $T$ (including the times where it is marked).
\item The columns of $T$ do not contain any singletons. In other words, each column of $T$ consists of marks, pairs, triples, and/or blocks of size $4$, $5$, or $6$.
\end{enumerate}
Whenever a partial table $T$ contains an unknown $4$-column or $6$-column, it must specify how these elements are paired up with each other.

Note that any table with nonzero expected value must be a partial table.
\end{definition}
\begin{definition}[Shells]
We say that a partial table $S$ is a shell if every column of $S$ contains a mark, a triple of a covered element, and/or a block of size $4$ of a covered element.
\end{definition}
\begin{definition}
We say that a partial table $C$ is a core if it only contains pairs (i.e., $C$ does not have any marks, triples, or blocks of size $4$, $5$, or $6$).
\end{definition}
\subsection{Reversing the Shell/Core Decomposition}
We now describe what data is needed in order to recover a table $T$ from $\shellT$ and $\coreT$. For this process, it is convenient to keep track of each component separately and then combine them at the end. In particular, we keep track of an expanded shell, an expanded core, and a floating component consisting of columns which were deleted. 
\begin{definition}
Given a table $T$, we define the expanded shell of $T$ to be the set of columns of $T$ which are present in $\shellT$ or become part of the columns of $\shellT$ through merges. Similarly, we define the expanded core of $T$ to be the set of columns of $T$ which are present in $\coreT$ or become part of the columns of $\coreT$ through merges. We define the floating component of $T$ to be the columns of $T$ which are not part of the expanded shell or expanded core of $T$.
\end{definition}
To reverse the second step of the decomposition, we need the following data:
\begin{enumerate}
\item To obtain the floating component of $T$, we need to know which known $6$-columns, cycles of known $4$-columns, and cycles of pairs of triples were deleted and the locations of these deleted columns in relation to each other.
\item To obtain the expanded shell after the first step of the decomposition, we need to know which paths of known $4$-columns were merged into the columns of $\shellT$ and the locations of the columns which were deleted during these merges. Similarly, to obtain the expanded core, we need to know which paths of known $4$-columns were merged into the columns of $\coreT$ and the locations of the columns which were deleted during these merges.
\end{enumerate}
\begin{remark}
To specify a path of known $4$-columns which are merged into a column of the shell or core without specifying the locations of these columns, we can give the following data:
\begin{enumerate}
\item[1.] The pair of elements which starts in one of these known $4$-columns and ends up as a pair in the resulting column of the shell or core. We call this pair the surviving pair as it is the only part remaining from these columns after the merges.
\item[2.] The final location of the surviving pair. This will be a pair in the shell or core and we consider this pair to be the destination of the path.
\item[3.] The sequence of elements which are deleted if at each step, we take the element which appears as a block of size $4$ in the column with the surviving pair and merge the two columns which contain this element.
\end{enumerate}
As in \cite{BLP23}, we represent this path of known $4$-columns by the sequence of elements which are deleted followed by the element in the surviving pair (we put this element last because it ends up at the destination of the path).

For example, we would represent the merges
\begin{center}
\begin{tabular}{|c|c|c|c|}
 \hline
  $f$ & $d$ & $e$ & $a$\\
 \hline
  $f$ & $d$ & $e$ & $a$\\
 \hline
  $b$ & $f$ & $d$ & $e$\\
 \hline
  $b$ & $f$ & $d$ & $e$\\
 \hline
  $c$ & $f$ & $d$ & $e$\\
 \hline
  $c$ & $f$ & $d$ & $e$\\
 \hline
\end{tabular}
\ \ $\to$ \ \ 
\begin{tabular}{|c|c|}
 \hline
  $a$\\
 \hline
  $a$\\
 \hline
  $b$\\
 \hline
  $b$\\
 \hline
  $c$\\
 \hline
  $c$\\
 \hline
\end{tabular}
\end{center}
by the path $e \to d \to f \to a$. and note the destination of the path (i.e., the final position of the surviving pair of $a$).
\end{remark}
To reverse the first step of the decomposition, we need the following data:
\begin{enumerate}
\item For each column in the shell with at least one mark, we need to know how the column was originally partitioned. As described in Section \ref{sec:rowsplitdefinition}, this partition can be described by a set of splits, each of which is a 4-2 split or a 3-3 split.
\item For each 4-2 split, we need to know about the sequence of uncovered elements which were originally in between the two parts (starting from the part of size $2$) and the locations of the columns which were deleted during the merges.
\item For each 3-3 split, we need to know about the sequence of uncovered elements which were originally in between the two parts (starting from the part of size $3$ which includes the first row) and the locations of the columns which were deleted during the merges.
\end{enumerate}
\begin{example}
Recall that the shell and core for Example \ref{ex:smalltabledecomp} are as follows: 
\begin{center}
\begin{tabular}{|c|c|}
 \hline
  $\times a$ & $b$\\
 \hline
  $\times b$ & $f$\\
 \hline
  $f$ & $\times a$\\
 \hline
  $a$ & $b$\\
 \hline
  $f$ & $b$\\
 \hline
  $a$ & $f$\\
 \hline
\end{tabular}
\ \ \ \ 
\begin{tabular}{|c|}
 \hline
  $f$\\
 \hline
  $a$\\
 \hline
  $b$\\
 \hline
  $f$\\
 \hline
  $a$\\
 \hline
  $b$\\
 \hline
\end{tabular}
\end{center}
For Example \ref{ex:smalltabledecomp}, there is no floating component. To recover the expanded shell after the first stage of the decomposition, we need to know that there was a path $g \to d \to f$ of known $4$-columns leading to the pair of $f$ in the second column of the shell, the known $4$-column with $4$ $d$ was to the right of the second column of the shell, and the known $4$-column with $4$ $g$ was to the right of the known $4$-column with $4$ $d$.

To recover the expanded core of $T$, we need to know that there was a path $e \to a$ of a known $4$-column leading to the pair of $a$ in the core and that the known $4$-column with $4$ $e$ was to the left of the column in the core. 

Using this data, we can deduce that the expanded shell and core after the first step of the decomposition were as follows.
\begin{center}
\begin{tabular}{|c|c|c|c|}
 \hline
  $\times a$ & $b$ & $d$ & $g$\\
 \hline
  $\times b$ & $d$ & $g$ & $f$\\
 \hline
  $f$ & $\times a$ & $d$ & $g$\\
 \hline
  $a$ & $b$ & $d$ & $g$\\
 \hline
  $f$ & $b$ & $d$ & $g$\\
 \hline
  $a$ & $d$ & $g$ & $f$\\
 \hline
\end{tabular}
\ \ \ \ 
\begin{tabular}{|c|c|}
 \hline
  $e$ & $f$\\
 \hline
  $a$ & $e$\\
 \hline
  $e$ & $b$\\
 \hline
  $e$ & $f$\\
 \hline
  $a$ & $e$\\
 \hline
  $e$ & $b$\\
 \hline
\end{tabular}
\end{center}
To reverse the effect of the first step of the decomposition on the expended shell of $T$, we need to know that the first column originally had a 4-2 split between the top two rows and the bottom four rows, that the uncovered element $c$ was originally between these two parts, and that the column with elements $c$, $c$, $f$, $a$, $f$, $a$ which was deleted was between the column with a block of $4$ $d$ and the column with a block of $4$ $g$. With this data, we can deduce that the expanded shell of $T$ is
\begin{center}
\begin{tabular}{|c|c|c|c|c|}
 \hline
  $\times a$ & $b$ & $d$ & $c$ & $g$\\
 \hline
  $\times b$ & $d$ & $g$ & $c$ & $f$\\
 \hline
  $c$ & $\times a$ & $d$ & $f$ & $g$\\
 \hline
  $c$ & $b$ & $d$ & $a$ & $g$\\
 \hline
  $c$ & $b$ & $d$ & $f$ & $g$\\
 \hline
  $c$ & $d$ & $g$ & $a$ & $f$\\
 \hline
\end{tabular}
\end{center}
We can complete the recovery of $T$ by specifying which columns of $T$ are in the shell, which columns of $T$ are in the core, and which columns of $T$ are in the floating component (which is empty in this example).
\end{example}
\begin{example}
Recall that the shell and core for Example \ref{ex:largetabledecomp} are as follows: 
\begin{center}
\begin{tabular}{|c|c|c|c|}
 \hline
  $\times a$ & $k$ & $f$ & $c$\\
 \hline
  $\times a$ & $k$ & $f$ & $c$\\
 \hline
  $h$ & $a$ & $\times b$ & $c$\\
 \hline
  $j$ & $a$ & $\times c$ & $b$\\
 \hline
  $h$ & $a$ & $\times c$ & $b$\\
 \hline
  $j$ & $a$ & $\times c$ & $b$\\
 \hline
\end{tabular}
\ \ \ \ 
\begin{tabular}{|c|c|c|c|}
 \hline
  $j$ & \cellcolor[HTML]{FD6864}$m$ & $h$ & $b$\\
 \hline
  $h$ & \cellcolor[HTML]{FD6864}$m$ & $j$ & $b$\\
 \hline
  $j$ & $f$ & $m$ & \cellcolor[HTML]{FD6864}$k$\\
 \hline
  $f$ & \cellcolor[HTML]{67FD9A}$m$ & $h$ & \cellcolor[HTML]{67FD9A}$k$\\
 \hline
  $f$ & \cellcolor[HTML]{67FD9A}$m$ & $j$ & \cellcolor[HTML]{67FD9A}$k$\\
 \hline
  $h$ & $f$ & $m$ & \cellcolor[HTML]{FD6864}$k$\\
 \hline
\end{tabular}
\end{center}
We first reverse the second step of the decomposition. In order to recover the floating component of $T$, we need the following data:
\begin{enumerate}
\item We need to know about the known $6$-column with a block of six $i$.
\item We need to know about the cycle $s \to p \to s$ of known $4$-columns where the blocks of size $4$ are in rows $1,3,5,6$.
\item We need to know about the cycle $n \to o \to q \to n$ of columns with two pairs of triples where the first triple is in rows $1,4,5$ and the second triple is in rows $2,3,6$.
\item We need to know that the columns of the floating component are ordered so that the order of the elements in the first row is $n,i,o,s,p,q$.
\end{enumerate}
From this data, we can deduce that the floating component is 
\begin{center}
\begin{tabular}{|c|c|c|c|c|c|}
 \hline
  $n$ & $i$ & $o$ & $s$ & $p$ & $q$\\
 \hline
  $o$ & $i$ & $q$ & $p$ & $s$ & $n$\\
 \hline
  $o$ & $i$ & $q$ & $s$ & $p$ & $n$\\
 \hline
  $n$ & $i$ & $o$ & $p$ & $s$ & $q$\\
 \hline
  $n$ & $i$ & $o$ & $s$ & $p$ & $q$\\
 \hline
  $o$ & $i$ & $q$ & $s$ & $p$ & $n$\\
 \hline
\end{tabular}
\end{center}
In order to recover the expanded shell and core after the first step of the decomposition, we need the following data:
\begin{enumerate}
\item There is a path $r \to f$ of known $4$-columns whose destination is the pair of $f$ in the third column of the shell. The column which contained the elements $r,r,\times b,\times c,\times c,\times c$ (after the first stage of the decomposition) was originally between the first and second columns of $\shellT$.
\item There is a path $g \to m$ of known $4$-columns whose destination is the pair of $m$ in rows $4$ and $5$ of the second column of the core. The column which contained a block of size $4$ of $g$ and a pair of $m$ was originally in between the second and third columns of $\coreT$.
\end{enumerate}
Using this data, we can deduce that the expanded shell and core after the first step of the decomposition were as follows:
\begin{center}
\begin{tabular}{|c|c|c|c|c|}
 \hline
  $\times a$ & $r$ & $k$ & $f$ & $c$\\
 \hline
  $\times a$ & $r$ & $k$ & $f$ & $c$\\
 \hline
  $h$ & $\times b$ & $a$ & $r$ & $c$\\
 \hline
  $j$ & $\times c$ & $a$ & $r$ & $b$\\
 \hline
  $h$ & $\times c$ & $a$ & $r$ & $b$\\
 \hline
  $j$ & $\times c$ & $a$ & $r$ & $b$\\
 \hline
\end{tabular}
\ \ \ \ 
\begin{tabular}{|c|c|c|c|c|}
 \hline
  $j$ & $m$ & $g$ & $h$ & $b$\\
 \hline
  $h$ & $m$ & $g$ & $j$ & $b$\\
 \hline
  $j$ & $f$ & $g$ & $m$ & \cellcolor[HTML]{FD6864}$k$\\
 \hline
  $f$ & $g$ & $m$ & $h$ & \cellcolor[HTML]{67FD9A}$k$\\
 \hline
  $f$ & $g$ & $m$ & $j$ & \cellcolor[HTML]{67FD9A}$k$\\
 \hline
  $h$ & $f$ & $g$ & $m$ & \cellcolor[HTML]{FD6864}$k$\\
 \hline
\end{tabular}
\end{center}
To reverse the effect of the first step of the decomposition on the expended shell of $T$, we need the following data:
\begin{enumerate}
\item For the column with elements $r,r,\times b,\times c,\times c,\times c$, we originally had a 4-2 split between rows $1,2,3,4$ and $5,6$ and a 3-3 split between rows $1,2,3$ and $4,5,6$.
\item For the 4-2 split, we originally had the uncovered elements $e$ and $l$ between the part of size $2$ in rows 5 and 6 and the part of size $4$ (which was originally split by the 3-3 split).
\item The column with elements $e,e,e,e,\times c,\times c$ was the first column of the expanded shell of $T$ and the column with elements $l,l,l,l,e,e$ was between the column with elements $f,f,r,r,r,r$ and the column with elements $c,c,c,b,b,b$.
\item For the 3-3 split, we originally had the uncovered element $d$ in between the two parts of size $3$. 
\item The column with elements $d,d,d,\times c,l,l$ was originally in between the column with elements $\times a,\times a,h,j,h,j$ and the column with elements $r,r,\times b,d,d,d$.
\end{enumerate}
With this data, we can deduce that the expanded shell of $T$ is
\begin{center}
\begin{tabular}{|c|c|c|c|c|c|c|c|}
 \hline
  $e$ & $\times a$ & $d$ & $r$ & $k$ & $f$ & $l$ & $c$\\
 \hline
  $e$ & $\times a$ & $d$ & $r$ & $k$ & $f$ & $l$ & $c$\\
 \hline
  $e$ & $h$ & $d$ & $\times b$ & $a$ & $r$ & $l$ & $c$\\
 \hline
  $e$ & $j$ & $\times c$ & $d$ & $a$ & $r$ & $l$ & $b$\\
 \hline
  $\times c$ & $h$ & $l$ & $d$ & $a$ & $r$ & $e$ & $b$\\
 \hline
  $\times c$ & $j$ & $l$ & $d$ & $a$ & $r$ & $e$ & $b$\\
 \hline
\end{tabular}
\end{center}
We can complete the recovery of $T$ by specifying which columns of $T$ are in the shell, which columns of $T$ are in the core, and which columns of $T$ are in the floating component.
\end{example}
\subsection{Sketch for Computing $F_6(t)$}
We now give a high-level sketch for how this decomposition of tables into an expanded shell, an expanded core, and a floating component can be used to compute $F_6(t)$. We need the following definition.  
\begin{definition}
We say that a partial table $C$ is compatible with a shell $S$ if the union of $S$ and $C$ contains six copies of each element which is present in $S$ or $C$, one for each row.

We say that a core $C$ is a minimum compatible core for $S$ if $C$ is compatible with $S$ and $C$ does not contain any elements which are not in $S$.
\end{definition}
As described in more detail in Section \ref{sec:Ftformula}, we can determine the contribution to the exponential generating function $F_6(t)$ from all tables with a given shell $S$ as follows: 
\begin{enumerate}
\item Given a shell $S$, we compute the exponential generating function corresponding to all of the possible expanded shells $S'$ which result in $S$.
\item We compute the net number of minimum compatible cores for $S$ (see Definition \ref{def:mincompatible}) and use this to compute the exponential generating function corresponding to all of the possible expanded cores $C'$ which are compatible with $S$.
\item 
The exponential generating function for the floating component can be computed using the same analysis as \cite{BLP25}.
\item Combining these exponential generating functions gives the contribution to the exponential generating function $F_6(t)$ from all tables $T$ such that $\shellT = S$.
\end{enumerate}
We obtain $F_6(t)$ by summing this contribution over all possible shells $S$.
\section{Exponential Generating Function for the Sixth Moment}\label{sec:Ftformula}
In this section, we implement the high-level plan described above and obtain an exact formula for $F_6(t)$.  
\subsection{Overview of the Formula for $F_6(t)$}
We can find $F_6(t)$ by summing the contributions from all of the possible shells (up to permutations of the rows, columns, and elements). For a given shell $S$ (note that $S$ includes pairings for unknown $4$-columns), we have the following factors:
\begin{enumerate}
\item In order to take permutations of the rows, columns, and elements into account while avoiding double counting, we have a factor of $\frac{6!}{|Aut(S)|}$ where $Aut(S)$ is the set of automorphisms of $S$ (see Definition \ref{def:Sautomorphisms}).
\item For each column $\mathbf{c}$ of $S$, we have a factor $G(\mathbf{c})$ (see Definition \ref{def:columnfactor}) which captures factors coming from the elements of $\mathbf{c}$, paths of known $4$-columns and/or paths of pairs of triples of elements between pieces of $\mathbf{c}$ if $\mathbf{c}$ is split up, and chains of known $4$-columns attached to pairs of elements of $\mathbf{c}$.
\item If the shell has $k$ more elements than columns, letting $t' = \frac{t}{(1 - (\mu_4 - 3)t)^3}$,  letting $N^{(k)}(t) = \frac{d^{k}N(t)}{dt^k}$ be the kth derivative of $N(t)$, and letting $Weight_{Cores}(S)$ be the total weight of the minimal cores compatible with $S$ (see Definition \ref{def:netweight}), we have a factor of  
\[
{t'}^{k}N^{(k)}(t')\cdot\frac{Weight_{Cores}(S)}{\prod_{j=1}^{k}{j(j+2)(j+4)}}
\]
to account for the contribution from the possible cores and chains of known 4-columns attached to the cores.
\item We have a factor of $\frac{O(t)}{N(t')}$ to account for the floating component which consists of known $6$-columns, cycles of known 4-columns, and cycles of columns with two unmarked triples of elements.
\item In order to ensure that we count tables with the correct sign, we have a sign $sign(S)$ for $S$ which we obtain by comparing $S$ to a canonical table for $S$ (see Definition \ref{def:shellcoresign}). 
\end{enumerate}
Thus, the total contribution from shells which are isomorphic to $S$ is 
\[
sign(S)\frac{6!}{|Aut(S)|} \cdot \frac{O(t)}{N(t')} \cdot \left(\prod_{\text{columns } \mathbf{c} \text{ of } S}G(\mathbf{c})\right) \cdot \frac{{t'}^{k}N^{(k)}(t')Weight_{Cores}(S)}{\prod_{j=1}^{k}{j(j+2)(j+4)}}
\]
\subsection{Precise Description of the Formula for Finding $F_6(t)$}
We now give a more precise description of the terms in our formula. 
\subsubsection{Shell automorphisms} 
We start by defining shell automorphisms.
\begin{definition}
We define the following sets of permutations on $6 \times j$ tables $T$.
\begin{enumerate}
\item We define $Sym(Rows) \cong Sym(6)$ to be the group of permutations of the rows.
\item We define $Sym(Columns) \cong Sym(j)$ to be the group of permutations of the columns.
\item We define $Sym(Elements)$ to be the group of permutations of the elements of the table.
\end{enumerate}
Note that $Sym(Rows) \cong Sym(6)$ and $Sym(Columns) \cong Sym(j)$. If the table $T$ contains $j$ elements than $Sym(Elements) \cong Sym(j)$ but we will also consider shells which have more elements than columns.
\end{definition}
\begin{definition}[Shell Automorphisms]\label{def:Sautomorphisms}
Given a shell $S$, we define the automorphism group $Aut(S)$ of $S$ to be the subgroup of $Sym(Rows) \times Sym(Columns) \times Sym(Elements)$ which preserves $S$. We define the orbit $orbit(S)$ of $S$ to be the set of shells which can be obtained from $S$ by applying an element of $Sym(Rows) \times Sym(Columns) \times Sym(Elements)$ to $S$.
\end{definition}
\subsubsection{Row splits of a column $\mathbf{c}$}\label{sec:rowsplitdefinition}
Before describing the column factor $G(\mathbf{c})$, we need to define the set $Splits(\mathbf{c)}$ of possible sets of row splits a column can have.
\begin{definition}[Row splits]\label{def:rowsplits}
Given a column $\mathbf{c}$ of a shell $S$, we define $Splits(\mathbf{c)}$ to be the set of possible sets $RS$ of partitions $(R_1,R_2)$ of the rows of $\mathbf{c}$ which have the following properties:
\begin{enumerate}
\item Either $|R_1| = 4$ and $|R_2| = 2$ (in which case we call $(R_1,R_2)$ a $4$-$2$ split) or $|R_1| = 3$ and $|R_2| = 3$ (in which case we call $(R_1,R_2)$ a $3$-$3$ split).
\item For each 4-2 split $(R_1,R_2) \in RS$, $\mathbf{c}$ has an $\times $ in the rows in $R_2$.
\item For each 3-3 split $(R_1,R_2) \in RS$, one of the following three cases holds for both $R_1$ and $R_2$:
\begin{enumerate}
\item[1.] $\mathbf{c}$ contains an unmarked triple in these rows.
\item[2.] $\mathbf{c}$ contains an $\times $ and an unmarked pair in these rows.
\item[3.] $\mathbf{c}$ contains three $\times $ in these rows.
\end{enumerate}
\item For any distinct pair of partitions $(R_1,R_2),(R'_1,R'_2) \in RS$, either $R'_2 \subseteq R_1$ or $R'_2 \subseteq R_2$. 
\end{enumerate}
\end{definition}
\begin{example}\label{ex:rowsplitexample}
If $\mathbf{c}$ is a column which has marks in rows $1$, $2$, and $3$ and a triple in rows $4$, $5$, and $6$ then there are eight possible sets $RS$ of partitions:
\begin{enumerate}
\item[1.] $RS = \emptyset$.
\item[2.] $RS = \{(\{1,2,3\},\{4,5,6\})\}$.
\item[3.] $RS = \{(\{1,2\},\{3,4,5,6\})\}$, $RS = \{(\{1,3\},\{2,4,5,6\})\}$, and $RS = \{(\{2,3\},\{1,4,5,6\})\}$.
\item[4.] $RS = \{(\{1,2\},\{3,4,5,6\}), (\{1,2,3\},\{4,5,6\})\}$, $RS = \{(\{1,3\},\{2,4,5,6\}), (\{1,2,3\},\{4,5,6\})\}$, and $RS = \{(\{2,3\},\{1,4,5,6\}), (\{1,2,3\},\{4,5,6\})\}$.
\end{enumerate}
\end{example}
\subsubsection{Column factors $G(\mathbf{c})$}\label{sec:shell-column-factors}
We now describe the column factor $G(\mathbf{c})$ for each column $\mathbf{c}$ of the shell. Before giving the definition of $G(\mathbf{c})$, we describe the various factors in $G(\mathbf{c})$.

For each column $\mathbf{c}$ of the shell, we sum over the possible sets of row splits $RS \in Splits(\mathbf{c)}$. Unless the column has five unmarked copies of an element (we handle this case separately)
, we have the following factors:
\begin{enumerate}
\item We have a factor of $t$ from the column itself.
\item For each $\times $ in the column, we have a factor of $m_1$.
\item We have factors from the unmarked elements in a column. If a column has a block of size $4$, 
we have a factor of $(\mu_4 - 3)$. If a column has unmarked triple(s) of elements, we have a factor of $\mu_3$ for each such triple.
\item For each 4-2 row split in $RS$, 
we have a factor of $\frac{(\mu_4 - 3)t}{1 - (\mu_4 - 3)t}$ to account for the fact that there is a path of known 4-columns of length at least $1$ between the split columns.
\item If there is a 3-3 row split in $RS$, 
we have a factor of $\frac{-{\mu_3^2}t}{1 + {\mu_3^2}t}$ to account for the fact that there is a path of pairs of triples of an element of length at least $1$ between the split columns.
\item For each unmarked pair of elements, we have a factor of $\frac{1}{1 - (\mu_4 - 3)t}$ to account for the fact that a chain of known 4-columns may be attached to this pair.
\end{enumerate}
Based on this, we define $G(\mathbf{c})$ as follows.
\begin{definition} \ 
\begin{enumerate}
\item Given an $RS \in Splits(\mathbf{c)}$, we define $Num_{4,2}(RS)$ to be the number of $4$-$2$ splits in $RS$ (which will be $0$, $1$, $2$, or $3$) and we define $Num_{3,3}(RS)$ to be the number of $3$-$3$ splits in $RS$ (which will be $0$ or $1$).
\item We define $1_{\text{known } 4}(\mathbf{c)}$ to be $1$ if $\mathbf{c}$ is a known $4$-column and $0$ otherwise.
\item We define $Num_{triples}(\mathbf{c)}$ to be the number of unmarked triples $\mathbf{c}$ contains.
\item We define $Num_{pairs}(\mathbf{c})$ to be the number of unmarked pairs $\mathbf{c}$ contains.
\end{enumerate}
\end{definition}
\begin{definition}[Column factors]\label{def:columnfactor}
Given a column $\mathbf{c}$ of $S$, unless $\mathbf{c}$ is a column with a block of size $5$ of an element 
and one marked copy of this element, we define
\begin{align*}
G(\mathbf{c}) = &t(m_1)^{\# \text{ of marks in } \mathbf{c}}{(\mu_4 - 3)}^{1_{\text{known } 4}(\mathbf{c)}}{\mu_3}^{Num_{triples}(\mathbf{c)}}\left(\frac{1}{1 - (\mu_4 - 3)t}\right)^{Num_{pairs}(\mathbf{c})}\\
&\sum_{RS \in Splits(\mathbf{c)}}{\left(\frac{(\mu_4 - 3)t}{1 - (\mu_4 - 3)t}\right)^{Num_{4,2}(RS)}\left(\frac{-{\mu_3^2}t}{1 + {\mu_3^2}t}\right)^{Num_{3,3}(RS)}}.
\end{align*}
If $\mathbf{c}$ is a column with block of size $5$ of an element 
and one marked copy of this element, we define $G(\mathbf{c}) = m_1(\mu_5 - 10\mu_3)t$.
\end{definition}
We recall the different types of columns in the shell.
\vspace{-1em}
\bgroup
\renewcommand{\arraystretch}{0.8}
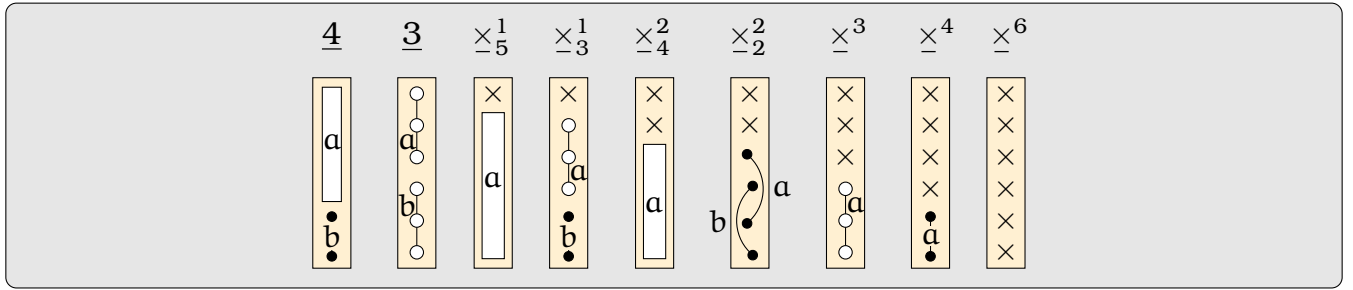
\begin{table}[H]
\centering
\setlength{\tabcolsep}{5.2pt}
\begin{GrayBox}
\vspace{-0.5em}
\centering
\begin{tabular}{ccccccccc}
$\underline{4}$ & $\underline{3}$ & ${\underline{\!\times\!}\,}_5^1$ & ${\underline{\!\times\!}\,}_3^1$ & ${\underline{\!\times\!}\,}_4^2$ & ${\underline{\!\times\!}\,}_2^2$ & ${\underline{\!\times\!}\,}^3$ & $\,{\underline{\!\times\!}\,}^4\!\!$ & $\,{\underline{\!\times\!}\,}^6\!\!\!$ \\[0.5em] 
        \begin{tikzpicture}[baseline = 8.2ex,scale = 0.42]
        \draw[fill=mycolor] (1.2-.6, 0.5) rectangle (1 *1.2+.6, 6.5);
        \pairA[1.2];
        \draw[fill=white] (1*1.2-0.3, 2.6) rectangle (1*1.2+0.3, 6.2);
        \node[vertex] (a) at (1*1.2,4.5) {$a$};
        \draw[pair] (A1.center) to node[fill=mycolor, inner sep=1.0pt] {$b$} (A2.center);
        \end{tikzpicture}

    &
        \begin{tikzpicture}[baseline = 8.2ex,scale = 0.42]
        \draw[fill=mycolor] (1.3-0.6, 0.5) rectangle (1*1.3+0.6, 6.5);
        \pairA[1.3];
        \draw (A1.center) -- (A3.center);
        \draw (A4.center) -- (A6.center);
        \foreach \i in {1,...,6}
            \draw[fill=white, draw=black] (A\i) circle (6pt);
        \node[vertex] (a) at (1.0,4.5) {$a$};
        \node[vertex] (a) at (1.0,2.5) {$b$};
        \end{tikzpicture}
    &
        \begin{tikzpicture}[baseline = 8.2ex,scale = 0.42]
        \draw[fill=mycolor] (1.3-0.6, 0.5) rectangle (1*1.3+0.6, 6.5);
        \pairA[1.3];
        \node at (A6) {$\times$};
        \draw[fill=white] (1*1.2-0.25, 0.8) rectangle (1*1.2+0.45, 5.4);
        \node[vertex] (a) at (1*1.2+0.07,3.3) {$a$};
        \end{tikzpicture}
    &
        \begin{tikzpicture}[baseline = 8.2ex,scale = 0.42]
        \draw[fill=mycolor] (1.3-0.6, 0.5) rectangle (1*1.3+0.6, 6.5);
        \pairA[1.3];
        \node at (A6) {$\times$};
        \draw (A3.center) -- (A5.center);
        \foreach \i in {3,...,5}
            \draw[fill=white, draw=black] (A\i) circle (6pt);
        \node[vertex] (a) at (1.6,3.5) {$a$};
        \draw[pair] (A1.center) to node[fill=mycolor, inner sep=1.0pt] {$b$} (A2.center);
        \end{tikzpicture}
    &
        \begin{tikzpicture}[baseline = 8.2ex,scale = 0.42]
        \draw[fill=mycolor] (1.3-0.6, 0.5) rectangle (1*1.3+0.6, 6.5);
        \pairA[1.3];
        \node at (A6) {$\times$};
        \node at (A5) {$\times$};
        \draw[fill=white] (1*1.2-0.25, 0.8) rectangle (1*1.2+0.45, 4.4);
        \node[vertex] (a) at (1*1.2+0.07,2.5) {$a$};
        \end{tikzpicture}
    &
        \begin{tikzpicture}[baseline = 8.2ex,scale = 0.42]
        \draw[fill=mycolor] (1.3-0.6, 0.5) rectangle (1*1.3+0.6, 6.5);
        \pairA[1.3];
        \node at (A6) {$\times$};
        \node at (A5) {$\times$};
        \draw[pair] (A1.center) to [bend left=45] node[inner sep=1.0pt] {$b$\,\,\,\,\,\,\,} (A3.center);
        \draw[pair] (A2.center) to [bend left=-45] node[inner sep=1.0pt] {\,\,\,\,\,\,\,$a$} (A4.center);
        \end{tikzpicture}
    &
        \begin{tikzpicture}[baseline = 8.2ex,scale = 0.42]
        \draw[fill=mycolor] (1.3-0.6, 0.5) rectangle (1*1.3+0.6, 6.5);
        \pairA[1.3];
        \node at (A6) {$\times$};
        \node at (A5) {$\times$};
        \node at (A4) {$\times$};
        \draw (A1.center) -- (A3.center);
        \foreach \i in {1,...,3}
            \draw[fill=white, draw=black] (A\i) circle (6pt);
        \node[vertex] (a) at (1.6,2.5) {$a$};
        \end{tikzpicture}
    &
        \begin{tikzpicture}[baseline = 8.2ex,scale = 0.42]
        \draw[fill=mycolor] (1.3-0.6, 0.5) rectangle (1*1.3+0.6, 6.5);
        \pairA[1.3];
        \node at (A6) {$\times$};
        \node at (A5) {$\times$};
        \node at (A4) {$\times$};
        \node at (A3) {$\times$};
        \draw[pair] (A1.center) to node[fill=mycolor, inner sep=1.0pt] {$a$} (A2.center);
        \end{tikzpicture}
    &
        \begin{tikzpicture}[baseline = 8.2ex,scale = 0.42]
        \draw[fill=mycolor] (1.3-0.6, 0.5) rectangle (1*1.3+0.6, 6.5);
        \pairA[1.3];
        \node at (A6) {$\times$};
        \node at (A5) {$\times$};
        \node at (A4) {$\times$};
        \node at (A3) {$\times$};
        \node at (A2) {$\times$};
        \node at (A1) {$\times$};
        \end{tikzpicture}
\end{tabular}
\vspace{-0.6em}
\end{GrayBox}
\end{table}
\egroup
\vspace{-1em}
\begin{lemma}\label{lem:columnfactors}
The column factors for the various types of columns are as follows:
\begin{enumerate}
\item ${\underline{\!\times\!}\,}_5^1$: $G(\mathbf{c}) = m_1(\mu_5 - 10\mu_3)t$.
\item ${\underline{\!\times\!}\,}_3^1$: $G(\mathbf{c}) = \frac{m_1{\mu_3}t}{(1-(\mu_4 - 3)t)(1 + {\mu_3^2}t)}$.
\item $\underline{4}$: $G(\mathbf{c}) = \frac{(\mu_4 - 3)t}{1-(\mu_4 - 3)t}$.
\item ${\underline{\!\times\!}\,}_4^2$: $G(\mathbf{c}) = \frac{m_1^2(\mu_4-3)t}{1-(\mu_4 - 3)t}$.
\item ${\underline{\!\times\!}\,}_2^2$: $G(\mathbf{c}) = {m_1^2}t\left(\frac{1 - {\mu_3^2}t + 2{\mu_3^2}(\mu_4-3)t^2}{(1 - (\mu_4-3)t)^{3}(1 + {\mu_3^2}t)}\right)$.
\item $\underline{3}$: $G(\mathbf{c}) = \frac{{\mu_3^2}t}{(1 + {\mu_3^2}t)}$.
\item ${\underline{\!\times\!}\,}^3$: $G(\mathbf{c}) = \frac{{m_1^3}\mu_3(1+2(\mu_4 - 3)t)t}{(1 - (\mu_4-3)t)(1 + {\mu_3^2}t)}$.
\item ${\underline{\!\times\!}\,}^4$: $G(\mathbf{c}) = {m_1^4}t\left(\frac{1+4(\mu_4-3)t -2(\mu_4-3)^{2}t^{2}-3{\mu_3^2}t + 6{\mu_3^2}(\mu_4-3)^{2}t^3}{(1 - (\mu_4-3)t)^{3}(1 + {\mu_3^2}t)}\right)$.
\item ${\underline{\!\times\!}\,}^6$: 
\begin{align*}
G(\mathbf{c}) = {m_1^6}t&\Bigg(\frac{1 + 12(\mu_4-3)t +18(\mu_4-3)^{2}t^{2}-16(\mu_4-3)^{3}t^{3}}{(1 - (\mu_4-3)t)^{3}(1 + {\mu_3^2}t)} \\
&+\frac{-9{\mu_3^2}t -18{\mu_3^2}(\mu_4-3)t^2 +18{\mu_3^2}(\mu_4-3)^{2}t^3 + 24{\mu_3^2}(\mu_4-3)^{3}t^4}{(1 - (\mu_4-3)t)^{3}(1 + {\mu_3^2}t)}\Bigg).
\end{align*}
\end{enumerate}
\end{lemma}
\begin{proof}
For the ${\underline{\!\times\!}\,}_3^1$ column shown, there are two possible sets $RS$ of partitions:
\begin{enumerate}
\item[1.] $RS = \emptyset$.
\item[2.] $RS = \{(\{1,5,6\},\{2,3,4\})\}$. 
\end{enumerate}
Thus,
\[
\sum_{RS \in Splits(\mathbf{c)}}{\left(\frac{(\mu_4 - 3)t}{1 - (\mu_4 - 3)t}\right)^{Num_{4,2}(RS)}\left(\frac{-{\mu_3^2}t}{1 + {\mu_3^2}t}\right)^{Num_{3,3}(RS)}} = 1 + \left(\frac{-{\mu_3^2}t}{1 + {\mu_3^2}t}\right) = \frac{1}{1 + {\mu_3^2}t}
\]
Multiplying this by ${m_1}{\mu_3}t\left(\frac{1}{1-(\mu_4 - 3)}\right)$ gives $G(\mathbf{c}) = \frac{{m_1}{\mu_3}t}{(1-(\mu_4 - 3))(1 + {\mu_3^2}t)}$.

For the $\underline{4}$ column shown, the only possible set $RS$ of partitions is $RS = \emptyset$ so $G(\mathbf{c}) = \frac{(\mu_4 - 3)t}{1-(\mu_4 - 3)t}$.

For the ${\underline{\!\times\!}\,}_4^2$ column shown, there are two possible sets $RS$ of partitions:
\begin{enumerate}
\item[1.] $RS = \emptyset$.
\item[2.] $RS = \{(\{1,2\},\{3,4,5,6\})\}$.
\end{enumerate}
Thus,
\begin{align*}
\sum_{RS \in Splits(\mathbf{c)}}{\left(\frac{(\mu_4 - 3)t}{1 - (\mu_4 - 3)t}\right)^{Num_{4,2}(RS)}\left(\frac{-{\mu_3^2}t}{1 + {\mu_3^2}t}\right)^{Num_{3,3}(RS)}} &= 1 + \left(\frac{(\mu_4-3)t}{1 - (\mu_4-3)t}\right) \\
&= \frac{1}{1 - (\mu_4-3)t}.
\end{align*}
Multiplying this by ${m_1^2}(\mu_4-3)t$ gives
$G(\mathbf{c}) = \frac{{m_1^2}(\mu_4-3)t}{1-(\mu_4 - 3)}$.

For the ${\underline{\!\times\!}\,}_2^2$ column shown, there are four possible sets $RS$ of partitions:
\begin{enumerate}
\item[1.] $RS = \emptyset$.
\item[2.] $RS = \{(\{1,2\},\{3,4,5,6\})\}$.
\item[3.] $RS = \{(\{1,3,5\},\{2,4,6\})\}$.
\item[4.] $RS = \{(\{2,4,6\},\{2,4,6\})\}$.
\end{enumerate}
Thus,
\begin{align*}
&\sum_{RS \in Splits(\mathbf{c)}}{\left(\frac{(\mu_4 - 3)t}{1 - (\mu_4 - 3)t}\right)^{Num_{4,2}(RS)}\left(\frac{-{\mu_3^2}t}{1 + {\mu_3^2}t}\right)^{Num_{3,3}(RS)}} \\
&= 1 + \left(\frac{(\mu_4-3)t}{1 - (\mu_4-3)t}\right) +2\left(\frac{-{\mu_3^2}t}{1 + {\mu_3^2}t}\right)\\
&= \frac{(1 - (\mu_4-3)t)(1 + {\mu_3^2}t) + (\mu_4-3)t(1 + {\mu_3^2}t) - 2{\mu_3^2}t(1 - (\mu_4-3)t)}{(1 - (\mu_4-3)t)(1 + {\mu_3^2}t)}\\
&= \frac{1 - {\mu_3^2}t + 2{\mu_3^2}(\mu_4-3)t^2}{(1 - (\mu_4-3)t)(1 + {\mu_3^2}t)}.
\end{align*}
Multiplying this by ${m_1^2}t\left(\frac{1}{1 - (\mu_4-3)t}\right)^2$ gives
$G(\mathbf{c}) = {m_1^2}t\left(\frac{1 - {\mu_3^2}t + 2{\mu_3^2}(\mu_4-3)t^2}{(1 - (\mu_4-3)t)^{3}(1 + {\mu_3^2}t)}\right)$.

For the $\underline{3}$ column shown, there are two possible sets $RS$ of partitions:
\begin{enumerate}
\item[1.] $RS = \emptyset$.
\item[2.] $RS = \{(\{1,5,6\},\{2,3,4\})\}$. 
\end{enumerate}
Thus,
\[
\sum_{RS \in Splits(\mathbf{c)}}{\left(\frac{(\mu_4 - 3)t}{1 - (\mu_4 - 3)t}\right)^{Num_{4,2}(RS)}\left(\frac{-{\mu_3^2}t}{1 + {\mu_3^2}t}\right)^{Num_{3,3}(RS)}} = 1 + \left(\frac{-{\mu_3^2}t}{1 + {\mu_3^2}t}\right) = \frac{1}{1 + {\mu_3^2}t}.
\]
Multiplying this by ${\mu_3^2}t$ gives $G(\mathbf{c}) = \frac{{\mu_3^2}t}{(1 + {\mu_3^2}t)}$.

As noted in example \ref{ex:rowsplitexample} For the ${\underline{\!\times\!}\,}^3$ column shown, there are eight possible sets $RS$ of partitions:
\begin{enumerate}
\item[1.] $RS = \emptyset$.
\item[2.] $RS = \{(\{1,2,3\},\{4,5,6\})\}$.
\item[3.] $RS = \{(\{1,2\},\{3,4,5,6\})\}$, $RS = \{(\{1,3\},\{2,4,5,6\})\}$, and $RS = \{(\{2,3\},\{1,4,5,6\})\}$.
\item[4.] $RS = \{(\{1,2\},\{3,4,5,6\}), (\{1,2,3\},\{4,5,6\})\}$, $RS = \{(\{1,3\},\{2,4,5,6\}), (\{1,2,3\},\{4,5,6\})\}$, and $RS = \{(\{2,3\},\{1,4,5,6\}), (\{1,2,3\},\{4,5,6\})\}$.
\end{enumerate}
Thus,
\begin{align*}
&\sum_{RS \in Splits(\mathbf{c)}}{\left(\frac{(\mu_4 - 3)t}{1 - (\mu_4 - 3)t}\right)^{Num_{4,2}(RS)}\left(\frac{-{\mu_3^2}t}{1 + {\mu_3^2}t}\right)^{Num_{3,3}(RS)}} \\
&= 1 + \left(\frac{-{\mu_3^2}t}{1 + {\mu_3^2}t}\right) +3\left(\frac{(\mu_4-3)t}{1 - (\mu_4-3)t}\right) + 3\left(\frac{(\mu_4-3)t}{1 - (\mu_4-3)t}\right)\left(\frac{-{\mu_3^2}t}{1 + {\mu_3^2}t}\right)\\
&= \frac{1+2(\mu_4 - 3)t}{(1 - (\mu_4-3)t)(1 + {\mu_3^2}t)}\\
\end{align*}
Multiplying this by ${m_1^3}{\mu_3}t$ gives 
$G(\mathbf{c}) = \frac{{m_1^3}\mu_3(1+2(\mu_4 - 3)t)t}{(1 - (\mu_4-3)t)(1 + {\mu_3^2}t)}$.

For the ${\underline{\!\times\!}\,}^4$ column shown, there are $26$ possible sets $RS$ of partitions:
\begin{enumerate}
\item[1.] $RS = \emptyset$.
\item[2.] There are $4$ different ways for $RS$ to contain a single 3-3 split.
\item[3.] There are $6$ different ways for $RS$ to contain a single 4-2 split.
\item[4.] There are $12$ different ways for $RS$ to contain a 4-2 split and a 3-3 split. For example, we can have 
$RS = \{(\{1,2\},\{3,4,5,6\}), (\{1,2,3\},\{4,5,6\})\}$.
\item[5.] There are $3$ different ways for $RS$ to contain two 4-2 splits. For example, we can have 
$RS = \{(\{1,2\},\{3,4,5,6\}), (\{3,4\},\{1,2,5,6\})\}$.
\end{enumerate}
Thus,
\begin{align*}
&\sum_{RS \in Splits(\mathbf{c)}}{\left(\frac{(\mu_4 - 3)t}{1 - (\mu_4 - 3)t}\right)^{Num_{4,2}(RS)}\left(\frac{-{\mu_3^2}t}{1 + {\mu_3^2}t}\right)^{Num_{3,3}(RS)}} \\
&= 1 + 4\left(\frac{-{\mu_3^2}t}{1 + {\mu_3^2}t}\right) +6\left(\frac{(\mu_4-3)t}{1 - (\mu_4-3)t}\right) + 12\left(\frac{(\mu_4-3)t}{1 - (\mu_4-3)t}\right)\left(\frac{-{\mu_3^2}t}{1 + {\mu_3^2}t}\right) + 3\left(\frac{(\mu_4-3)t}{1 - (\mu_4-3)t}\right)^2\\
&= \frac{(1 - (\mu_4-3)t)^{2}(1 + {\mu_3^2}t)-4(1 - (\mu_4-3)t)^{2}{\mu_3^2}t + 6(\mu_4-3)t(1 - (\mu_4-3)t)(1 + {\mu_3^2}t)}{(1 - (\mu_4-3)t)^{2}(1 + {\mu_3^2}t)} \\
&+\frac{-12(\mu_4-3)(1 - (\mu_4-3)t){\mu_3^2}t^2 + 3(\mu_4-3)^{2}(1 + {\mu_3^2}t)t^{2}}{(1 - (\mu_4-3)t)^{2}(1 + {\mu_3^2}t)} \\
&=\frac{1+4(\mu_4-3)t -2(\mu_4-3)^{2}t^{2}-3{\mu_3^2}t + 6{\mu_3^2}(\mu_4-3)^{2}t^3}{(1 - (\mu_4-3)t)^{2}(1 + {\mu_3^2}t)}
\end{align*}
Multiplying this by $\frac{{m_1^4}t}{1 - (\mu_4-3)t}$ gives 
\[
G(\mathbf{c}) = {m_1^4}t\left(\frac{1+4(\mu_4-3)t -2(\mu_4-3)^{2}t^{2}-3{\mu_3^2}t + 6{\mu_3^2}(\mu_4-3)^{2}t^3}{(1 - (\mu_4-3)t)^{3}(1 + {\mu_3^2}t)}\right)
\]

Finally, for an ${\underline{\!\times\!}\,}^6$ column, there are $236$ possible sets $RS$ of partitions:
\begin{enumerate}
\item[1.] $RS = \emptyset$.
\item[2.] There are $10$ different ways for $RS$ to contain a single 3-3 split.
\item[3.] There are $15$ different ways for $RS$ to contain a single 4-2 split.
\item[4.] There are $60$ different ways for $RS$ to contain a 4-2 split and a 3-3 split. For example, we can have 
$RS = \{(\{1,2\},\{3,4,5,6\}), (\{1,2,3\},\{4,5,6\})\}$.
\item[5.] There are $45$ different ways for $RS$ to contain two 4-2 splits. For example, we can have 
$RS = \{(\{1,2\},\{3,4,5,6\}), (\{3,4\},\{1,2,5,6\})\}$.
\item[4.] There are $90$ different ways for $RS$ to contain two 4-2 splits and a 3-3 split. For example, we can have 
$RS = \{(\{1,2\},\{3,4,5,6\}), (\{1,2,3\},\{4,5,6\}), (\{1,2,3,4\},\{5,6\})\}$.
\item[5.] There are $15$ different ways for $RS$ to contain three 4-2 splits. For example, we can have 
$RS = \{(\{1,2\},\{3,4,5,6\}), (\{3,4\},\{1,2,5,6\}), (\{5,6\},\{1,2,3,4\})\}$.
\end{enumerate}
Thus,
\begin{align*}
&\sum_{RS \in Splits(\mathbf{c)}}{\left(\frac{(\mu_4 - 3)t}{1 - (\mu_4 - 3)t}\right)^{Num_{4,2}(RS)}\left(\frac{-{\mu_3^2}t}{1 + {\mu_3^2}t}\right)^{Num_{3,3}(RS)}} \\
&= 1 + 10\left(\frac{-{\mu_3^2}t}{1 + {\mu_3^2}t}\right) +15\left(\frac{(\mu_4-3)t}{1 - (\mu_4-3)t}\right) + 60\left(\frac{(\mu_4-3)t}{1 - (\mu_4-3)t}\right)\left(\frac{-{\mu_3^2}t}{1 + {\mu_3^2}t}\right) \\
&+45\left(\frac{(\mu_4-3)t}{1 - (\mu_4-3)t}\right)^2 + 90\left(\frac{(\mu_4-3)t}{1 - (\mu_4-3)t}\right)^2\left(\frac{-{\mu_3^2}t}{1 + {\mu_3^2}t}\right) + 15\left(\frac{(\mu_4-3)t}{1 - (\mu_4-3)t}\right)^3\\
&= \frac{(1 - (\mu_4-3)t)^{3}(1 + {\mu_3^2}t)-10(1 - (\mu_4-3)t)^{3}{\mu_3^2}t + 15(\mu_4-3)t(1 - (\mu_4-3)t)^{2}(1 + {\mu_3^2}t)}{(1 - (\mu_4-3)t)^{3}(1 + {\mu_3^2}t)} \\
&+\frac{-60(\mu_4-3)(1 - (\mu_4-3)t)^2{\mu_3^2}t^2 + 45(\mu_4-3)^{2}(1 - (\mu_4-3)t)(1 + {\mu_3^2}t)t^{2}}{(1 - (\mu_4-3)t)^{3}(1 + {\mu_3^2}t)} \\
&+\frac{-90(\mu_4-3)^2(1 - (\mu_4-3)t){\mu_3^2}t^3 + 15(\mu_4-3)^{3}(1 + {\mu_3^2}t)t^{3}}{(1 - (\mu_4-3)t)^{3}(1 + {\mu_3^2}t)} \\
&=\frac{1 + 12(\mu_4-3)t +18(\mu_4-3)^{2}t^{2}-16(\mu_4-3)^{3}t^{3}}{(1 - (\mu_4-3)t)^{3}(1 + {\mu_3^2}t)} \\
&+\frac{-9{\mu_3^2}t -18{\mu_3^2}(\mu_4-3)t^2 +18{\mu_3^2}(\mu_4-3)^{2}t^3 + 24{\mu_3^2}(\mu_4-3)^{3}t^4}{(1 - (\mu_4-3)t)^{3}(1 + {\mu_3^2}t)}
\end{align*}
Multiplying this by ${m_1^6}t$ gives the given formula for $G(\mathbf{c})$.
\end{proof}
\subsubsection{Counting Minimal Compatible Cores for a Shell $S$}
\begin{definition}\label{def:mincompatible}
Recall that a core $C$ is a minimal compatible core for a shell $S$ if 
\begin{enumerate}
\item $C$ does not contain any elements which are not in $S$.
\item If an element is contained in rows $R$ of $S$ then it is contained in rows $[6] \setminus R$ of $C$. In other words, if we combine $S$ and $C$ into a table $T$ then each element which appears in $S$ and/or $C$ appears exactly once in each row of $T$ (including appearances where it is marked).
\end{enumerate}
Given a shell $S$, we define $MinCores(S)$ to be the set of minimal compatible cores of $S$.
\end{definition}
\begin{example}\label{ex:mincoresfirstexample}
If $S$ is the shell
\begin{center}
\begin{tabular}{|c|}
 \hline
  $\times a$\\
 \hline
  $\times a$\\
 \hline
  $b$\\
 \hline
  $b$\\
 \hline
  $c$\\
 \hline
  $c$\\
 \hline
\end{tabular}
\end{center}
then the minimum compatible cores for $S$ can be obtained by starting with the tables shown below
\begin{center}
\begin{tabular}{|c|c|}
 \hline
  $b$ & $c$\\
 \hline
  $b$ & $c$\\
 \hline
  $a$ & $c$\\
 \hline
  $a$ & $c$\\
 \hline
  $a$ & $b$\\
 \hline
  $a$ & $b$\\
 \hline
\end{tabular}
\ \ \ \ 
\begin{tabular}{|c|c|}
 \hline
  $b$ & $c$\\
 \hline
  $b$ & $c$\\
 \hline
  $a$ & $c$\\
 \hline
  $a$ & $c$\\
 \hline
  $b$ & $a$\\
 \hline
  $b$ & $a$\\
 \hline
\end{tabular}
\ \ \ \
\begin{tabular}{|c|c|}
 \hline
  $b$ & $c$\\
 \hline
  $b$ & $c$\\
 \hline
  $c$ & $a$\\
 \hline
  $c$ & $a$\\
 \hline
  $a$ & $b$\\
 \hline
  $a$ & $b$\\
 \hline
\end{tabular}
\ \ \ \ 
 \begin{tabular}{|c|c|}
 \hline
  $b$ & $c$\\
 \hline
  $b$ & $c$\\
 \hline
  $c$ & $a$\\
 \hline
  $c$ & $a$\\
 \hline
  $b$ & $a$\\
 \hline
  $b$ & $a$\\
 \hline
\end{tabular}
\ \ \ \ 
\begin{tabular}{|c|c|}
 \hline
  $b$ & $c$\\
 \hline
  $c$ & $b$\\
 \hline
  $a$ & $c$\\
 \hline
  $c$ & $a$\\
 \hline
  $a$ & $b$\\
 \hline
  $b$ & $a$\\
 \hline
\end{tabular}
\ \ \ \ 
\begin{tabular}{|c|c|}
 \hline
  $b$ & $c$\\
 \hline
  $c$ & $b$\\
 \hline
  $a$ & $c$\\
 \hline
  $c$ & $a$\\
 \hline
  $b$ & $a$\\
 \hline
  $a$ & $b$\\
 \hline
\end{tabular}
\ \ \ \
\begin{tabular}{|c|c|}
 \hline
  $b$ & $c$\\
 \hline
  $c$ & $b$\\
 \hline
  $c$ & $a$\\
 \hline
  $a$ & $c$\\
 \hline
  $a$ & $b$\\
 \hline
  $b$ & $a$\\
 \hline
\end{tabular}
\ \ \ \ 
 \begin{tabular}{|c|c|}
 \hline
  $b$ & $c$\\
 \hline
  $c$ & $b$\\
 \hline
  $c$ & $a$\\
 \hline
  $a$ & $c$\\
 \hline
  $b$ & $a$\\
 \hline
  $a$ & $b$\\
 \hline
\end{tabular}
\end{center}
and making the following modifications:
\begin{enumerate}
\item For the tables which have two $4$-columns, we need to specify the pairings for these columns to obtain a core $C$.
\item For each of these cores, we can swap the columns of the table to obtain a different core.
\end{enumerate}
Thus, there are a total of $2(9+9+1+9+1+1+1+1) = 64$ minimum cores for $S$. However, we have to be careful as some of these minimum cores have negative sign. 
\end{example}
In order to have the right sign for our tables, we make the following definitions
\begin{definition}[Canonical tables]\label{def:canonical-comparison-table}
Given a shell $S$, we construct a canonical table $canon(S)$ for $S$ as follows.
\begin{enumerate}
\item We assign a label to the elements of $S$ based on the set of rows $R$ they are contained in. If $|R| = 2$ then we assign the label $R$. If $|R| = 4$ then we assign the label $[6] \setminus R$. If $|R| = 6$ then we assign the label $\emptyset$. We then assign an order to the elements so that the labels of the elements are in lexicographic order. Note that while the order may not be unique, the sign of the resulting table is fixed.
\item We sort the rows so that the elements are in ascending order according to the order they were given, not their order in $[n]$. Note that this makes the canonical table invariant under permutations of $[n]$.
\end{enumerate}
Given a core $C$, we construct a canonical table $canon(C)$ using the same procedure.
\end{definition}
\begin{definition}[Signs of shells and cores]\label{def:shellcoresign}
Given a shell $S$, we define $sign(S)$ to be $1$ if it takes an even number of swaps (where each swap takes two elements in one row and swaps them) to transform $S$ into $canon(S)$ and $-1$ otherwise. Similarly, given a core $C$, we define $sign(C)$ to be $1$ if it takes an even number of swaps (where each swap takes two elements in one row and swaps them) to transform $C$ into $canon(C)$ and $-1$ otherwise.
\end{definition}
\begin{example}
If $C$ is the core 
\begin{center}
 \begin{tabular}{|c|c|c|}
 \hline
  $b$ & $d$\\
 \hline
  $b$ & $c$\\
 \hline
  $e$ & $d$\\
 \hline
  $e$ & $a$\\
 \hline
  $f$ & $a$\\
 \hline
  $f$ & $c$\\
 \hline
\end{tabular}
\end{center}
then the order of the elements would be $b,d,c,e,a,f$ as $b$ appears in rows $1,2$, $d$ appears in rows $1,3$, $c$ appears in rows $2,6$, $e$ appears in rows $3,4$, $a$ appears in rows $4,5$, and $f$ appears in rows $5,6$. Thus, $canon(C)$ is 
\begin{center}
 \begin{tabular}{|c|c|c|}
 \hline
  $b$ & $d$\\
 \hline
  $b$ & $c$\\
 \hline
  $d$ & $e$\\
 \hline
  $e$ & $a$\\
 \hline
  $a$ & $f$\\
 \hline
  $c$ & $f$\\
 \hline
\end{tabular}
\end{center}
Since it takes $3$ swaps to go from $C$ to $canon(C)$, $sign(C) = -1$.
\end{example}
\begin{remark}
While $canon(C)$ is not a valid partial table as $a$, $c$, $d$, and $e$ don't appear as pairs, this is okay because we are only using $canon(C)$ as a tool to determine the signs of our tables.  
\end{remark}
\begin{lemma}
If $C$ is a minimum compatible core for $S$ then the table formed by merging $canon(S)$ and $canon(C)$ has positive sign.
\end{lemma}
\begin{proof}
Note: This proof uses the fragments which are described later on in the implementation.

Consider the fragments of $S$ and $C$. $canon(S)$ orders the fragments of $S$ by the lexicographic order of their supports. Similarly, $canon(C)$ orders the fragments of $C$ (which are the same as the fragments of $S$ but with opposite polarity) by the lexicographic order of their supports.

Observe that we can reach a table with positive sign by swapping the order of the fragments of $S$ and $C$ so that each fragment of $S$ is adjacent to the corresponding fragment of $C$ (i.e., the fragment with the same element and opposite polarity). We can do this by going through the fragments $F$ of $S$ in reverse order and swapping each fragment $F$ with all of the fragments in $C$ which correspond to earlier fragments of $S$.

When we swap the order of two fragments, this gives a factor of $(-1)$ if the row supports of the fragments have an odd number of elements in common (note that the row support is equal to the support for $-$ fragments and the complement of the support for $+$ fragments). Otherwise, this swap leaves the sign unchanged. Thus the net sign for these swaps is the same as sign for swapping each pair of fragments in $S$ an odd number of times. This implies that the overall sign when we merge $canon(C)$ and $canon(S)$ is 
\[
(-1)^{\sum_{i=1}^{6}{\binom{\text{number of times } i \text{ appears in the row support of a fragment } F \text{ of S}}{2}}} \\
= (-1)^{6\binom{\# \text{ of columns of } S - \# \text{ of elements appearing } 6 \text{ times in } S}{2}} = 1.
\]
as for all $i \in [6]$, the number of fragments $F$ of $S$ such that $i$ is in the row support of $S$ is equal to the number of columns of $S$ minus the number of elements which appear $6$ times in $S$.

Note: The following proposition is nice but should appear later on.
\begin{proposition}
For all shells $S$, 
\[
\sum_{\text{fragments } F \text{ with } + \text{ polarity}}{1_{i \in Support(F)}} - \sum_{\text{fragments } F \text{ with } - \text{ polarity}}{1_{i \in Support(F)}}
\]
is the same for all row indices $i \in [6]$.
\end{proposition}
\begin{proof}
Observe that the total number of times an element appears in row $i$ is equal to 
\[
\sum_{\text{fragments } F \text{ with } - \text{ polarity}}{1_{i \in Support(F)}} + \sum_{\text{fragments } F \text{ with } + \text{ polarity}}{(1 - 1_{i \in Support(F)})}
\]
Since this is the same for all $i \in [6]$, 
\[
\sum_{\text{fragments } F \text{ with } + \text{ polarity}}{1_{i \in Support(F)}} - \sum_{\text{fragments } F \text{ with } - \text{ polarity}}{1_{i \in Support(F)}}
\]
must be the same for all row indices $i \in [6]$.
\end{proof}
\end{proof}
\begin{corollary}\label{cor:shell-core-relative-sign}
If $C$ is a minimum compatible core for $S$ then the table formed by merging $S$ and $C$ has sign $sign(S)sign(C)$.
\end{corollary}
\begin{definition}\label{def:netweight}
Given a shell $S$, we define $Weight_{Cores}(S) = \sum_{C \in MinCores(S)}{sign(C)}$.
\end{definition}
\begin{example}
If $S$ is the shell
\begin{center}
\begin{tabular}{|c|}
 \hline
  $\times a$\\
 \hline
  $\times a$\\
 \hline
  $b$\\
 \hline
  $b$\\
 \hline
  $c$\\
 \hline
  $c$\\
 \hline
\end{tabular}
\end{center}
then for each of the minimum compatible cores $C$ for $S$, $canon(C)$ is 
\begin{center}
\begin{tabular}{|c|c|}
 \hline
  $b$ & $c$\\
 \hline
  $b$ & $c$\\
 \hline
  $a$ & $c$\\
 \hline
  $a$ & $c$\\
 \hline
  $a$ & $b$\\
 \hline
  $a$ & $b$\\
 \hline
\end{tabular}
\end{center}
To see this, observe that the order of the elements is $a,b,c$ as $a$ is missing from rows $1,2$, $b$ is missing from rows $3,4$, and $c$ is missing from rows $5,6$.

For the tables shown in Example \ref{ex:mincoresfirstexample}, the first four tables give minimum compatible cores with positive sign while the last four tables give minimum compatible cores with negative sign. Thus, 
$Weight_{Cores}(S) = 2(9+9+1+9-1-1-1-1) = 48$.
\end{example}
\subsection{The generating function $F_6(t)$}
We are now ready to state and prove our formula for computing $F_6(t)$.
\begin{definition}\label{def:shell-excess}
For each shell $S$, we define $k_S$ to be the number of elements in $S$ minus the number of columns of $S$.
\end{definition}
\begin{definition}
We define $\mathcal{S}$ to be the set of possible shells up to permutations of the rows, columns, and elements. 
\end{definition}
\begin{theorem}[Formula for $F_6(t)$]\label{thm:Ftformula}
Recall that $N(t) = \frac{1}{48}\sum_{n=0}^\infty (n+1)(n+2)(n+4)! \,t^n$ (see \ref{Eq:N}) and $O(t) = (1+\mu_3^2t)^{10}\,\frac{e^{t (\mu_6-15\mu_4-10\mu_3^2 + 30)}}{\left(1+3t-\mu_4t\right)^{15}} N\left(\frac{t}{\left(1+3t-\mu_4t\right)^{3}}\right)$ (see \ref{Eq:O}). Letting $t' = \frac{t}{(1 - (\mu_4 - 3)t)^3}$,
\[
F_6(t) = \sum_{S \in \mathcal{S}}{sign(S)\frac{6!}{|Aut(S)|} \cdot \frac{O(t)}{N(t')} \cdot \left(\prod_{\text{columns } \mathbf{c} \text{ of } S}G(\mathbf{c})\right) \cdot \frac{{t'}^{k_S}N^{(k_S)}(t')Weight_{Cores}(S)}{\prod_{j=1}^{k_S}{j(j+2)(j+4)}}}
\]
\end{theorem}
\begin{remark}\label{rem:shell-core-row-parity}
Note that $sign(S)$ may be affected by permutations of the rows. The reason that this is not an issue is that if $sign(S)$ is flipped by a permutation of the rows, for each minimal compatible core $C$ for $S$, $sign(C)$ will also be flipped by the permutation of the rows.
\end{remark}
\begin{proof}
We show this by handling the expanded shell, the expanded core, and the floating cycles of known $4$-columns and/or columns with two unmarked triples separately. In particular, we show the following:
\begin{enumerate}
\item Lemma \ref{lem:shellgeneratingfunction} shows that if a shell $S$ has $a$ columns and $a+k_S$ elements then taking $F_{S}(t) = sign(S)\prod_{\text{columns } \mathbf{c} \text{ of } S}G(\mathbf{c})$, if we write $F_{S}(t) = \sum_{x=0}^{\infty}{{s_x}t^x}$ then the total weight of the partial tables $S'$ with $x$ columns such that 
\begin{enumerate}
\item[1.] The shell of $S'$ is $S$, the core of $S'$ is empty, and $S'$ does not have a floating component. In other words, $S'$ is a possible expanded shell for $S$.
\item[2.] The $x-a$ elements of $S'$ which are not in $S$ are labeled only by their relative order. In particular, instead of giving the final labels of the elements which are in $S'$ but not $S$, we only specify how these elements compare with each other and we do not specify how these elements compare to the elements in $S$.
\end{enumerate}
is $\frac{x!(x-a)!}{a!}{s_x}$.
\item Corollary \ref{cor:coregeneratingfunction} shows that for each shell $S$, if we take $C_{S}(t) = \frac{{t'}^{k_S}N^{(k_S)}(t')Weight_{Cores}(S)}{\prod_{j=1}^{k_S}{j(j+2)(j+4)}}$ and write $C_{S}(t) = \sum_{y=0}^{\infty}{{c_y}t^y}$ then the total weight of the partial tables $C'$ with $y$ columns such that
\begin{enumerate}
\item $C'$ is compatible with $S$, $shell(C')$ is empty, and $C'$ does not have a floating component. In other words, $C'$ is a possible expanded core for a table $T$ such that $\shellT = S$.
\item The $y-k_S$ elements of $C'$ which are not in $S$ are labeled only by their relative order.
\end{enumerate}
is $y!(y-k_S)!{c_y}$.
\item As noted in Section \ref{sec:disjointcomponentfactorisation}, if we take $Fl(t) = \frac{O(t)}{N(t')}$ and write $Fl(t) = \sum_{z=0}^{\infty}{fl_{z}t^z}$ then the total weight of tables with $z$ columns consisting only of known $\underline{6}$-columns, cycles of known
$\underline{4}$-columns, and cycles of columns with two unmarked triples where the elements are labeled only by their relative order is $(z!)^2{fl_z}$.
\end{enumerate}
Let $n = x+y+z$. We consider a shell $S_0 \in \mathcal{S}$ and compute the total weight of the tables $T$ such that $S = \shellT$ is isomorphic to $S_0$, the expanded shell of $T$ has $x$ columns, the expanded core of $T$ has $y$ columns, and the floating component of $T$ has $z$ columns. We can choose such a table $T$ as follows:
\begin{enumerate}
\item Letting $a$ be the number of columns of $S_0$, there are $\frac{6!(k_S+a)!a!}{Aut(S_0)}$ different shells $S$ which are isomorphic to $S_0$ (where elements are labeled by their relative order) as there are $6!$ permutations of the rows, $(k_S+a)!$ permutations of the elements, and $a!$ permutations of the columns but we need to divide by $Aut(S_0)$ to avoid double counting.
\item We choose an expanded shell $S'$ with $x$ columns, an expanded core $C'$ compatible with $S$, and a floating component with $z$ columns satisfying the conditions described above. The total weight of the choices for $S'$ is $\frac{x!(x-a)!}{a!}{s_x}$, the total weight of the choices for $C'$ is $y!(y-k_S)!{c_y}$, and the total weight of the choices for the floating component is $(z!)^2{fl_z}$.
\item There are $\binom{n}{k_S+a} = \frac{n!}{(k_S+a)!(n-(k_S+a))!}$ ways to choose the actual elements of the shell $S$ (since we have specified the relative order of the elements in $S$ and their positions, specifying the $k_S+a$ elements which appear in $S$ uniquely determines $S$).
\item There are $\frac{(n-(k_S+a))!}{(x_S-a)!(y-k_S)!z!}$ ways to order the elements which are not in $S$ as we have already chosen the relative order for the $x-a$ elements in $S'$ which are not in $S$, the relative order for the $y-k_S$ elements of $C'$ which are not in $S$, and the relative order for the $z$ elements in the floating component.
\item There are $\frac{n!}{x!y!z!}$ ways to complete the description of $T$ by choosing which $x$ columns of $T$ are in $S'$ and which $y$ columns of $T$ are in $S'$ (the remaining $z$ columns are the floating component).
\end{enumerate}
Combining these factors, the total weight of all such tables $T$ is 
\begin{align*}
&\frac{6!(k_S+a)!a!}{Aut(S_0)} \cdot \frac{x!(x-a)!}{a!}{s_x} \cdot y!(y-k_S)!{c_y} \cdot (z!)^2{fl_z} \\
&\cdot \frac{n!}{(k_S+a)!(n-(k_S+a))!} \cdot \frac{(n-(k_S+a))!}{(x-a)!(y-k_S)!z!} \cdot \frac{n!}{x!y!z!}\\
&=\frac{6!}{Aut(S_0)}(n!)^2{s_x}{c_y}{fl_z}.
\end{align*}
Thus, for each $S \in \mathcal{S}$, the contribution to $F_6(t)$ from tables $T$ whose shell is isomorphic to $S$ is 
$\frac{6!}{Aut(S)}\sum_{x=0}^{\infty}{\sum_{y=0}^{\infty}{\sum_{z=0}^{\infty}{{s_X}{c_y}{fl_z}t^{x+y+z}}}} = \frac{6!}{Aut(S)}F_S(t)C_S(t)Fl(t)$. The result follows by summing over all $S \in \mathcal{S}$.
\end{proof}
\subsection{Generating functions for the expanded shells and cores}
We now complete the proof of Theorem \ref{thm:Ftformula} by computing the exponential generating functions corresponding to the expanded shell and the expanded core.
\begin{lemma}\label{lem:shellgeneratingfunction}
If a shell $S$ has $a$ columns and $a+k$ elements then taking 
\[
F_{S}(t) = sign(S)\prod_{\text{columns } \mathbf{c} \text{ of } \shellT}G(\mathbf{c}),
\]
if we write $F_{S}(t) = \sum_{x=0}^{\infty}{{s_x}t^x}$, the total weight of the expanded tables $S'$ with $x$ columns such that 
\begin{enumerate}
\item[1.] The shell of $S'$ is $S$, the core of $S'$ is empty, and $S'$ does not have a floating component. In other words, $S'$ is a possible expanded shell for $S$.
\item[2.] The $x-a$ elements of $S'$ which are not in $\shellT$ are labeled by their relative order
\end{enumerate}
is $\frac{x!(x-a)!}{a!}{s_x}$.
\end{lemma}
\begin{proof}
The key idea is that whenever we add a column due to a $4$-$2$, split, a $3$-$3$ split, or attaching a chain to an unmarked pair, if we already have $x'$ columns then there are $x'+1$ places where the new column can fit in and there are $x'-a+1$ choices for where the new element is inserted in the relative order of the added elements. This gives a factor of $\left(\prod_{j=a+1}^{x}{j(j-a)}\right) = \frac{x!(x-a)!}{a!}$.

We now make the following observations:
\begin{enumerate}
\item For each $4$-$2$ row split, we add $l$ columns/elements for some $l \geq 1$. If we add $l$ elements and columns then this gives a factor of $((\mu_4 - 3)t)^{l}$. Summing over the possible $l$ gives a factor of $\sum_{l=1}^{\infty}{((\mu_4 - 3)t)^{l}} = \frac{(\mu_4 - 3)t}{1 - (\mu_4 - 3)t}$.
\item For each $3$-$3$ row split, we add $l$ columns/elements for some $l \geq 1$. If we add $l$ elements and columns then this gives a factor of $(-{\mu_3^2}t)^{l}$ (we have a minus sign as it takes 3 swaps to undo each step of the split). Summing over the possible $l$ gives a factor of $\sum_{l=1}^{\infty}{(-{\mu_3^2}t)^{l}} = \frac{-{\mu_3^2}t}{1 + {\mu_3^2}t}$.
\item For each unmarked pair of elements, we attach a chain of known $4$-columns of length $l$ for some $l \geq 0$. If we attach a chain of length $l$ then this gives a factor of $((\mu_4 - 3)t)^{l}$. Summing over the possible $l$ gives a factor of $\sum_{l=0}^{\infty}{((\mu_4 - 3)t)^{l}} = \frac{1}{1 - (\mu_4 - 3)t}$.
\end{enumerate}
The result follow by combining these factors with $sign(S)$ and the existing factors of $t$, $m_1$, $\mu_3$, $(\mu_4-3)$ from the columns, marks, triples, and blocks of size $4$ in $S$.
\end{proof}
\begin{lemma}\label{lem:coretotalweight}
For each shell $S$, for all $y' \geq k_S$, the total weight of the cores $C$ with $y'$ columns such that 
\begin{enumerate}
\item $C$ is compatible with $S$.
\item The $y'-k_S$ elements of $C$ which are not in $S$ are labeled only by their relative order.
\end{enumerate}
is $\left(\prod_{j=k+1}^{y'}{j(j+2)(j+4)}\right)Weight_{Cores}(S)$
\end{lemma}
\begin{proof}
Let $W_S(y')$ be the total weight of the cores $C$ with $y'$ columns such $C$ is compatible with $S$ and the $y'-k_S$ elements of $C$ which are not in $S$ are labeled only by their relative order. To prove this result, it is sufficient to show that for all $y' > k_S$, $W_S(y') = y'(y'+2)(y'+4)W_S(y'-1)$. We prove this using the techniques in Appendix A of \cite{BLP23}.

Let $n' = y'-k_S$ and consider the cores $C$ with $y'$ columns which are compatible with $S$ such that the elements of $C$ which are not in $S$ are $1,2,\ldots,n'$. Observe that 
\begin{enumerate}
\item The total weight of the cores $C$ such that all six $n'$ are in the same column is $15{y'}W_S(y'-1)$ as each such core $C$ can be obtained by starting from a core $C_0$ with $y'-1$ columns which is compatible with $S$, choosing one of $y'$ positions to insert the column with six $n'$, and then choosing one of $15$ different pairings for the six $n'$.
\item The total weight of the cores $C$ such that four $n'$ are in one column and the remaining two $n'$ are in a different column is $9n(n-1)W_S(y'-1)$. To see this, observe that each such core can be obtained by starting from a core $C_0$ with $y'-1$ columns which is compatible with $S$, choosing one of the $3(y'-1)$ pairs of elements $x$ in $C_0$, attaching a chain $n \to x$ of length $1$ to this pair of elements, choosing one of $y'$ possible locations for the column with four $n'$, and then choosing one of $3$ possible pairings for this column.
\end{enumerate}
Analyzing the total weight of the cores $C$ such that the six $n'$ are in three different columns is trickier. One way to do this is to consider the following procedure for choosing such a core $C$ and compare it with the following procedure for choosing a core $C_2$ with $y'-1$ columns which is compatible with $S$.

Procedure for choosing $C$:
\begin{enumerate}
\item Start with a partial table $T$ which has all of the elements of $C$ except for the six $n'$ and twelve other elements.
\item Choose how to arrange these elements into three columns such that each column contains a pair of $n'$ and the other elements are paired up as well.
\item Insert these $3$ columns into $T$ by choosing which $3$ of the $y'$ columns will be these columns. There are $\binom{y'}{3}$ choices for this.
\end{enumerate}

Procedure for choosing $C_2$:
\begin{enumerate}
\item Start with a partial table $T$ which has all but twelve of the elements of $C_2$:
\item Choose how to arrange these twelve elements into two columns so that these elements are paired up.
\item Choose two of the $y'-1$ positions for these columns. There are $\binom{y'-1}{2}$ choices for this.
\end{enumerate}

We make the following observations about these procedures:
\begin{enumerate}
\item By Lemma A.2 of \cite{BLP23}, after taking signs into account, the total weight of the choices for the second step for choosing $C$ is $6$ times the total weight of the choices for the second step for choosing $C_2$.
\item For each such core $C$, there is a unique way to choose $C$ via this procedure as $T$ must be the $y'-3$ columns of $C$ which don't contain $n'$. On the other hand, for each such core $C_2$, there are $\binom{y'-1}{2}$ different ways of choosing $C_2$ as we can take $T$ to be any $y'-3$ of the $y'-1$ columns of $C$.
\end{enumerate}
Putting these observations together, the total weight of the columns $C$ where the six $n'$ appear in $3$ different columns is $6\binom{y'}{3} = y'(y'-1)(y'-2)W_S(y'-1)$ as the total weight of the columns $C_2$ with $y'-1$ columns which are compatible with $S$ is $W_S(y'-1)$.

Adding these contributions together, we have that 
\[
W_S(y') = \left(15y' + 9y'(y'-1) + y'(y'-1)(y'-2)\right)W_S(y'-1) = y'(y'+2)(y'+4)W_S(y'-1),
\]
as needed.
\end{proof}
\begin{corollary}\label{cor:coregeneratingfunction}
for each shell $S$, if we take $C_{S}(t) = \frac{{t'}^{k_S}N^{(k_S)}(t')Weight_{Cores}(S)}{\prod_{j=1}^{k_S}{j(j+2)(j+4)}}$ and write $C_{S}(t) = \sum_{y=0}^{\infty}{{c_y}t^y}$ then the total weight of the partial tables $C'$ with $y$ columns such that
\begin{enumerate}
\item $C'$ is compatible with $S$, $shell(C')$ is empty, and $C'$ does not have a floating component. In other words, $C'$ is a possible expanded core for a table $T$ such that $\shellT = S$.
\item The $y-k_S$ elements of $C'$ which are not in $S$ are labeled only by their relative order.
\end{enumerate}
is $y!(y-k_S)!{c_y}$.
\end{corollary}
\begin{proof}
We first show that if we write 
\[
\frac{{t}^{k_S}N^{(k_S)}(t)Weight_{Cores}(S)}{\prod_{j=1}^{k_S}{j(j+2)(j+4)}} = \sum_{y'=0}^{\infty}{{c_{y', \text{ no chains}}}t^{y'}}
\]
then the total weight of the cores $C$ with $y'$ columns which are compatible with $S$ where the elements of $C$ which are not in $S$ are labeled only by their relative order is ${y'}!(y'-k_S)!c_{y', \text{ no chains}}$.

To see this, recall that $N(t) = \sum_{y'=0}^{\infty}{n_{y'}t^{y'}}$ where $n_{y'} = \frac{y'!(y'+2)!(y'+4)!}{(y'!)^2}$ so 
\[
t^{k_S}N^{k_S}(t) = \sum_{y'=0}^{\infty}{\frac{y'!}{(y'-k_S)!}n_{y'}t^{y'}} = \sum_{y'=0}^{\infty}{\frac{y'!(y'+2)!(y'+4)!}{y'!(y' - k_S)!}t^{y'}}
\]
and thus $c_{y', \text{ no chains}} = \frac{y'!(y'+2)!(y'+4)!}{y'!(y' - k_S)!} \cdot \frac{Weight_{Cores}(S)}{\prod_{j=1}^{k_S}{j(j+2)(j+4)}}$.
By Lemma \ref{lem:coretotalweight}, the total weight of the cores $C$ with $y'$ columns which are compatible with $S$ where the elements of $C$ which are not in $S$ are labeled only by their relative order is
\begin{align*}
\left(\prod_{j=k+1}^{y'}{j(j+2)(j+4)}\right)Weight_{Cores}(S) &= \frac{y'!(y'+2)!(y'+4)!}{\prod_{j=1}^{k_S}{j(j+2)(j+4)}}Weight_{Cores}(S) \\
&= {y'}!(y'-k_S)!c_{y', \text{ no chains}}.
\end{align*}

The reason why we replace $t$ with $t' = \frac{t}{\left(1-(\mu_4 - 3)t\right)^3}$ is the same reason why we have $N(t')$ rather than $N(t)$ for the case when $m_1 = 0$. For each column in the core, each of the three pairs in the column may have a chain of known $4$-columns leading to it. Thus, for each $t$, we have a factor of $\frac{1}{\left(1-(\mu_4 - 3)t\right)^3}$ which corresponds to replacing $t$ by $t' = \frac{t}{\left(1-(\mu_4 - 3)t\right)^3}$.

Since our setting is more subtle, we give a more careful verification that replacing $t$ with $t'$ is correct. To do this, we show that the total weight of the expanded cores $C'$ such that $C'$ is compatible with $S$, $C'$ has $y$ columns, and $core(C') = C$ has $y'$ columns is the same as the coefficient of $t^{y}$ in $y!(y-k_S)!c_{y', \text{ no chains}}(t')^{y'}$. Observe that 
\begin{enumerate}
\item The coefficient of $t^{y}$ in $y!(y-k_S)!c_{y', \text{ no chains}}(t')^{y'}$ is 
\[
y!(y-k_S)!\left(\mu_4 - 3\right)^{y-y'}\binom{3y' + (y-y') - 1}{3y'-1}c_{y', \text{ no chains}}
\] as we have to choose $y-y'$ factors of $(\mu_4 - 3)t$ from $3y'$ factors of $\frac{1}{1 - (\mu_4 - 3)t} = \sum_{j=0}^{\infty}{(\mu_4 - 3)^j}$ which corresponds to putting $y-y'$ unlabeled balls into $3y'$ labeled bins.
\item The total weight of the expanded cores $C'$ such that $C'$ is compatible with $S$, $C'$ has $y$ columns, and $C = core(C')$ has $y'$ columns is 
\begin{align*}
&(\mu_4 - 3)^{y-y'}\binom{3y' + (y-y') - 1}{3y'-1}\left(\prod_{j=y'+1}^{y}{j(j-k_S)}\right){y'}!(y'-k_S)!c_{y', \text{ no chains}} \\
&= y!(y-k_S)!\left(\mu_4 - 3\right)^{y-y'}\binom{3y' + (y-y') - 1}{3y'-1}c_{y', \text{ no chains}}
\end{align*}
as 
\begin{enumerate}
\item[1.] In order to add $y - y'$ known $4$-columns to a core $C$ with $y'$ columns, we need to have a total length of $y-y'$ for the $3y'$ chains (possibly of length $0$) leading to pairs of elements of $C$.
\item[2.] Whenever we increase the length of a chain by $1$, if we currently have $j-1$ columns then we have a factor of $j(j-k_S)$ as there are $j$ choices for the position of the new column and there are $j-k_S$ choices for where the new element is inserted in the ordering of the added elements (i.e., the elements which are in $C'$ but not in $C$).
\end{enumerate}
\end{enumerate}

\end{proof}
\section{Implementation Details}
\label{ch:implementation}

When enumerating the possible shells, it is useful to first place every part of the shell except the pairs. 
Starting from a fixed mark-position table \(P\), the program enumerates every
admissible placement of triples of covered elements and blocks of size 4 (we can avoid blocks of size $5$ using the reduction in \ref{sec:connectionformula} or by adjusting the column factor for ${\underline{\!\times\!}\,}_5^1$ columns, see \cref{def:columnfactor}). 
The program then partitions the remaining unmarked vacancies into
anonymous pairs without giving them persistent global names.  We call the
resulting table 
a \emph{partial shell} 
and denote it by
\(\widetilde S\). 

The total contribution to $F_6(t)$ from shells $S$ whose partial shell is $\widetilde{S}$ can be split into two parts. The first part consists of a symmetry factor (see \cref{sec:implementation-paper-symmetries}) and the product $\prod_cG(c)$ of the column factors. The second part comes from the possible choices for which pairs of anonymous elements are equal to each other and/or to a covered element and the possible choices for combining the resulting shell $S$ with an expanded core. This second part only depends on the row supports of the covered elements and anonymous pairs in $\widetilde{S}$ which allows us to reuse computations.

To capture this, the program expresses the row supports of the covered elements of $\widetilde{S}$ which don't appear in every row of $\widetilde{S}$ and the anonymous pairs of $\widetilde{S}$ as \emph{fragments} where a fragment records the following:
\begin{enumerate}
\item whether the element appears two or four times in $\widetilde{S}$. This is recorded with a polarity which is $+$ if the element appears four times in $\widetilde{S}$ and $-$ if the element appears twice in $\widetilde{S}$.
\item The row support of the element in $\widetilde{S}$. We record this by recording the pair of row indices $i < j \in [6]$ such that the element either has $-$ polarity and has row support $\{i,j\}$ or has $+$ polarity and has row support $[6] \setminus \{i,j\}$.
\item Whether the element is a covered element or part of an anonymous pair. Note that fragments with $+$ polarity must be covered elements.
\end{enumerate}
We call the collection of fragments for $\widetilde{S}$ a partial shell interface key. This is formalized in \cref{sec:implementation-key}.

To avoid repeating the same computations, we have an \emph{outer partial-shell dictionary}. Whenever we encounter a partial shell interface key which is in our outer partial-shell dictionary or is isomorphic to a partial shell interface key in our dictionary, we look up the result for this partial shell interface key while being careful to use the correct sign (see \cref{sec:implementation-signs} for details). 

When we encounter a partial shell interface key which is not isomorphic to a partial shell interface key in our dictionary, we branch over the possible ways that the elements of $\widetilde{S}$ can be identified with each other. Each identification branch \(b\) results in a shell \(S_b\). At this point, we only need to compute the net number $Weight_{Cores}(S_b)$ of minimum compatible cores for $S_b$.

We obtain a further savings of computations by observing that $Weight_{Cores}(S_b)$ only depends on the row supports of the elements of $S_b$ which do not appear in every row of $S_b$. This data is recorded in a shell interface key which is obtained from the partial shell interface key and the branch $b$ as follows:
\begin{enumerate}
\item When elements in fragments are identified with each other, we either merge these fragments into a new $+$ fragment or delete these fragments depending on whether the element appears four or six times in the resulting shell $S_b$. This must be handled carefully as it is important to have the correct multiplicity and parity for these identifications. We describe how we do this in \cref{sec:implementation-identifications}.
\item For each element which is still anonymous, we give it a fresh label. Thus, we now consider all fragments to be labeled and it is no longer true that $+$ fragments must be covered elements.
\end{enumerate}

When computing $Weight_{Cores}(S_b)$, we use an \emph{inner compatible-core dictionary} with these shell interface keys. Whenever we encounter a shell interface key which is in our inner dictionary or is isomorphic to a shell interface key in our inner dictionary, we look up the result for this shell interface key while being careful to use the correct sign (see \cref{sec:implementation-signs} for details). When we encounter a shell interface key which is not isomorphic to a shell interface key in our inner dictionary, we compute $Weight_{Cores}(S_b)$ and add it to our dictionary.

In summary, at the interface key level, each identification branch follows the roadmap
\begin{equation}
  K_{\mathrm{par}}(\widetilde S)=K_0
  \xrightarrow{\operatorname{identifications}} K_{\mathrm{par},b}
  \xrightarrow{\operatorname{label}} K(S_b).
  \label{eq:implementation-interface-key-roadmap}
\end{equation}
where $K_{\mathrm{par}}(\widetilde S)=K_0$ is the initial partial shell interface key, $K_{\mathrm{par},b}$ is the modified interface key after the identifications given by the branch $b$, and $K(S_b)$ is the final shell interface key obtained by giving labels to the anonymous elements in $K_{\mathrm{par},b}$. 
The outer partial-shell dictionary stores the complete sum over identification branches for each partial shell interface key while the inner dictionary stores the total weight of the minimum compatible cores for shells with a given shell interface key, which is the data needed by one branch of the identification.

We now describe how we implement this framework as a six-stage executable pipeline, including both dictionary lookups. 
The calculations for this project are implemented in MATLAB.
\subsection{Computational pipeline}
\label{sec:implementation-target-pipeline}

The computation begins with a feasible mark-position table \(P\), specified in \cref{sec:implementation-paper-symmetries}, and then follows six stages:
\begin{enumerate}
  \item Initialize the case from its mark-position table \(P\) and record \(|\Aut(P)|\) for the multiplicity factor;
  \item Enumerate every admissible placement, compatible with \(P\), of triples and known blocks of size \(4\) of covered elements, retaining branches on which some copies of covered elements 
  remain unplaced and emitting separate \(\underline{4}\)-columns in a prescribed covered-element order;
  \item Pair all remaining unmarked vacancies, thereby producing the literal paired partial shells, record \(n_{\underline{3}}(\widetilde S)\) where $n_{\underline{3}}(\widetilde S)$ is the number of $\underline{3}$ columns in $\widetilde{S}$, 
  and attach the symmetry factor defined in \cref{sec:implementation-paper-symmetries};
  \item Compute the local factor $G(c)$ of every shell column using the definitions in \cref{sec:shell-column-factors}, extract the 
  partial-shell interface 
  key, and query the outer partial-shell dictionary;
  \item 
  If the partial shell interface key is not in the outer partial shell dictionary, enumerate every admissible identification branch, give each surviving anonymous fragment a distinct fresh shell label, compute the determinant parity and excess \(k_S\), and either obtain the signed compatible-core weight from the inner dictionary or compute it and store it in the dictionary. After this, sum the contributions from the identification branches and store the result in the outer partial-shell dictionary.
  \item 
  Combine the reused or newly computed result with the current local and case-level factors.
\end{enumerate}
The column patterns and corresponding local factors used by the implementation are listed in \cref{lem:columnfactors}. The dictionary hit--miss control flow is summarized in \cref{fig:implementation-outer-dictionary}. Note that the inner dictionary is reached only along the outer miss branch. For each shell \(S_b\), it either returns the previously computed value of $Weight_{Cores}(S_b)$ 
or invokes the exact compatible-core calculation described in \cref{sec:implementation-core-dictionary} and stores the result, as displayed later in \cref{fig:implementation-core-weight-dictionary}. The two dictionaries therefore remove repetition at different scales.



The inner dictionary is reached only along the outer miss branch. For each materialized shell \(S_b\), it either returns the previously computed value of \(Weight_{Cores}(S_b)\) or invokes the exact compatible-core calculation described in \cref{sec:implementation-core-dictionary} and stores the result, as displayed later in \cref{fig:implementation-core-weight-dictionary}. The two dictionaries therefore remove repetition at different scales. 

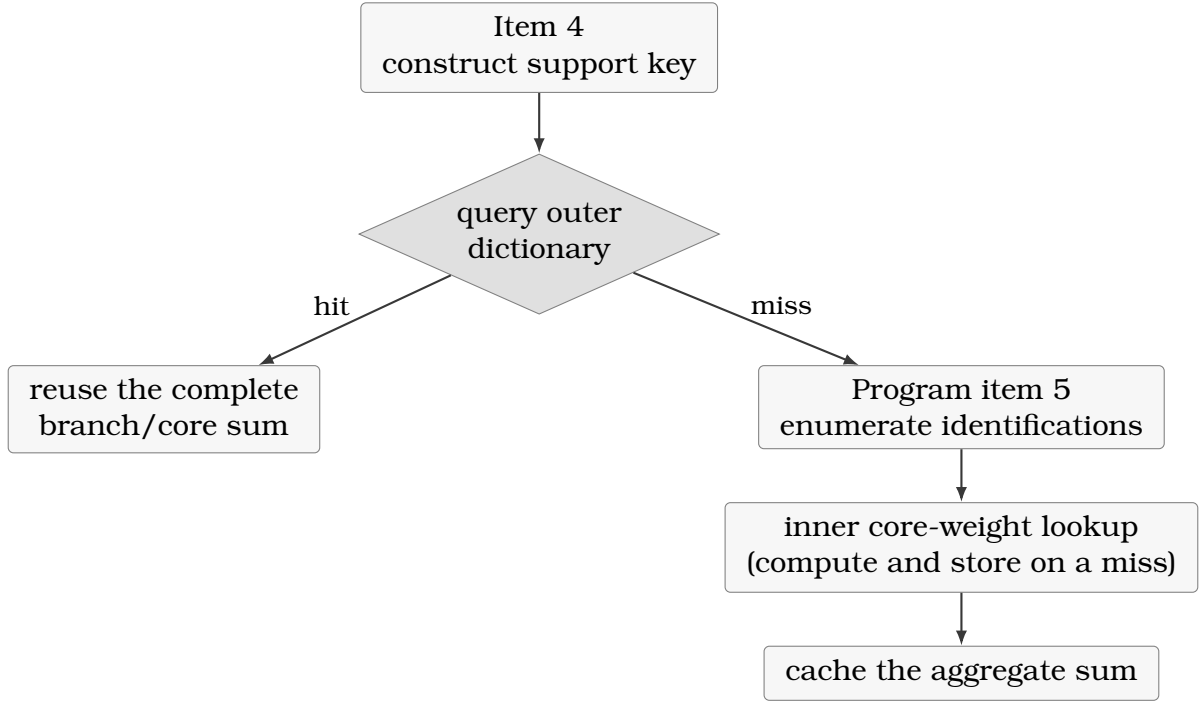
\begin{figure}[htbp]
\centering
\begin{tikzpicture}[
  cache stage/.style={
    draw=diagramgray,
    fill=diagramlightblue,
    rounded corners=2pt,
    align=center,
    inner xsep=8pt,
    inner ysep=5pt,
    font=\small
  },
  cache query/.style={
    diamond,
    aspect=2.25,
    draw=diagramgray,
    fill=diagramlightorange,
    align=center,
    inner xsep=4pt,
    inner ysep=3pt,
    font=\small
  },
  cache flow/.style={-{Latex[length=2mm]},draw=diagramblue,thick},
  cache branch/.style={font=\footnotesize,fill=white,inner sep=1.5pt}
]
  \node[cache stage] (key)
    {Item 4\\construct support key};
  \node[cache query,below=8mm of key] (query)
    {query outer\\dictionary};
  \node[cache stage,below left=12mm and 17mm of query] (hit)
    {reuse the complete\\branch/core sum};
  \node[cache stage,below right=12mm and 17mm of query] (identify)
    {Program item 5\\enumerate identifications};
  \node[cache stage,below=7mm of identify] (cores)
    {inner core-weight lookup\\(compute and store on a miss)};
  \node[cache stage,below=7mm of cores] (store)
    {cache the aggregate sum};

  \draw[cache flow] (key) -- (query);
  \draw[cache flow] (query) -- node[cache branch,above left] {hit} (hit);
  \draw[cache flow] (query) -- node[cache branch,above right] {miss} (identify);
  \draw[cache flow] (identify) -- (cores);
  \draw[cache flow] (cores) -- (store);
\end{tikzpicture}
\caption{The outer partial-shell lookup; the inner compatible-core lookup is
reached only on an outer miss.}
\label{fig:implementation-outer-dictionary}
\end{figure}

\subsection{Mark-position tables and symmetry factors}\label{sec:implementation-paper-symmetries}
For the formula in Theorem \ref{thm:Ftformula}, we need to only sum over non-isomorphic shells $S \in \mathcal{S}$ and we need to multiply this by $\frac{6!}{|Aut(S)|}$. Since checking whether a shell $S$ is isomorphic to a shell we've already considered and computing the size of the automorphism group of $S$ is complicated, we instead use a simpler method which gives the same result. To describe this method, we need a few definitions.
\begin{definition}
We define a mark placement table $P$ to be a $6 \times j$ table such that 
\begin{enumerate}
\item Each entry of $P$ is either blank or an $\times $ together with the element it covers.
\item Each column of $P$ contains at least one $\times $ and each row of $P$ contains at most one $\times $.
\end{enumerate}
\end{definition}
\begin{definition}
We say that a shell $S$ is consistent with a mark placement table $P$ if the positions of the marks in $S$ and $P$ are the same.

We say that a mark placement table $P$ is feasible if there exists a shell $S$ such that $S$ is consistent with $P$. We define $\mathcal{P}$ to be the set of all feasible possible mark placement tables up to isomorphism.
\end{definition}
\begin{definition}
Given a mark position table $P$, we define the automorphism group $Aut(P)$ of $P$ to be the subgroup of $Sym(Rows) \times Sym(Columns) \times Sym(Elements)$ which preserves $P$ (where we only consider the columns and elements contained in $P$). 
\end{definition}
Our method for handling symmetries of the shells is as follows:
\begin{enumerate}
\item We consider each of the possible mark position tables $P \in \mathcal{P}$ one by one.
\item We enumerate over the possible shells $S$ which are consistent with $P$ such that
\begin{enumerate}
\item[1.] For pairs in $S$ which are not equal to a covered element, we do not specify a specific element for the pair. That said, we do specify which which pairs are equal to covered elements and which pairs are equal to each other.
\item[2.] For the columns of $S$, we first have columns containing marks followed by $\underline{3}$ columns and then $\underline{4}$ columns where the $\underline{4}$ columns are in a specific order.
\end{enumerate}
\item For each shell $S$, We multiply by $\frac{6!}{|Aut(P)| \cdot (\# \text{ of } \underline{3} \text{ columns})!}$.
\end{enumerate}
We now explain why this is equivalent to summing over non-isomorphic shells $S \in \mathcal{S}$ and multiplying by $\frac{6!}{|Aut(S)|}$. Let $P \in \mathcal{P}$ be a mark position table which we consider and let $S$ be a shell which is consistent with $P$. Observe that since we do not consider other mark position tables $P_2$ which are isomorphic to $P$ but not equal to $P$, the only permutations $\pi \in Sym(Rows) \times Sym(Columns) \times Sym(Elements)$ which can yield a shell $S_2$ that we count which is isomorphic to $S$ (including $S$ itself) are permutations $\pi$ which preserve $P$ and preserve the fact that the columns are ordered so that the marked columns come first followed by the $\underline{3}$ columns and then the $\underline{4}$ columns where the $\underline{4}$ columns are in a specific order. Equivalently, $\pi$ must consist of an element of $Aut(P)$ together with a permutation of the $\underline{3}$ columns. 
Thus, there are $|Aut(P)| \cdot (\# \text{ of } \underline{3} \text{ columns})!$ 
such permutations $\pi$. 
\begin{remark}
Note that we only consider permutations of the covered elements as not assigning a specific element to the other pairs in $S$ ensures that we do not count two different isomorphic shells which are only different because of how the pairs are labeled.
\end{remark}
Now observe that for each shell $S$, there are $\frac{|Aut(P)| \cdot (\# \text{ of } \underline{3} \text{ columns})!}
{|Aut(S)|}$ different shells which are isomorphic to $S$ which we consider. Since we only want to consider one of the shells which are isomorphic to $S$, we divide by $\frac{|Aut(P)| \cdot (\# \text{ of } \underline{3} \text{ columns})!}
{|Aut(S)|}$. Combined with the factor of $\frac{6!}{|Aut(S)|}$ from the formula in Theorem \ref{thm:Ftformula}, this gives a factor of 
\begin{equation}
    \frac{6!}{|Aut(P)| \cdot (\# \text{ of } \underline{3} \text{ columns})!}
\end{equation}\label{eq:implementation-partial-shell-factor}

as needed.
\begin{example}
If $P$ is the following mark placement table then $|Aut(P)| = 4!*2! = 48$ so $\frac{6!}{|Aut(P)|} = 15$.
\begin{center}
\begin{tabular}{|c|}
\hline
$\times a$ \\ \hline
$\times a$ \\ \hline
$\times a$ \\ \hline
$\times a$ \\ \hline
$ $ \\ \hline
$ $ \\ \hline
\end{tabular}
\end{center}
For the two shells $S$ which are consistent with $P$, $Aut(S) = Aut(P)$ and there are no shells $S_2 \neq S$ which are isomorphic to $S$ and consistent with $P$. Thus, for both of these shells, $\frac{6!}{|Aut(S)|} = \frac{6!}{|Aut(P)|} = 15$.
\end{example}
\begin{example}
If $P$ is the following mark placement table 
\begin{center}
\begin{tabular}{|c|c|}
\hline
$\times a$ & $ $ \\ \hline
$\times a$ & $ $ \\ \hline
$ $ & $\times b$ \\ \hline
$ $ & $\times b$ \\ \hline
$ $ & $ $ \\ \hline
$ $ & $ $ \\ \hline
\end{tabular}
\end{center}
then $|Aut(P)| = 16$ as in addition to being able to swap rows $1$ and $2$, swap rows $3$ and $4$, and/or swap rows $5$ and $6$, we can also swap rows $1,3$, swap rows $2,4$, swap the elements $a,b$, and swap the columns and together, these swaps give an additional automorphism of $P$. Thus, $\frac{6!}{|Aut(P)|} = 45$.

If $S$ is the following shell which is consistent with $P$
\begin{center}
\begin{tabular}{|c|c|}
\hline
$\times a$ & $d$ \\ \hline
$\times a$ & $e$ \\ \hline
$a$ & $\times b$ \\ \hline
$c$ & $\times b$ \\ \hline
$a$ & $d$ \\ \hline
$c$ & $e$ \\ \hline
\end{tabular}
\end{center}
then $|Aut(S)| = 1$. To see this, observe that none of the elements can be mapped to a different element by an automorphism as $a$ and $b$ are the only covered elements, $a$ appears four times in a column while $b$ does not, $c$ is the only element which is not $a$ and appears in the same column as $a$, and $d$ is the only element such that both of the rows it appears in also contain $a$. Since the columns can't be swapped by an automorphism and the rows are all distinct, $S$ has no non-trivial automorphisms. That said, there are $16$ shells $S_2$ (including $S$ itself), which are compatible with $P$ and are isomorphic to $S$. To see this, observe that we can choose either $a$ or $b$ to appear four times in its column, there are $4$ choices for the additional rows where this element appears (once in row $5$ or $6$ and once in the remaining two rows), and there are $2$ choices for the pairing in the other column (as rows $5$ and $6$ cannot be part of a pair). 

Thus, $\frac{6!}{|Aut(S)|*16} = \frac{6!}{|Aut(P)|} = 45$.
\end{example}
\begin{example}
If $P$ is the following mark placement table then $|Aut(P)| = 6! = 720$ so $\frac{6!}{|Aut(P)|} = 1$.
\begin{center}
\begin{tabular}{|c|}
\hline
$\times a$ \\ \hline
$\times b$ \\ \hline
$\times c$ \\ \hline
$\times d$ \\ \hline
$\times e$ \\ \hline
$\times f$ \\ \hline
\end{tabular}
\end{center}
For this mark placement table, we will always have three $\underline{3}$ columns so for all shells $S$ which are consistent with $P$, we have an additional factor of $\frac{1}{3!} = \frac{1}{6}$. For example, if $S$ is the following shell
\begin{center}
\begin{tabular}{|c|c|c|c|}
\hline
$\times a$ & $d$ & $e$ & $f$ \\ \hline
$\times b$ & $d$ & $e$ & $f$ \\ \hline
$\times c$ & $d$ & $e$ & $f$ \\ \hline
$\times d$ & $a$ & $b$ & $c$ \\ \hline
$\times e$ & $a$ & $b$ & $c$ \\ \hline
$\times f$ & $a$ & $b$ & $c$ \\ \hline
\end{tabular}
\end{center}
then 
\begin{enumerate}
\item $|Aut(S)| = 12$ as in addition to being able to apply a permutation to $a,b,c$ and then apply the same permutation to $d,e,f$, rows $1,2,3$, rows $4,5,6$, and columns $2,3,4$, we can also swap $a$ and $d$, swap $b$ and $e$, swap $c$ and $f$, swap rows $1$ and $4$, swap rows $2$ and $5$, and swap rows $3$ and $6$ and together, these swaps give an automorphism of $S$.
\item There are $360$ shells $S_2$ (including $S$ itself) which are isomorphic to $S$ and consistent with $P$ as there are $10$ ways to choose the partition of the $\underline{3}$ columns into two triples, $6$ ways to choose how the elements in the top triple are matched up with the elements in the bottom triple (e.g., for the above shell $S$, the matching is $(a,d), (b,e), (c,f)$), and $6$ ways to choose the order of the $\underline{3}$ columns.
\end{enumerate}
Observe that $\frac{6!}{|Aut(S)|*360} = \frac{6!}{|Aut(P)|*6} = \frac{1}{6}$, as needed.
\end{example}
\subsubsection{Data for mark position tables} 
For our implementation, we need to enumerate over all of the possible mark position tables. To do this, we partition the possible mark position tables based on how many marks they have and assign each mark position table a case identifier. We then determine $|Aut(P)|$ and record  $n_{\underline{3}}$ (which can be deduced just from $P$). This is sufficient to give the symmetry factor in \eqref{eq:implementation-partial-shell-factor} which is inherited by every partial shell $\widetilde{S}$ which is consistent with $P$.


The number of different mark position tables for each number of marks is shown in \cref{tab:implementation-term-counts}.

\begin{table}[H]
\centering
\caption{Numbers of different mark position tables for a given number of marks
.}
\label{tab:implementation-term-counts}
\small
\begin{tabular}{@{}cc@{}}
\toprule
Number of marks & Listed cases\\
\midrule
1 & 1\\
2 & 3\\
3 & 8\\
4 & 24\\
5 & 51\\
6 & 159\\
\bottomrule
\end{tabular}
\end{table}


We now describe how we enumerate over all of the paired partial shells $\widetilde{S}$ which are consistent with a given mark position table $P$.
\subsection{Generating paired partial shells}
\label{sec:implementation-partial-shells}

Throughout the implementation, rows are numbered by $1,2,\ldots,6$.  A row
pair is the sorted tuple $(i,j)$ with $i<j$.  Its complement is the
four-element subset $[6]\setminus\{i,j\}$, but only the pair itself is stored.

Fix one of the mark-position records counted above.  Stages two and three
proceed in two passes: first enumerate the covered triples and blocks of size $4$;
then pair the vacancies.  The first pass does not attempt to place all six
copies of every covered element.

\subsubsection{Placing covered triples and four-blocks}
\label{sec:implementation-covered-labels}

For every covered element \(a\), record the rows in which the marks cover \(a\)
and the complementary rows available to its unmarked copies.  At this stage
the generator places only triples and known blocks of size \(4\) of covered
elements:
\begin{enumerate}
  \item If \(a\) is covered in two rows, we either don't place further copies of $a$, place a block of size $4$ of $a$ in a compatible four-position vacancy of an existing column (which can only be a column with two marks covering $a$ and no other marks), or place a block of size $4$ of $a$ in a new column. 
  \item If \(a\) is covered in one or three rows, enumerate every compatible placement of a triple of $a$. 
  A placement is rejected if it creates an unmarked singleton or violates a known pairing.  In the one-mark case, the triple accounts for only three of the five unmarked copies, so two copies remain unresolved at this stage.
\end{enumerate}
Note that some copies of a covered element may remain unplaced and later enter the step where we branch over possible identifications of elements. 
\subsubsection{Pairing the vacancies}

After the selected covered triples and four-blocks have been placed, every
empty unmarked position must belong to a pair.  A two-position vacancy has one
pairing.  Four positions $a<b<c<d$ have the three pairings
\begin{equation}
  ab|cd,
  \qquad ac|bd,
  \qquad ad|bc.
  \label{eq:implementation-three-pairings}
\end{equation}
No shell column has six vacancies: such a column is wholly unmarked and
belongs to the Gaussian core rather than the shell.  Thus the two- and
four-position cases above exhaust the shell pairing step. 

Each admissible choice for placing blocks of size $4$ of covered elements (including not placing such blocks), placing triples of covered elements, and pairing the vacancies gives a separate partial shell.
\subsection{Interface keys}
\label{sec:implementation-key}

\subsubsection{Fragments}
After we place the triples and blocks of size $4$ of covered elements, we represent the current row supports of the covered elements and anonymous pairs by fragments. Each fragment consists of the following:
\begin{center}
\begin{tabular}{@{}ll@{}}
\toprule
Field & Meaning\\
\midrule
\texttt{support} & one of the fifteen row pairs;\\
\texttt{polarity} & \texttt{MINUS} for a two-row support, \texttt{PLUS} for its four-row complement;\\
\texttt{label} & an immutable shell-element identifier, or \texttt{None} in a partial-shell key.\\
\bottomrule
\end{tabular}
\end{center}
The word ``plus'' does not mean positive determinant sign.  It means that the element occupies the four rows complementary to the stored pair.  Multiplicity is not a field of a fragment: each anonymous pair 
in a partial shell gives a separate fragment, even when several fragments have the same three field values.

A labeled fragment is written $(\ell,\sigma,ij)$; an anonymous fragment is $(u,\sigma,ij)$. 
The value \texttt{None} occurs only in partial shell and intermediate interface keys; every fragment in a shell key has a label. If $j$ distinct anonymous fragments have row-pair index $ij$ and polarity $s\in\{-,+\}$, we use the display shorthand
\begin{equation}
  j u_{ij}^{s}.
  \label{eq:implementation-anonymous-fragment-shorthand}
\end{equation}
The leading coefficient $j$ abbreviates repeated entries in the surrounding fragment multiset; it is not part of any one fragment.  An implementation may run-length encode that multiset internally, but every choice that distinguishes fragments must behave exactly as the occurrence-level enumeration.

\paragraph{Example.}
Suppose the partial shell contains 
a covered element $a$ in rows $1,2$, a second covered element $b$ in rows $3,4,5,6$, a third covered element $c$ in rows $5,6$, a fourth covered element $d$ in rows $1,2,5,6$, 
an anonymous pair in rows $1,2$, and two distinct anonymous pairs in rows $3,4$.  The labeled fragments are
\begin{align*}
  a^-_{12}&=\texttt{Fragment}((1,2),\texttt{MINUS},a),\\
  b^+_{12}&=\texttt{Fragment}((1,2),\texttt{PLUS},b),\\
  c^-_{56}&=\texttt{Fragment}((5,6),\texttt{MINUS},c).\\
  d^+_{34}&=\texttt{Fragment}((3,4),\texttt{PLUS},d).
\end{align*}
The unlabeled fragments are $u^-_{12} = \texttt{Fragment}((1,2),\texttt{MINUS},\texttt{None})$ and $2u_{34}^{-}$. The shorthand $2u_{34}^{-}$ abbreviates the two separate records
\[
  \texttt{Fragment}((3,4),\texttt{MINUS},\texttt{None}),
  \qquad
  \texttt{Fragment}((3,4),\texttt{MINUS},\texttt{None}).
\]
Note that the $+$ fragments store the two rows which do not contain the corresponding covered element rather than the four rows which contain this element. The fragments $a^-_{12}$ and $b^+_{12}$ cannot be identified because their existing labels are distinct though the fragments $b^+_{12}$ and $u^-_{12}$ can be identified. Similarly, the $a^-_{12}$ and $c^-_{56}$ fragments cannot be identified because they have different labels. The two $u^-_{34}$ fragments are separate fragments which can be identified with either $a^-_{12}$ or $c^-_{56}$. 

\subsubsection{Definition of interface keys}

Literal fragment names are immaterial, but their supports, polarities, multiplicities, and labeled equality classes are not.  The following definition captures this data.

\begin{definition}[Interface key]
\label{def:implementation-interface-key}
Let $T$ be a partial shell or shell. 
The \emph{interface key} of $T$ is the tuple 
\begin{equation}
  K(T)=
  \bigl(
  K_{\ell,-},K_{u,-},K_{\ell,+},K_{u,+}
  \bigr),
  \label{eq:implementation-key}
\end{equation}
where
\begin{itemize}
  \item $K_{\ell,-}$ is the multiset of pairs $(i,j)$ such that there is a labeled $-$ fragment with support $\{i,j\}$.
  \item $K_{u,-}$ is the multiset of pairs $(i,j)$ such that there is an unlabeled $-$ fragment with support $\{i,j\}$.
  \item $K_{\ell,+}$ is the multiset of pairs $(i,j)$ such that there is a labeled $+$ fragment with support $\{i,j\}$. 
  \item $K_{u,+}$ is the multiset of  pairs $(i,j)$ such that there is an unlabeled $+$ fragment with support $\{i,j\}$. 
\end{itemize}
Here $\ell$ and $u$ mean labeled and unlabeled, while $-$ and $+$ encode two-row and complementary four-row support. Repeated entries are retained as separate fragment occurrences.
\end{definition}

For a partial shell $\widetilde S$, we write
$K(\widetilde S) = K_{\mathrm{par}}(\widetilde S)$ and call this a \emph{partial shell interface
key}; its fragments may be unlabeled but there can be no unlabeled $+$ fragments (these can only occur in the intermediate identification steps). For a shell $S$, every surviving
fragment belongs to a shell element and is labeled (i.e., $K_{u,-}=K_{u,+}=\varnothing$), and we call $K(S)$ a \emph{shell interface key}. 
This distinction only concerns the stage of the computation, partial shell fragments and shell fragments are essentially the same kind of data type.


For example, suppose the partial shell is 
\[
\widetilde S=
\begin{pmatrix}
 a&\alpha\\
 a&\alpha\\
 b&\beta\\
 b&\beta\\
 b&c\\
 b&c
\end{pmatrix},
\]
where $a,b,c$ are labeled and $\alpha,\beta$ are anonymous. The shell interface key $K_{par}(\widetilde{S})$ is 
\[
  K_{\ell,-}=\{12,56\},\qquad
  K_{u,-}=\{12,34\},\qquad
  K_{\ell,+}=\{12\},\qquad
  K_{u,+}=\varnothing.
\]
Note that the fragment $b^+_{12}$ cannot be identified with the fragment $a^-_{12}$ but can be identified with the fragment $u^-_{12}$ corresponding to the pair of $\alpha$ in the diagram. 
On the 
identification branch with no identifications, the two anonymous fragments receive distinct fresh labels, say $d$ and $e$,
so the corresponding shell key has support projection
\[
  K_{\ell,-}=\{12,12,34,56\},\qquad
  K_{u,-}=\varnothing,\qquad
  K_{\ell,+}=\{12\},\qquad
  K_{u,+}=\varnothing.
\]

\begin{implementationbox}
The four blocks of a partial-shell interface key must remain semantically distinguishable.  Labeled versus anonymous and plus versus minus are not just cosmetic tags. A shell interface key uses the same four-block serialization with empty anonymous blocks. 
\end{implementationbox}

In our implementation, we use an equality-aware
interface record.  During a cache-miss calculation, the implementation carries
the corresponding label information internally, but it does not construct
this record as a separate lookup key.  For the outer lookup, the implementation
suppresses literal positive label values and serializes the resulting four
support blocks.  At this pre-identification point, each positive element that
remains on the boundary is represented by at most one aggregate two- or
four-row fragment, while every anonymous pair contributes one separate
fragment.  The program key retains row-pair multiplicities,
labeled/anonymous status, and polarity, but it does not retain the equality
partition of the positive labels.

The executable then examines all \(6!=720\) row images of this support key and
uses the least serialization as its row-orbit key, denoted
\(\widehat K_{\mathrm{par}}\).  The stored value is not merely a list of
identification templates or a single compatible-core count.  It is the
complete Mathematica expression for the identification-branch and
compatible-core sum computed from the first literal representative on a miss:
\begin{equation}
  \widehat K_{\mathrm{par}}
  \longmapsto
  \bigl(
    \widetilde S_{\mathrm{ref}},
    \mathcal A_{\mathrm{miss}}(\widetilde S_{\mathrm{ref}})
  \bigr).
  \label{eq:implementation-outer-cache}
\end{equation}
The aggregate \(\mathcal A_{\mathrm{miss}}\) is defined explicitly in
\cref{sec:implementation-core-memoized-lookup}.

The cache key does not include the originating mark-position table,
\(n_{\underline{3}}(\widetilde S)\), the partial-shell symmetry factor,
literal column incidence, the literal determinant sign, or the local
column-factor product.  Nor does it separately retain
\(n_{\underline{4}}(\widetilde S)\), which belongs to the literal placement
record and is not a denominator in the symmetry factor.  The reference sign
is already included in the stored expression, and a hit supplies the one
relative sign described in \cref{sec:implementation-signs}.  The other
omitted quantities belong to the surrounding literal-shell or later
case-level calculation.

The four-block support projection is the key used by the MATLAB
implementation.  Before any identification branches are enumerated, this key
is canonicalized over the \(720\) row permutations and queried in the outer
cache.  A hit reuses the stored complete branch/core aggregate with the
required relative sign, whereas a miss computes that aggregate and stores it
for later reuse.

The same four-block data type represents every intermediate identification
key.

\subsection{Admissible identifications}
\label{sec:implementation-identifications}
Given a partial shell interface key $K_0=K_{\mathrm{par}}(\widetilde S)$, the program enumerates all admissible identifications of its fragments.
\begin{implementationbox}
At no point may the implementation identify two distinct labeled shell
elements.  Anonymous pairs may be identified when their row supports allow
it, and an anonymous fragment may be absorbed into one labeled element, but
two existing labels encode distinct entries of the original permutation
table.
\end{implementationbox}

\paragraph{Identification branches.}
Let $\mathcal F_0$ be the occurrence set of $K_0$. Regard each fragment as one initial class. 
An \emph{identification branch} is an admissible set
partition $\mathcal P$ of $\mathcal F_0$ that coarsens these initial classes. A singleton block leaves its fragment unchanged.  Every nonsingleton block must have one of the three support patterns described below: a complementary
minus--plus pair, two minuses on disjoint row pairs, or three minuses on disjoint pairs that cover all six rows. Moreover, the defined labels in one block must all be equal.  Consequently, fragments already carrying distinct labels can never lie in the same block.

The 
resulting interface key $K_{\mathrm{par},b}$ is obtained from these blocks. 
Complementary minus--plus blocks and three-minus blocks form a completed
element and are omitted.  A two-minus block becomes one plus fragment on the
complementary pair and inherits the block's label if it has one; otherwise it
remains anonymous.  Singleton blocks are carried forward.  This set-partition
description makes a branch independent of the order in which a recursive
program happens to apply the displayed operations.  In particular, first
joining two minuses and then joining the resulting plus to the third minus is
the same three-minus block, not an additional branch.

\subsubsection{Complementary plus/minus identification}

A minus and a plus with the same row-pair index have complementary supports:
\begin{equation}
  (-_{ij},+_{ij})\longmapsto\varnothing.
  \label{eq:implementation-pm}
\end{equation}
Such a block occupies all six rows and disappears from the residual
partial-shell key.  Before branch enumeration begins, every plus fragment in
the initial key \(K_0\) is labeled.  Indeed, every anonymous name in a freshly
generated literal partial shell comes from a paired vacancy and therefore
occupies exactly two rows.  An anonymous plus fragment can occur only later,
for example as the output of a two-minus identification in
\eqref{eq:implementation-mm}.

Consequently, a complementary minus--plus block of \(\mathcal F_0\) is
admissible only when the minus is anonymous as the plus must carry a label, the plus and minus are forbidden from having the same label, and the plus and minus cannot both carry the same existing label as in this case the corresponding element would have already appeared $6$ times in $\widetilde{S}$ so it would not have formed a fragment. 

\subsubsection{Two-minus identification}

Let $ij,k\ell,rs$ be a partition of $[6]$ into three disjoint pairs.  A block
containing two minus fragments produces a plus on the complementary pair:
\begin{equation}
  (-_{ij},-_{k\ell})\longmapsto +_{rs}.
  \label{eq:implementation-mm}
\end{equation}
If one input is labeled, the output inherits that label.  Two distinct labeled inputs are forbidden.

\subsubsection{Three-minus identification}
For a block containing three minus fragments on disjoint row pairs, the corresponding element appears in all $6$ rows of $S$ so the fragment vanishes.
\begin{equation}
  (-_{ij},-_{k\ell},-_{rs})\longmapsto\varnothing.
  \label{eq:implementation-mmm}
\end{equation}
Again, at most one distinct pre-existing label may occur among the inputs.

In the implementation, on an outer-cache miss we examine these merge choices
in the fixed row-pair order.  The unchanged choice is included, so the resulting 
terms include nonmaximal as well as maximal identifications.  Mathematically, 
each resulting term is interpreted by the set partition of the original
occurrences defined above, rather than by an ordered history of rewrites.

\subsubsection{Counting repeated identification choices}

Every fragment is a distinct unit entity, but several fragments may have the same support, polarity, and label status. Suppose two eligible classes contain $u$ and $v$ individual fragments.  The number of ways to choose $q$ disjoint pair identifications is
\begin{equation}
  \frac{\fall{u}{q}\fall{v}{q}}{q!}.
  \label{eq:implementation-pair-multiplicity}
\end{equation}
Here
\[
  \fall{v}{q}=(v)_q
  :=v(v-1)\cdots(v-q+1)=\frac{v!}{(v-q)!},
  \qquad (v)_0:=1,
\]
is the falling factorial; it counts ordered selections of $q$ distinct
objects from $v$ available objects.
For three input classes the numerator has three falling factorials and the
denominator is again $q!$.  These formulas count occurrence-level blocks in
the branch partition, not sequences of elementary rewrites.  Every $q$ from
zero to the smallest available class size must be included.  Branches may be
combined into one coefficient only when they have the same parity and the
same materialized-shell contribution.  Choices with different parities or
different compatible-core contributions remain separate.  Every branch from
the same literal partial shell inherits the factor in
\eqref{eq:implementation-partial-shell-factor}. 

\paragraph{Example.}
Consider the partial-shell interface key
\begin{equation}
  K_{\mathrm{ex}}=
  \bigl(
    \varnothing,
    \{12,12,12,34,34\},
    \varnothing,
    \varnothing
  \bigr),
  \label{eq:implementation-multiplicity-example-intermediate}
\end{equation}
which contains three distinct anonymous minuses on $12$ and two distinct
anonymous minuses on $34$---in the shorthand of
\eqref{eq:implementation-anonymous-fragment-shorthand},
$3u_{12}^{-}$ and $2u_{34}^{-}$.  One fragment from each class may form an
anonymous plus on the complementary pair $56$.
Grouped by their residual key, the structural templates have the counts
\begin{equation}
\begin{array}{ccl}
\hline
q & \text{key-level count }n_q & \text{residual partial-shell key }K_{\mathrm{ex},q}\\
\hline
0 & 1
  & (\varnothing,\{12,12,12,34,34\},\varnothing,\varnothing)\\
1 & \fall{3}{1}\fall{2}{1}=6
  & (\varnothing,\{12,12,34\},\varnothing,\{56\})\\
2 & \fall{3}{2}\fall{2}{2}/2!=6
  & (\varnothing,\{12\},\varnothing,\{56,56\})\\
\hline
\end{array}
  \label{eq:implementation-multiplicity-example-branches}
\end{equation}
Thus retaining only the maximal choice $q=2$ would lose the unchanged
partition and the six occurrence-level partitions with exactly one
identification. 

The displayed $n_q$ records how many branches have the same interface key at this point. However, we must be careful as there may be a change in sign due to such an identification (see \cref{sec:implementation-signs}). 
Thus, we also need to record the parity of each branch (see 
\eqref{eq:implementation-branch}).  Before a
residual key is sent to the core calculation, each of its anonymous fragments receives a distinct fresh label.

\paragraph{Interpreting a branch as a shell.}
Fix a partial shell $\widetilde S$ and one admissible partition
$\mathcal P_b$ of its fragment occurrences.  Write
$K_{\mathrm{par},b}$ for the residual partial-shell key.
After applying exactly the blocks of $\mathcal P_b$, omit every element which now appears in 
all six rows. Each remaining fragment inherits the label of its
class if that class already has one; every remaining anonymous fragment
receives its own fresh label.  This produces the shell $S_b$ and its shell interface key
\[
  K':=K(S_b)
  =(K'_{\ell,-},\varnothing,K'_{\ell,+},\varnothing).
\]
The unchanged identification branch has
$K_{\mathrm{par},b}=K_0$, but its shell key $K(S_b)$ is the all-labeled
materialization of $K_0$ and need not have the same serialization.  A
nonmaximal branch similarly declares the anonymous fragments that it does not
identify to be distinct shell elements.

This is the mathematical interpretation used in the shell formula.  The
MATLAB helper does not allocate a separate all-labeled \(K(S_b)\)
object: after the identifications have been encoded in its branch state, it
passes the remaining polarity arrays to the compatible-core routine.  Giving
the survivors fresh distinct labels states explicitly which ownership
convention that computation represents.

A branch is summarized by its resulting shell key, multiplicity, and parity:
\begin{equation}
  b=(K',m,\varepsilon),
  \label{eq:implementation-branch}
\end{equation}
where $K'=K(S_b)$ is the shell interface key, $m$ the combinatorial branch
multiplicity, and $\varepsilon$ the parity change induced by the corresponding
row swaps. We then compute compute the excess, sign, and compatible-core contribution for $S_b$. 
Every admissible identification choice, including the choice to make no identifications, is included in the sum.
\paragraph{Example.}
Consider the one-column partial shell
\[
  \widetilde S=
  \begin{pmatrix}
    \alpha\\
    \alpha\\
    \beta\\
    \beta\\
    a\\
    a
  \end{pmatrix},
\]
where $a$ is labeled and $\alpha,\beta$ are anonymous.  The fragment $a$
occupies rows $56$, while $\alpha$ and $\beta$ occupy rows $12$ and $34$.
Thus the initial partial-shell interface key is
\[
  K_0=(\{56\},\{12,34\},\varnothing,\varnothing).
\]
These three minus fragments occupy disjoint row pairs that together cover all
six rows.  They may therefore be identified as one completed element, which
inherits the label $a$.  Since completed elements are omitted from an
interface key, both the residual partial-shell key and the resulting shell key
are
\[
  K'=(\varnothing,\varnothing,\varnothing,\varnothing).
\]
There is one occurrence-level choice and its parity is $+1$.  Thus
\eqref{eq:implementation-branch} gives
\[
  b=(K',1,+1).
\]
The corresponding identification is
\[
  (\alpha\text{ on }12,\;\beta\text{ on }34,\;a\text{ on }56)
  \longmapsto a.
\]

For each branch, \(m\) and \(\varepsilon\) multiply the contribution computed
from its materialized shell.  The quantities \(k_{S_b}\) and
$Weight_{Cores}(S_b)$ 
belong to that shell. 
The next section
describes the determinant-sign comparison used when the complete branch sum
is reused from the outer dictionary.

\subsection{Row-orbit cache and sign transport}
\label{sec:implementation-signs}
\label{sec:implementation-canonicalization}

\subsubsection{Sign convention and cache transport}

The underlying table sign is defined as \(\sgn(T)=\prod_{j=1}^{6}\sgn(\pi_j)\).
\Cref{def:canonical-comparison-table,def:shellcoresign} then define the
canonical comparison tables \(\operatorname{canon}(S)\) and
\(\operatorname{canon}(C)\) and the relative signs
\(\sgn(S)\) and \(\sgn(C)\).  Their factorization is stated in
\cref{cor:shell-core-relative-sign}.

The implementation evaluates the same signs by applying one recorded common
ordering to the element names and multiplying the six row inversion signs.
A common relabeling or common column permutation changes every row by the
same parity and hence contributes a sixth power, so it does not change the
completed-table sign.  The symbols \(+\) and \(-\) on fragments remain support
polarities and have no determinant-sign meaning.

For a shell $S_b$ 
with interface key \(K\) and a compatible
completed paired core \(C\), write
\[
  \varepsilon(C\mid K)\in\{\pm1\}
\]
for the relative core sign in the common element-ordering convention.  The
partial shell parity is denoted by
\(\varepsilon_{\widetilde S}\), while the branch parity is
\(\varepsilon\) in \eqref{eq:implementation-branch}.  These signs are computed
on a cache miss and are included in the branch sum.

\subsubsection{Relative sign under row-orbit reuse}

The outer support key is row-canonicalized before any identification branch is
enumerated.  Suppose a current literal partial shell \(\widetilde S\) and the
first stored literal representative \(\widetilde S_{\mathrm{ref}}\) have the
same row-orbit key \(\widehat K_{\mathrm{par}}\).  The dictionary routine
chooses a row permutation \(\pi\) carrying the current support key to the
stored support key and compares the two literal tables in the common
element-ordering convention.  The data used by the lookup are
\begin{equation}
  (\widehat K_{\mathrm{par}},\rho_{\mathrm{rel}},\pi),
  \qquad
  \rho_{\mathrm{rel}}\in\{\pm1\},\quad \pi\in S_6,
  \label{eq:implementation-canonical-parity}
\end{equation}
where \(\rho_{\mathrm{rel}}\) is the relative dictionary sign computed after
row transport and conversion to the stored element-ordering convention.  On a
cache hit, the driver defines the expression used for the current partial
shell by applying this one factor to the complete cached branch/core
aggregate:
\begin{equation}
  \mathcal A_{\mathrm{use}}
    (\widetilde S;\widetilde S_{\mathrm{ref}})
  :=
  \rho_{\mathrm{rel}}\,
  \mathcal A_{\mathrm{miss}}(\widetilde S_{\mathrm{ref}}).
  \label{eq:implementation-outer-cache-sign-transport}
\end{equation}
Thus the relative sign is part of a cache hit, not another identification
branch or another coefficient.  It is applied exactly once.  The case
calculation continues to carry its local factors and later normalization
outside this reuse.


\subsubsection{Exhaustive row-orbit normalization}

The row group has only \(6!=720\) elements, so the executable examines the
complete row orbit of each newly encountered support serialization.  For each
\(\pi\in S_6\), it:
\begin{enumerate}
  \item Applies \(\pi\) to every stored row pair;
  \item Sorts the two entries of each resulting pair;
  \item Retains the four labeled/anonymous and minus/plus blocks; and
  \item Sorts the entries within each block and serializes the result.
\end{enumerate}
The lexicographically least string is
\(\widehat K_{\mathrm{par}}\).  The auxiliary map \texttt{key\_dict} stores
the association between an encountered support string and this row-orbit key.
For the first literal partial shell under that key, the executable evaluates
the complete aggregate and stores both that value and the literal reference
table, as displayed in \eqref{eq:implementation-outer-cache}.
Later literal partial shells with the same row-orbit key use
\eqref{eq:implementation-outer-cache-sign-transport}.  No automorphism group
of a materialized final shell is computed here. 

\subsubsection{Example: a negative row-transport sign}
Four-mark case~17 supplies two different 
partial shells,
\begin{equation}
  \widetilde S_{\mathrm{ref}}=
  \begin{pmatrix}
    -1& 3&-4\\
    -1& 3& {\times}2\\
     {\times}1& 3&-3\\
     2&-2&-4\\
     2&-2& {\times}3\\
     2& {\times}1&-3
  \end{pmatrix},
  \qquad
  \widetilde S=
  \begin{pmatrix}
    -1&-2& {\times}2\\
    -1& 3&-4\\
     {\times}1& 3&-3\\
     2&-2& {\times}3\\
     2& {\times}1&-3\\
     2& 3&-4
  \end{pmatrix}.
  \label{eq:implementation-case17-cache-tables}
\end{equation}
Positive integers denote labeled covered elements.  Negative integers are
temporary names for anonymous pairs, not determinant signs.  Suppressing the
equality-class identifiers of labeled fragments, their support projections are
\begin{align}
  K_{\mathrm{par}}(\widetilde S_{\mathrm{ref}})
    &=\bigl(\{36\},\{12,14,36,45\},\{13,46\},\varnothing\bigr),
    \label{eq:implementation-case17-key-reference}\\
  K_{\mathrm{par}}(\widetilde S)
    &=\bigl(\{35\},\{12,14,26,35\},\{15,23\},\varnothing\bigr).
    \label{eq:implementation-case17-key-current}
\end{align}
The regression implementation selects the row reordering
\begin{equation}
  \pi=(2,6,3,1,4,5),
  \qquad
  T_\pi A=A[\pi,:].
  \label{eq:implementation-case17-row-action}
\end{equation}
Thus an occurrence on old row \(r\) moves to position \(\pi^{-1}(r)\).
On the three nonempty support blocks of
\(K_{\mathrm{par}}(\widetilde S)\), respectively, this gives
\[
  35\longmapsto36,\qquad
  \{12,14,26,35\}\longmapsto\{12,14,36,45\},\qquad
  \{15,23\}\longmapsto\{13,46\}.
\]
Hence \(T_\pi\) carries the current support key to the stored reference key.
This fact alone does not determine the relative sign, because each literal
table is first put into its own canonical element convention.

After canonicalizing element names within the two literal tables, the
comparison arrays produced by the implementation are
\begin{equation}
  A_{\widetilde S}=
  \begin{pmatrix}
    -1&-2& 2\\
    -1& 3&-4\\
     3&-3& 1\\
    -2& 3& 2\\
     2&-3& 1\\
     3& 2&-4
  \end{pmatrix},
  \qquad
  A_{\mathrm{ref}}=
  \begin{pmatrix}
    -1&-4& 3\\
    -1& 2& 3\\
    -3& 1& 3\\
     2&-4&-2\\
     2&-2& 3\\
     2&-3& 1
  \end{pmatrix}.
  \label{eq:implementation-case17-comparison-tables}
\end{equation}
Applying \(\pi\) to \(A_{\widetilde S}\) and then relabeling in the reference
convention gives
\begin{equation}
  A_{\widetilde S\to\mathrm{ref}}=
  \begin{pmatrix}
    -4& 3&-1\\
     3& 2&-1\\
     3&-3& 1\\
    -4&-2& 2\\
    -2& 3& 2\\
     2&-3& 1
  \end{pmatrix}.
  \label{eq:implementation-case17-transported-table}
\end{equation}

For any six-row array \(B\) written in a fixed element-ordering convention,
let
\[
  \chi(B)
  =
  (-1)^{\sum_{r=1}^{6}\operatorname{inv}(B_r)}
\]
be the product of its six row inversion signs.  For the two literal records
in \eqref{eq:implementation-case17-cache-tables}, this gives
\begin{align*}
  \chi(\widetilde S_{\mathrm{ref}})
  &=(-1)^{2+1+2+3+1+3}=(-1)^{12}=+1,\\
  \chi(\widetilde S)
  &=(-1)^{1+2+2+1+3+2}=(-1)^{11}=-1.
\end{align*}
Thus their displayed-array inversion parities are opposite.  These are
auxiliary parities of the literal partial-shell arrays, not yet signs of
individual compatible cores.
Once an identification branch and a compatible paired core \(C\) have been
fixed, the relative core sign is computed by the same row-inversion
comparison:
\[
  \varepsilon(C\mid K)
  =
  \chi(C)\,
  \chi\bigl(\operatorname{canon}(C)\bigr)
  =
  (-1)^{
    \sum_{r=1}^{6}
    \left(
      \operatorname{inv}(C_r)
      +
      \operatorname{inv}(\operatorname{canon}(C)_r)
    \right)
  },
\]
with both arrays written in the common element-ordering convention.  The
complete row-by-row calculation for the two literal arrays and the comparison
arrays used in the cache transport is shown in
\cref{tab:implementation-case17-inversions}.
\begin{table}[htbp]
\centering
\caption{Row inversion calculation for the negative case-17 transport sign.}
\label{tab:implementation-case17-inversions}
\small
\begin{tabular}{@{}lccr@{}}
\toprule
Array & Row inversion counts & Row signs & \(\chi\)\\
\midrule
\(A_{\widetilde S}\)
  & \((1,2,2,1,2,3)\) & \((-,+,+,-,+,-)\) & \(-1\)\\
\(\widetilde S\)
  & \((1,2,2,1,3,2)\) & \((-,+,+,-,-,+)\) & \(-1\)\\
\(\widetilde S_{\mathrm{ref}}\)
  & \((2,1,2,3,1,3)\) & \((+,-,+,-,-,-)\) & \(+1\)\\
\(A_{\mathrm{ref}}\)
  & \((1,0,0,2,1,2)\) & \((-,+,+,+,-,+)\) & \(+1\)\\
\(A_{\widetilde S\to\mathrm{ref}}\)
  & \((1,3,2,0,1,2)\) & \((-,-,+,+,-,+)\) & \(-1\)\\
\bottomrule
\end{tabular}
\end{table}

The current shell agrees in relative sign with its own comparison array,
since
\(\chi(A_{\widetilde S})\chi(\widetilde S)=(-1)(-1)=+1\).
Likewise,
\(\chi(A_{\mathrm{ref}})\chi(\widetilde S_{\mathrm{ref}})=+1\).
The transported comparison contributes the only minus sign:
\[
  \chi(A_{\mathrm{ref}})
  \chi(A_{\widetilde S\to\mathrm{ref}})
  =(+1)(-1)=-1.
\]
Multiplying these three comparisons, the two occurrences of
\(\chi(A_{\mathrm{ref}})\) cancel because their square is \(1\).  Therefore
the four-factor expression used by the dictionary routine is
\begin{align}
  \rho_{\mathrm{rel}}
  &=
  \chi(A_{\widetilde S})\,
  \chi(\widetilde S)\,
  \chi(\widetilde S_{\mathrm{ref}})\,
  \chi(A_{\widetilde S\to\mathrm{ref}})
  \notag\\
  &=(-1)(-1)(+1)(-1)=-1.
  \label{eq:implementation-case17-cache-sign}
\end{align}
Because each record in
\eqref{eq:implementation-case17-cache-tables} emits several identification
branches and compatible cores, neither record has one standalone core sign.
The corresponding statement at the level of the complete signed core
calculation is nevertheless exact.  A cache-free run on each record emits
twenty branches and gives
\begin{equation}
  \begin{aligned}
  \mathcal A_{\mathrm{miss}}(\widetilde S_{\mathrm{ref}})
    &=2R_1(t)+64R_2(t)+1200R_3(t)+11520R_4(t),\\
  \mathcal A_{\mathrm{miss}}(\widetilde S)
    &=-2R_1(t)-64R_2(t)-1200R_3(t)-11520R_4(t),
  \end{aligned}
  \label{eq:implementation-case17-direct-aggregates}
\end{equation}
where the derivative blocks \(R_k(t)\) are defined in
\eqref{eq:implementation-derivative-block}.  In particular, the single
unchanged branch with \(k=4\) has net signed compatible-core weights
\(+11520\) and \(-11520\), respectively.  Thus the directly recomputed
branch/core aggregates have opposite signs, exactly as
\(\rho_{\mathrm{rel}}=-1\) predicts.

The two row-isomorphic records may therefore share one outer row-orbit cache
entry, but the contribution of \(\widetilde S\) must acquire a minus sign
relative to the stored aggregate.

The displayed matrices are initial literal partial-shell records, and
\(\widetilde S_{\mathrm{ref}}\) is the first cached representative rather than
an asserted lexicographic minimum among literal tables.  The executable
applies the relative sign \(-1\) to the complete cached branch/core aggregate;
it does not repeat this comparison after each identification branch.  The
example therefore isolates the sign calculation performed on an outer-cache
hit.

The next section describes the compatible-core calculation performed for
each materialized branch on a cache miss.

\subsection{Computing the signed compatible-core weight}
\label{sec:implementation-core-computation}
\label{sec:implementation-core-dictionary}
\label{sec:implementation-core-enumeration}

For each identification branch $b$, it remains to compute $Weight_{Cores}(S)$. 
At this point every surviving fragment represents a
distinct shell element, as explained in
\cref{sec:implementation-identifications}.  The implementation therefore
passes only the fifteen lists of support polarities to the compatible-core
calculation and introduces one fresh placeholder for each fragment.  The
branch multiplicity, accumulated shell sign, inherited symmetry factor, and
local factors \(G(c)\), as well as the derivative block, remain outside this
calculation.

Before carrying out that calculation, the accelerated implementation
serializes the fifteen residual polarity lists and queries the inner
compatible-core dictionary.  Because all surviving fragments now represent
distinct elements, their temporary placeholder names are immaterial to this
lookup.  A hit reuses the stored net weight; a miss evaluates $Weight_{Cores}(S)$
and stores it for later branches with the same
residual support configuration.  This is distinct from the outer dictionary,
whose value is the complete sum over every identification branch of a partial
shell.

The core-weight dictionary is the runtime cache; its miss evaluator consults
a separate fixed table of pair-demand decompositions.  The compatible-core
machinery therefore has an offline precomputation and a per-branch evaluation
stage.  The precomputation tabulates how multisets of row-pair demands can be
divided into paired-core columns.  On a core-weight miss, the per-branch stage
completes the complementary rows of the fresh placeholders, queries the fixed
table, sums the signed multiplicities, and returns the value to be stored in
the runtime cache.

This inner hit--miss flow is summarized in
\cref{fig:implementation-core-weight-dictionary}.

\subsubsection{Precomputing pair-demand decompositions}
\label{sec:implementation-column-decompositions}

A paired core column is a perfect matching of the six rows: it consists of three disjoint row pairs whose union is $[6]$.  There are $(6-1)!!=15$ such matching types.  Denote them by
\[
  P_1<P_2<\cdots<P_{15},
\]
where the three pairs within each $P_j$ are first sorted in the fixed row-pair order and the resulting triples are then ordered lexicographically.
For example,
\[
  \begin{aligned}
    P_1&=\{12,34,56\},\qquad
    P_2=\{12,35,46\},\qquad
    P_3=\{12,36,45\},\\
    P_4&=\{13,24,56\},\qquad
    \ldots,\qquad
    P_{15}=\{16,25,34\}.
  \end{aligned}
\]
Thus \(P_1,P_2,P_3\) exhaust the matchings beginning with the pair \(12\);
the lexicographic ordering then continues with the matchings beginning with
\(13\), and so on.

In the implementation, these matching types are used to build a finite
lookup table before any compatible-core weight is evaluated.  Adjoin a blank
symbol \(P_{16}=\varnothing\), placed after the fifteen matching types.  Six
nested loops make the nondecreasing choices
\[
  1\leq i_1\leq i_2\leq\cdots\leq i_6\leq16,
\]
with the all-blank choice omitted.  After the blank symbols are deleted, one
obtains a multiset of \(h\) matching types with \(1\leq h\leq6\).  The
nondecreasing order ensures that every such multiset is generated exactly
once.

For each choice, concatenate the three row pairs in every selected matching
and sort the resulting \(3h\) pair occurrences in the fixed row-pair order.
The resulting untagged multiset \(Q\) is the lookup key.  Let \(q_p\) be the
multiplicity of \(p\in\mathcal R_2\) in \(Q\), and let
\(d=(d_1,\ldots,d_{15})\in\mathbb Z_{\geq0}^{15}\) record how many selected
columns have each matching type.  The entries stored under \(Q\) are exactly
the vectors \(d\) satisfying
\begin{equation}
  q_p=\sum_{j=1}^{15}d_j\,\mathbf 1_{\{p\in P_j\}}
  \quad(p\in\mathcal R_2),
  \qquad
  1\leq h=\sum_{j=1}^{15}d_j\leq6.
  \label{eq:implementation-untagged-decomposition}
\end{equation}
Write \(\mathcal D(Q)\) for this stored list of decomposition vectors.  Every
stored key necessarily satisfies \(|Q|=3h\), and every row occurs in exactly
\(h\) of its pair occurrences:
\[
  \sum_{p\ni r}q_p=h
  \qquad(r\in[6]).
\]
Consequently, a demand multiset failing either condition has no lookup-table
entry.

Alongside each decomposition, the precomputation stores the multiplicity
\[
  \nu(d,Q)
  =
  \frac{h!}{\prod_{j=1}^{15}d_j!}
  \prod_{p\in\mathcal R_2}q_p!.
\]
The first factor counts the distinct placements of the multiset of matching
types into \(h\) ordered core-column positions.  The second counts
permutations among repeated occurrences of the same row-pair support.  After
the fresh placeholders are introduced, these become their assignments to the
corresponding slots.  Thus a later sign calculation may evaluate one
representative of such a class and multiply by its stored exact cardinality.

\paragraph{Example: one support key with two decompositions.}
The following \(h=4\) example illustrates every step of the precomputation.
\begin{enumerate}
  \item Two choices made by the six nondecreasing loops are
  \[
    (1,3,5,9,16,16)
    \qquad\text{and}\qquad
    (1,2,6,8,16,16).
  \]
  After the two blank entries are deleted, their matching multisets are
  \begin{align*}
    \mathcal P_A
    &=
    \bigl\{
      \{12,34,56\},
      \{12,36,45\},
      \{13,25,46\},
      \{14,26,35\}
    \bigr\},\\
    \mathcal P_B
    &=
    \bigl\{
      \{12,34,56\},
      \{12,35,46\},
      \{13,26,45\},
      \{14,25,36\}
    \bigr\}.
  \end{align*}

  \item Concatenating the pairs in either line and sorting them gives the
  same support multiset
  \[
    Q
    =
    \{12,12,13,14,25,26,34,35,36,45,46,56\}.
  \]
  Its serialized lookup key is
  \[
    \texttt{121213142526343536454656}.
  \]

  \item The entry under this key retains two decomposition vectors.  Their
  nonzero coordinates are
  \[
    d^{(A)}_1=d^{(A)}_3=d^{(A)}_5=d^{(A)}_9=1,
    \qquad
    d^{(B)}_1=d^{(B)}_2=d^{(B)}_6=d^{(B)}_8=1.
  \]
  Equivalently, the stored split strings are
  \[
    \begin{aligned}
      s_A&=\texttt{123456123645132546142635},\\
      s_B&=\texttt{123456123546132645142536}.
    \end{aligned}
  \]

  \item In either decomposition the four matching types are distinct, so
  the column-placement factor is \(4!=24\).  The pair \(12\) occurs twice
  in \(Q\), while every other pair occurs once, so the repeated-support
  factor is \(2!=2\).  Therefore
  \[
    \nu(d^{(A)},Q)
    =
    \nu(d^{(B)},Q)
    =
    24\cdot2
    =
    48.
  \]
\end{enumerate}
A direct run of the precomputation returned the displayed lookup key, exactly
the two split strings \(s_A,s_B\), and the value \(48\) for each.  Thus
sorting the pair demands does not discard alternative ways of grouping them
into core columns; it places both decompositions under one key for the signed
calculation in the next subsection.

This lookup deliberately omits placeholder names; it records only which three
row pairs share each core column.  The all-blank case is not stored, and the
caller handles the empty core separately by setting
\(W_{\mathrm{core}}(\varnothing)=1\).  The present table covers only
\(1\leq h\leq6\); a computation requiring more than six core columns would
require the precomputation to be extended.  The next stage introduces the
fresh placeholders and evaluates the corresponding core signs.

\subsubsection{The signed compatible-core weight}
\label{sec:implementation-signed-core-weight}

After an identification branch has been fixed, scan its fifteen row-pair
lists in the order \(12,13,\ldots,56\) and give every residual fragment a
fresh consecutive placeholder.  Literal element names are no longer needed.
A plus fragment \(a^+_{ij}\), which occupies the four rows complementary to
\(ij\), has the single forced core demand \(ij\).  A minus fragment
\(a^-_{ij}\), which occupies rows \(ij\), must be completed on the other four
rows; their three perfect matchings give the three possible pairs of core
demands.  Let \(\Omega_b\) be the Cartesian product of these three choices
over the minus fragments of branch \(b\).  Plus-fragment demands introduce no
choice.

For \(\omega\in\Omega_b\), concatenate all selected pair demands and sort
them in the fixed row-pair order.  Denote the resulting untagged multiset by
\(Q_\omega\).  The lookup from \cref{sec:implementation-column-decompositions}
returns every \(d\in\mathcal D(Q_\omega)\) together with
\(\nu(d,Q_\omega)\).  For each returned decomposition, the implementation
constructs one representative paired core \(C_{\omega,d}\): for each support
\(p\), queue the fresh placeholders whose chosen completion contains \(p\);
then traverse the pair slots of the stored column decomposition and consume
the next placeholder from the corresponding queue.

Independently, the same fresh placeholders form a fixed comparison array
\(C_{0,b}\).  Each placeholder is placed once in every row missing from its
shell fragment, with row positions filled in the fixed scan order.  As in
\cref{def:canonical-comparison-table}, this array is used only for parity and
need not itself display the selected pairings.  With \(\chi\) denoting the
product of the six row inversion signs from
\cref{sec:implementation-signs}, the representative relative sign is
\[
  \varepsilon_{\omega,d}
  =
  \chi(C_{\omega,d})\chi(C_{0,b}).
\]
The compatible-core routine returns
\begin{equation}
  W_{\mathrm{core}}(S_b)
  =
  \sum_{\omega\in\Omega_b}
  \ \sum_{d\in\mathcal D(Q_\omega)}
  \nu(d,Q_\omega)\,\varepsilon_{\omega,d}.
  \label{eq:implementation-core-weight-dictionary-sum}
\end{equation}
If \(Q_\omega\) has no dictionary entry, that completion contributes zero.
The caller assigns the empty residual record the value
\(W_{\mathrm{core}}(\varnothing)=1\).

Only the representative \(C_{\omega,d}\) is constructed.  The factor
\(\nu(d,Q_\omega)\) counts the other ordered column placements and
equal-support placeholder assignments represented by the same decomposition.
They have the same relative sign: a common permutation of the core columns
acts in all six rows, while exchanging two placeholders on the same pair
support exchanges them in two rows.  Both changes preserve the total parity.
Thus \eqref{eq:implementation-core-weight-dictionary-sum} is the compressed
evaluation of the net minimal-core weight defined in
\cref{def:netweight}; the routine does not enumerate the individual cores
counted by \(\nu(d,Q_\omega)\).

\paragraph{Example.}
Suppose a branch leaves one minus fragment \(\alpha^-_{12}\) and one plus
fragment \(b^+_{12}\), with distinct fresh placeholders \(\alpha\) and \(b\).
The plus fragment forces the core demand \(12\), while the four missing rows
of \(\alpha\) have the three pairings
\[
  \{34,56\},\qquad \{35,46\},\qquad \{36,45\}.
\]
The three completion choices therefore give
\begin{align*}
  &\{(b,12),(\alpha,34),(\alpha,56)\},\\
  &\{(b,12),(\alpha,35),(\alpha,46)\},\\
  &\{(b,12),(\alpha,36),(\alpha,45)\}.
\end{align*}
After the placeholders are suppressed, each line is one perfect matching of
the six rows.  Its dictionary entry has one single-column decomposition with
\(\nu=1\), and the representative sign is \(+1\).  Hence
\[
  W_{\mathrm{core}}(S_b)=1+1+1=3.
\]
The three paired cores have the same displayed element-name column
\((b,b,\alpha,\alpha,\alpha,\alpha)^{\mathsf T}\), but they remain distinct
because, as specified in \cref{def:mincompatible}, the pairing data are part
of the core.  As a larger check, three minus fragments on \(12,34,56\) give
\(W_{\mathrm{core}}=48\), agreeing with the direct signed count of \(64\)
minimal cores in \cref{ex:mincoresfirstexample}.

\begin{figure}[H]
\centering
\begin{tikzpicture}[
  cache stage/.style={
    draw=diagramgray,
    fill=diagramlightblue,
    rounded corners=2pt,
    align=center,
    inner xsep=8pt,
    inner ysep=5pt,
    font=\small
  },
  cache query/.style={
    diamond,
    aspect=2.25,
    draw=diagramgray,
    fill=diagramlightorange,
    align=center,
    inner xsep=4pt,
    inner ysep=3pt,
    font=\small
  },
  fixed table/.style={
    cache stage,
    fill=white,
    densely dashed
  },
  cache flow/.style={-{Latex[length=2mm]},draw=diagramblue,thick},
  cache branch/.style={font=\footnotesize,fill=white,inner sep=1.5pt}
]
  \node[cache stage] (residual)
    {serialize the residual\\polarity-list key};
  \node[cache query,below=8mm of residual] (query)
    {query core-weight\\dictionary};
  \node[cache stage,below left=12mm and 18mm of query] (hit)
    {reuse the stored\\\(W_{\mathrm{core}}(S_b)\)};
  \node[cache stage,below right=12mm and 18mm of query] (demands)
    {enumerate completion choices\\\(\omega\in\Omega_b\)};
  \node[fixed table,below=7mm of demands] (pairtable)
    {fixed pair-demand table\\retrieve
     \(\mathcal D(Q_\omega)\) and \(\nu(d,Q_\omega)\)};
  \node[cache stage,below=7mm of pairtable] (signed)
    {construct representative cores and sum\\
     \(\nu(d,Q_\omega)\varepsilon_{\omega,d}\)};
  \node[cache stage,below=7mm of signed] (store)
    {store and return\\\(W_{\mathrm{core}}(S_b)\)};

  \draw[cache flow] (residual) -- (query);
  \draw[cache flow] (query) -- node[cache branch,above left] {hit} (hit);
  \draw[cache flow] (query) -- node[cache branch,above right] {miss} (demands);
  \draw[cache flow] (demands) -- (pairtable);
  \draw[cache flow] (pairtable) -- (signed);
  \draw[cache flow] (signed) -- (store);
\end{tikzpicture}
\caption{Inner core-weight dictionary lookup for a materialized shell \(S_b\).
A hit returns the stored net signed number of minimal compatible cores.  On a
miss, the routine queries the fixed pair-demand table, sums the signed
decomposition multiplicities, and stores \(W_{\mathrm{core}}(S_b)\).  The
empty residual record is handled separately by
\(W_{\mathrm{core}}(\varnothing)=1\).}
\label{fig:implementation-core-weight-dictionary}
\end{figure}
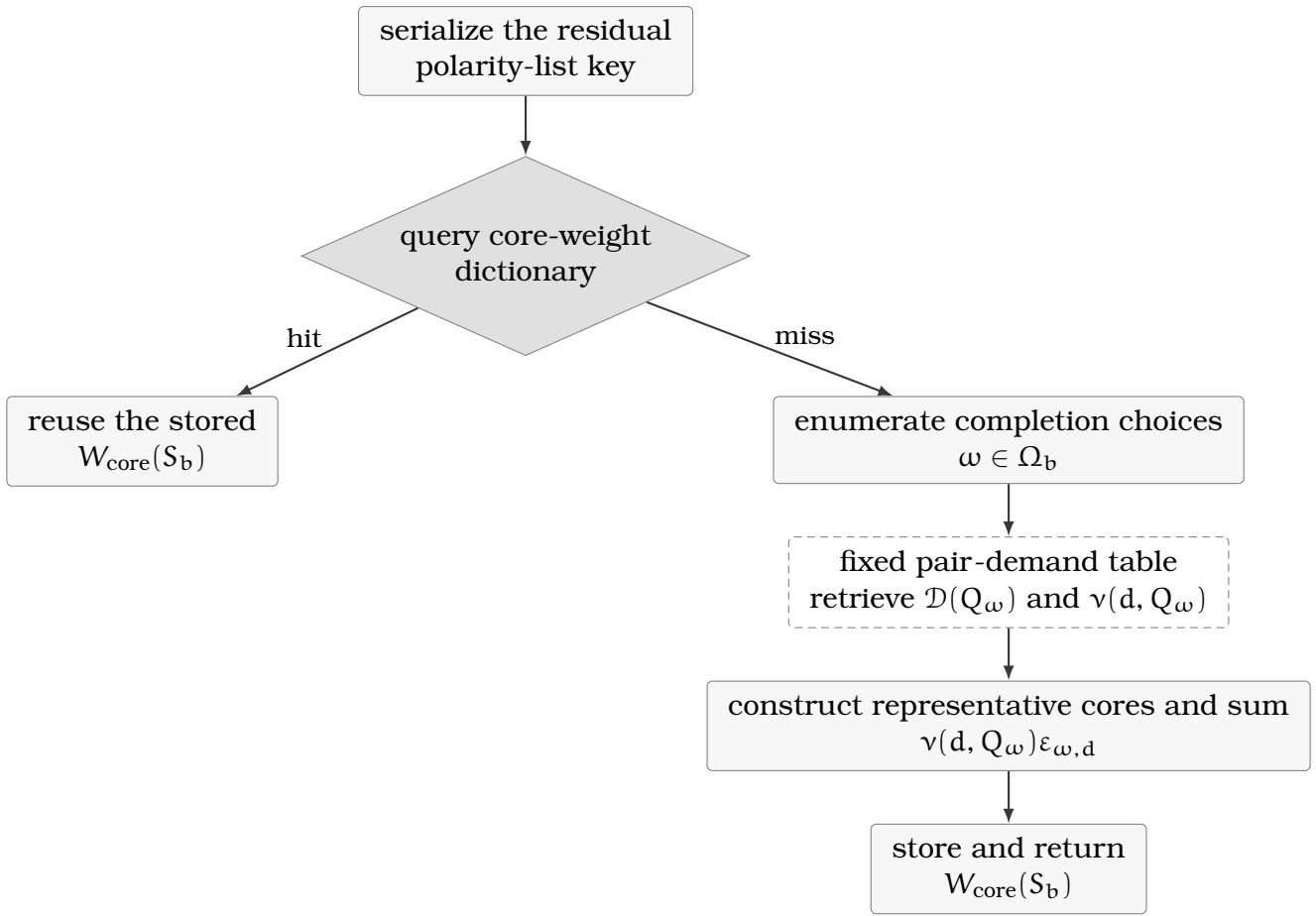

\subsubsection{Assembling and caching the partial-shell expression}
\label{sec:implementation-core-memoized-lookup}
\label{sec:implementation-output}

After the identification stage, the implementation visits each generated
branch state once.  Let \(a\) be the number of columns of the originating
literal partial shell, and for a branch \(b\) let \(e_b\) be its current
number of shell elements.  The stored branch fields also contain its
identification multiplicity \(m_b\) and one accumulated sign
\(\sigma_b\).  This sign is initialized with the literal shell sign and is
updated by the signs of the selected identification operations.  Thus the
implementation carries their product rather than two separate sign objects.
The derivative order is
\begin{equation}
  k_b=e_b-a.
  \label{eq:implementation-branch-excess}
\end{equation}
This is the shell excess from \cref{def:shell-excess}.

For each branch, the program discards the fragment-label column and passes
its fifteen polarity lists to the compatible-core routine.  If every list is
empty, it sets the core weight to \(1\); otherwise it queries the inner
compatible-core dictionary.  A cache miss evaluates $Weight_{Cores}(S)$
by
\eqref{eq:implementation-core-weight-dictionary-sum} and stores the result,
whereas a hit reuses that exact integer.  Write
\begin{equation}
  R_k(t)
  =
  \frac{\ZeroMark(t)}{\GaussianCore(\xt)}
  \frac{\xt^k\GaussianCore^{(k)}(\xt)}
       {\displaystyle\prod_{j=1}^{k}j(j+2)(j+4)},
  \qquad R_0(t)=\ZeroMark(t),
  \label{eq:implementation-derivative-block}
\end{equation}
where the product is \(1\) when \(k=0\).  This is the floating-factor and
derivative block in the shell formula of
\cref{thm:Ftformula,lem:coretotalweight}.  The program appends the
branch term
\[
  m_b\,\sigma_b\,W_{\mathrm{core}}(S_b)\,R_{k_b}(t)
\]
to the partial-shell expression.  Hence the complete value computed on an
outer-cache miss is
\begin{equation}
  \mathcal A_{\mathrm{miss}}(\widetilde S)
  =
  \sum_{b\in\mathcal B(\widetilde S)}
    m_b\,
    \sigma_b\,
    W_{\mathrm{core}}(S_b)\,
    R_{k_b}(t).
  \label{eq:implementation-outer-aggregate}
\end{equation}
Here \(\mathcal B(\widetilde S)\) is the list of branch states produced by the
identification loops.  Thus the returned parenthesized expression already
contains every branch multiplicity, accumulated sign, signed core weight, and
excess-dependent derivative block.

On the first occurrence of a row-orbit support key, the caller stores this
complete expression together with the corresponding literal partial shell.
On a later occurrence it bypasses the identification and core routines,
retrieves the stored expression, and applies the single relative sign from
\eqref{eq:implementation-outer-cache-sign-transport}.  The current product of
local column factors \(\prod_cG(c)\) is attached afterward, and repeated
output records with the same constructed key are grouped by an integer
multiplicity.  The mark-position and partial-shell symmetry factor from
\cref{sec:implementation-paper-symmetries}, which realizes the orbit
normalization in \cref{thm:Ftformula}, belongs to the later
aggregation over cases with the same number of marks and is not part of the
cached compatible-core expression.


\appendix
\section{Explicit Analysis for One Mark}\label{sec:explicitonemark}
In this appendix, we use our framework to explicitly compute the total contribution of tables with one mark to $F_6(t)$. For one mark, there is a single mark placement table $P$. For this table $P$, $\frac{6!}{|Aut(P)|} = 6$. 
\begin{center}
\begin{tabular}{|c|}
\hline
$\times a$ \\ \hline
  \\ \hline
  \\ \hline
  \\ \hline
  \\ \hline
  \\ \hline
\end{tabular}
\end{center}
Before enumerating the possible shells $S$ which are consistent with $P$, we first compute the total weight $Weight_{Cores}(S)$ of the minimum compatible cores for these shells. When there is one mark, there are two cases which appear.
\begin{enumerate}
\item $S$ has no $+$ or $-$ fragments. In this case, the unique minimum compatible core is the empty core with no columns. 
\item $S$ has an $a_{ij}^{+}$ fragment and a $b_{ij}^{-}$ fragment for some distinct elements $a,b \in [n]$ and rows $i < j \in [6]$ (in this case, $a$ must be covered and $b$ must be uncovered, though this is irrelevant here). In this case, the minimum compatible cores consist of one column which has an $a$ in rows $i,j$ and a $b$ in rows $[6] \setminus \{i,j\}$. There are three such minimum compatible cores $C$ as there are three possible pairings for the four $a$ and all of these cores have positive sign. Thus, $Weight_{Cores}(S) = 3$.

Below, we illustrate the table for these minimum compatible cores when $i = 5$ and $j = 6$.
\begin{center}
\begin{tabular}{|c|}
\hline
$b$ \\ \hline
$b$ \\ \hline
$b$ \\ \hline
$b$ \\ \hline
$a$ \\ \hline
$a$ \\ \hline
\end{tabular}
\end{center}
\end{enumerate}
\subsection{Possible Shells with One Mark}
We have the following possible shells:
\begin{enumerate}
\item \ 
\begin{center}
\begin{tabular}{|c|}
\hline
$\times a$ \\ \hline
$a$ \\ \hline
$a$ \\ \hline
$a$ \\ \hline
$a$ \\ \hline
$a$ \\ \hline
\end{tabular}
\end{center}
where the $a$ in rows 2-5 are a block of size $5$. The contribution from this case is
\[
6G(\mathbf{c})O(t) = 6{m_1}t(\mu_5 - 10\mu_3)O(t)
\]
\item There are $10$ shells isomorphic to 
\begin{center}
\begin{tabular}{|c|}
\hline
$\times a$ \\ \hline
$a$ \\ \hline
$a$ \\ \hline
$a$ \\ \hline
\cellcolor[HTML]{FD6864}$a$\\ \hline
\cellcolor[HTML]{FD6864}$a$\\ \hline
\end{tabular}
\end{center}
which are consistent with $P$. The contribution from these shells is
\[
60G(\mathbf{c})O(t) = \frac{60m_1{\mu_3}t}{(1-(\mu_4 - 3)t)(1 + {\mu_3^2}t)}O(t)
\]
\item There are $10$ shells isomorphic to 
\begin{center}
\begin{tabular}{|c|}
\hline
$\times a$ \\ \hline
$a$ \\ \hline
$a$ \\ \hline
$a$ \\ \hline
$\beta$ \\ \hline
$\beta$ \\ \hline
\end{tabular}
\end{center}
which are consistent with $P$. The contribution from these shells is
\[
60G(\mathbf{c})\frac{O(t)}{N(t')}\cdot\frac{t'N^{(1)}(t')Weight_{Cores}(S)}{15} = \frac{12m_1{\mu_3}t}{(1-(\mu_4 - 3)t)(1 + {\mu_3^2}t)}O(t)\frac{t'N^{(1)}(t')}{N(t')}.
\]
\end{enumerate}
\subsection{Final Result for One Mark}
The final result for one mark is 
\begin{align*}
&= 6{m_1}t(\mu_5 - 10\mu_3)O(t) + \frac{60m_1{\mu_3}tO(t)}{(1-(\mu_4 - 3)t)(1 + {\mu_3^2}t)} + \frac{12m_1{\mu_3}tO(t)}{(1-(\mu_4 - 3)t)(1 + {\mu_3^2}t)}\frac{t'N^{(1)}(t')}{N(t')} \\
&= 6{m_1}t(\mu_5 - 10\mu_3)O(t) + \frac{60m_1{\mu_3}tO(t)}{(1-(\mu_4 - 3)t)(1 + {\mu_3^2}t)}\left(1 + \frac{t'N^{(1)}(t')}{5N(t')}\right)
\end{align*}
Note that $1 + \frac{t'N^{(1)}(t')}{5N(t')}$ is the factor corresponding to the possible identifications when we start with an $a_{ij}^{+}$ fragment and a $u_{ij}^{-}$ fragment and this is stored in the shell dictionary.
\section{Explicit Analysis for Two Marks}\label{sec:explicittwomarks}
In this appendix, we use our framework to explicitly compute the total contribution of tables with two marks to $F_6(t)$. For two marks, there are three mark placement tables. 
\begin{enumerate}
\item Letting $P_1$ be the mark placement table shown below, $|Aut(P_1)| = 2! * 4! = 48$ so $\frac{6!}{|Aut(P_1)} = 15$. 
\begin{center}
\begin{tabular}{|c|}
\hline
$\times a$ \\ \hline
$\times a$  \\ \hline
  \\ \hline
  \\ \hline
  \\ \hline
  \\ \hline
\end{tabular}
\end{center}
\item Letting $P_2$ be the mark placement table shown below, $|Aut(P_1)| = 2! * 4! = 48$ so $\frac{6!}{|Aut(P_2)} = 15$. 
\begin{center}
\begin{tabular}{|c|}
\hline
$\times a$ \\ \hline
$\times b$  \\ \hline
  \\ \hline
  \\ \hline
  \\ \hline
  \\ \hline
\end{tabular}
\end{center}
\item Letting $P_3$ be the mark placement table shown below, $|Aut(P_3)| = 2! * 4! = 48$ so $\frac{6!}{|Aut(P_3)} = 15$. 
\begin{center}
\begin{tabular}{|c|c|}
\hline
$\times a$ & \\ \hline
 & $\times b$  \\ \hline
 & \\ \hline
 & \\ \hline
 & \\ \hline
 & \\ \hline
\end{tabular}
\end{center}
\end{enumerate}
\subsection{Minimum Compatible Cores for Shells With Two Marks}
Before enumerating the possible shells $S$ which are consistent with $P_1$, $P_2$, and $P_3$, we first compute the total weight $Weight_{Cores}(S)$ of the minimum compatible cores for these shells. When there are two marks, there are several cases which appear. Note that when discussing which cores are isomorphic to each other, we consider permutations of the rows and columns but not of the elements.
\begin{enumerate}
\item $S$ has no $+$ or $-$ fragments. In this case, the unique minimum compatible core is the empty core with no columns. 
\item $S$ has an $a_{ij}^{+}$ fragment and a $b_{ij}^{-}$ fragment for some distinct elements $a,b \in [n]$ and rows $i < j \in [6]$. In this case, as before, $Weight_{Cores}(S) = 3$.
\item $S$ has an $a_{ij}^{+}$ fragment, a $b_{ij}^{+}$ fragment, a $c_{ij}^{-}$ fragment, and a $d_{ij}^{-}$ fragment for some distinct elements $a,b,c,d \in [n]$ and rows $i < j \in [6]$. When $i = 5$ and $j = 6$, we have the following tables for the minimum compatible cores:
\begin{enumerate}
\item[1.] There are $4$ tables isomorphic to 
\begin{center}
\begin{tabular}{|c|c|}
\hline
$c$ & $d$ \\ \hline
$c$ & $d$ \\ \hline
$c$ & $d$ \\ \hline
$c$ & $d$ \\ \hline
$a$ & $b$ \\ \hline
$a$ & $b$ \\ \hline
\end{tabular}
\end{center}
as we can swap the positions of $a$ and $b$ and swap the positions of $c$ and $d$. Each of these tables has $9$ possible pairings and has positive sign.
\item[2.] There are $12$ tables isomorphic to 
\begin{center}
\begin{tabular}{|c|c|}
\hline
$c$ & $d$ \\ \hline
$c$ & $d$ \\ \hline
$d$ & $c$ \\ \hline
$d$ & $c$ \\ \hline
$a$ & $b$ \\ \hline
$a$ & $b$ \\ \hline
\end{tabular}
\end{center}
as we can swap the positions of $a$ and $b$ and choose which two of the first four rows have a $c$ in the first column. Each of these tables has $1$ possible pairing and has positive sign.
\end{enumerate}
Thus, for this case, $Weight_{Cores}(S) = 4*9 + 12*1 = 48$.
\item $S$ has an $a_{ij}^{+}$ fragment, a $b_{ik}^{+}$ fragment, a $c_{ij}^{-}$ fragment, and a $d_{ik}^{-}$ fragment for some distinct elements $a,b,c,d \in [n]$ and distinct rows $i,j,k \in [6]$. When $i = 4$, $j = 5$, and $k = 6$, we have the following tables for the minimum compatible cores:
\begin{enumerate}
\item[1.] There are $2$ tables isomorphic to 
\begin{center}
\begin{tabular}{|c|c|}
\hline
$c$ & $d$ \\ \hline
$c$ & $d$ \\ \hline
$c$ & $d$ \\ \hline
$a$ & $b$ \\ \hline
$a$ & $d$ \\ \hline
$c$ & $b$ \\ \hline
\end{tabular}
\end{center}
as we can swap the two columns. Each of these tables has $9$ possible pairings and has positive sign.
\item[2.] There are $6$ tables isomorphic to 
\begin{center}
\begin{tabular}{|c|c|}
\hline
$c$ & $d$ \\ \hline
$d$ & $c$ \\ \hline
$d$ & $c$ \\ \hline
$a$ & $b$ \\ \hline
$a$ & $d$ \\ \hline
$c$ & $b$ \\ \hline
\end{tabular}
\end{center}
as we can swap the two columns and we can permute the first three rows. Each of these tables has $1$ possible pairing and has positive sign.
\end{enumerate}
Thus, for this case, $Weight_{Cores}(S) = 2*9 + 6*1 = 24$.
\item $S$ has an $a_{ij}^{+}$ fragment, a $b_{kl}^{+}$ fragment, a $c_{ij}^{-}$ fragment, and a $d_{kl}^{-}$ for some distinct elements $a,b,c,d \in [n]$ and some rows $i,j,k,l \in [6]$ such that $i,j,k,l$ are distinct, $i < j$, and $k < l$.  When $i = 3$, $j = 4$, $k = 5$, and $l = 6$, we have the following tables for the minimum compatible cores:
\begin{enumerate}
\item[1.] There are $2$ tables isomorphic to 
\begin{center}
\begin{tabular}{|c|c|}
\hline
$c$ & $d$ \\ \hline
$c$ & $d$ \\ \hline
$a$ & $d$ \\ \hline
$a$ & $d$ \\ \hline
$c$ & $b$ \\ \hline
$c$ & $b$ \\ \hline
\end{tabular}
\end{center}
as we can swap the two columns. Each of these tables has $9$ possible pairings and has positive sign.
\item[2.] There are $2$ tables isomorphic to 
\begin{center}
\begin{tabular}{|c|c|}
\hline
$d$ & $c$ \\ \hline
$d$ & $c$ \\ \hline
$a$ & $d$ \\ \hline
$a$ & $d$ \\ \hline
$c$ & $b$ \\ \hline
$c$ & $b$ \\ \hline
\end{tabular}
\end{center}
as we can swap the two columns. Each of these tables has $1$ possible pairing and has positive sign.
\item[3.] There are $2$ tables isomorphic to 
\begin{center}
\begin{tabular}{|c|c|}
\hline
$d$ & $c$ \\ \hline
$d$ & $c$ \\ \hline
$a$ & $d$ \\ \hline
$a$ & $d$ \\ \hline
$b$ & $c$ \\ \hline
$b$ & $c$ \\ \hline
\end{tabular}
\end{center}
as we can swap the two columns. Each of these tables has $3$ possible pairings and has positive sign.
\item[4.] There are $2$ tables isomorphic to 
\begin{center}
\begin{tabular}{|c|c|}
\hline
$c$ & $d$ \\ \hline
$c$ & $d$ \\ \hline
$a$ & $d$ \\ \hline
$a$ & $d$ \\ \hline
$b$ & $c$ \\ \hline
$b$ & $c$ \\ \hline
\end{tabular}
\end{center}
as we can swap the two columns. Each of these tables has $3$ possible pairings and has positive sign.
\end{enumerate}
Thus, for this case, $Weight_{Cores}(S) = 2*9 + 2*1 + 2*3 + 2*3 = 32$.
\item $S$ has an $a_{ij}^{+}$ fragment, a $b_{kl}^{+}$ fragment, a $c_{ik}^{-}$ fragment, and a $d_{jl}^{-}$ for some distinct elements $a,b,c,d \in [n]$ and some rows $i,j,k,l \in [6]$ such that $i,j,k,l$ are distinct and $i < j$. When $i = 3$, $j = 4$, $k = 5$, and $l = 6$, we have the following tables for the minimum compatible cores:
\begin{enumerate}
\item[1.] There are $4$ tables isomorphic to 
\begin{center}
\begin{tabular}{|c|c|}
\hline
$c$ & $d$ \\ \hline
$d$ & $c$ \\ \hline
$a$ & $d$ \\ \hline
$a$ & $c$ \\ \hline
$d$ & $b$ \\ \hline
$c$ & $b$ \\ \hline
\end{tabular}
\end{center}
as we can swap the two columns and/or swap rows $1$ and $2$. Each of these tables has $1$ possible pairings and has negative sign.
\item[2.] There are $2$ tables isomorphic to 
\begin{center}
\begin{tabular}{|c|c|}
\hline
$c$ & $d$ \\ \hline
$c$ & $d$ \\ \hline
$a$ & $d$ \\ \hline
$a$ & $c$ \\ \hline
$b$ & $d$ \\ \hline
$b$ & $c$ \\ \hline
\end{tabular}
\end{center}
as we can swap the two columns. Each of these tables has $3$ possible pairings and has positive sign.
\item[3.] There are $2$ tables isomorphic to 
\begin{center}
\begin{tabular}{|c|c|}
\hline
$d$ & $c$ \\ \hline
$d$ & $c$ \\ \hline
$a$ & $d$ \\ \hline
$a$ & $c$ \\ \hline
$b$ & $d$ \\ \hline
$b$ & $c$ \\ \hline
\end{tabular}
\end{center}
as we can swap the two columns. Each of these tables has $3$ possible pairings and has positive sign.
\end{enumerate}
Thus, for this case, $Weight_{Cores}(S) = 4*(-1) + 2*3 + 2*3 = 8$.
\begin{remark}
To check that row permutations do not affect the sign of $Weight_{Cores}(S)$, observe that regardless of the order of the rows, either $a$ and $b$ are first and last in the ordering of $a,b,c,d$ and $c,d$ are in the middle or $c$ and $d$ are first and last in the ordering of $a,b,c,d$ and $a,b$ are in the middle. In either case, it takes an even number of swaps to put the rows of the table 
\begin{center}
\begin{tabular}{|c|c|}
\hline
$c$ & $d$ \\ \hline
$c$ & $d$ \\ \hline
$a$ & $d$ \\ \hline
$a$ & $c$ \\ \hline
$b$ & $d$ \\ \hline
$b$ & $c$ \\ \hline
\end{tabular}
\end{center}
in order so this table has positive sign regardless of how the rows are permuted.
\end{remark}
\item $S$ has an $a_{ij}^{+}$ fragment, a $b_{kl}^{+}$ fragment, and a $c_{mo}^{+}$ fragment for some distinct elements $a,b,c \in [n]$ and some distinct rows $i,j,k,l,m,o \in [6]$ such that $i < j$, $k < l$, and $m < o$. In this case, there is a single minimum compatible core which has positive sign so $Weight_{Cores}(S) = 1$. When $i = 1$, $j = 2$, $k = 3$, $l = 4$, $m = 5$, and $o = 6$, this core is shown below. 
\begin{center}
\begin{tabular}{|c|}
\hline
$a$ \\ \hline
$a$ \\ \hline
$b$ \\ \hline
$b$ \\ \hline
$c$ \\ \hline
$c$ \\ \hline
\end{tabular}
\end{center}
\item $S$ has an $a_{ij}^{-}$ fragment, a $b_{kl}^{-}$ fragment, and a $c_{mo}^{-}$ fragment for some distinct elements $a,b,c \in [n]$ and some distinct rows $i,j,k,l,m,o \in [6]$ such that $i < j$, $k < l$, and $m < o$. In this case, as shown in Example \ref{ex:mincoresfirstexample}, $Weight_{Cores}(S) = 48$.
\end{enumerate}
We summarize the results we obtained below.
\begin{enumerate}
\item $\emptyset$: $Weight_{Cores}(S) = 1$.
\item $a_{ij}^{+}, b_{ij}^{-}$: $Weight_{Cores}(S) = 3$.
\item $a_{ij}^{+}$, $b_{ij}^{+}$, $c_{ij}^{-}$, $d_{ij}^{-}$: $Weight_{Cores}(S) = 48$.
\item $a_{ij}^{+}$, $b_{ik}^{+}$, $c_{ij}^{-}$, $d_{ik}^{-}$: $Weight_{Cores}(S) = 24$.
\item $a_{ij}^{+}$, $b_{kl}^{+}$, $c_{ij}^{-}$, $d_{kl}^{-}$: $Weight_{Cores}(S) = 32$.
\item $a_{ij}^{+}$, $b_{kl}^{+}$, $c_{ik}^{-}$, $d_{jl}^{-}$: $Weight_{Cores}(S) = 8$.
\item $a_{ij}^{+}$, $b_{kl}^{+}$, $c_{mo}^{+}$: $Weight_{Cores}(S) = 1$.
\item $a_{ij}^{-}$, $b_{kl}^{-}$, $c_{mo}^{-}$: $Weight_{Cores}(S) = 48$.
\end{enumerate}
\begin{remark}
the reason that these are the only cases which appear is that for each shell $S$, for each $i \in [6]$, the sum of the polarities of the fragments involving $i$ is the same. 
\end{remark}
\subsection{Possible Shells with Two Marks}
We are now ready to enumerate the possible shells with two marks. For simplicity, we use the following shorthand for the column factors.
\begin{enumerate}
\item We take $G({\underline{\!\times\!}\,}_3^1) = {\underline{\!\times\!}\,}_3^1$: $G(\mathbf{c}) = \frac{m_1{\mu_3}t}{(1-(\mu_4 - 3)t)(1 + {\mu_3^2}t)}$.
\item We take $G({\underline{\!\times\!}\,}_4^2) = \frac{m_1^2(\mu_4-3)t}{1-(\mu_4 - 3)t}$. 
\item We take $G({\underline{\!\times\!}\,}_2^2) = {m_1^2}t\left(\frac{1 - {\mu_3^2}t + 2{\mu_3^2}(\mu_4-3)t^2}{(1 - (\mu_4-3)t)^{3}(1 + {\mu_3^2}t)}\right)$. To split this into contributions from tables where the two marks are in the same column and contributions from tables where the two marks are in separate columns, we take $G_{together}({\underline{\!\times\!}\,}_2^2) = {m_1^2}t\left(\frac{1}{(1 - (\mu_4-3)t)^{3}}\right)$ and $G_{split}({\underline{\!\times\!}\,}_2^2) = {m_1^2}t\left(\frac{-2{\mu_3^2}t}{(1 - (\mu_4-3)t)^{2}(1 + {\mu_3^2}t)}\right)$.
\end{enumerate}
\subsubsection{Shells Consistent with $P_1$}
The following shells are consistent with $P_1$:
\begin{enumerate}
\item There are $3$ shells isomorphic to 
\begin{center}
\begin{tabular}{|c|}
\hline
$\times a$ \\ \hline
$\times a$ \\ \hline
$\alpha$ \\ \hline
$\alpha$ \\ \hline
$\beta$ \\ \hline
$\beta$ \\ \hline
\end{tabular}
\end{center}
which are compatible with $P_1$. The contribution from these shells is 
\[
15\left(3G({\underline{\!\times\!}\,}_2^2)O(t)\right)\frac{48(t')^2N^{(2)}(t')}{720N(t')} = \frac{3G({\underline{\!\times\!}\,}_2^2)O(t)(t')^2N^{(2)}(t')}{N(t')}.
\]
\item There are $6$ shells isomorphic to 
\begin{center}
\begin{tabular}{|c|}
\hline
$\times a$ \\ \hline
$\times a$ \\ \hline
$a$ \\ \hline
$a$ \\ \hline
$\beta$ \\ \hline
$\beta$ \\ \hline
\end{tabular}
\end{center}
which are compatible with $P_1$. The contribution from these shells is 
\[
15\left(6G({\underline{\!\times\!}\,}_2^2)O(t)\right)\frac{3t'N^{(1)}(t')}{15N(t')} = \frac{18G({\underline{\!\times\!}\,}_2^2)O(t)t'N^{(1)}(t')}{N(t')}.
\]
\item There are $3$ shells isomorphic to 
\begin{center}
\begin{tabular}{|c|}
\hline
$\times a$ \\ \hline
$\times a$ \\ \hline
\cellcolor[HTML]{FD6864} $\alpha$ \\ \hline
\cellcolor[HTML]{FD6864} $\alpha$ \\ \hline
\cellcolor[HTML]{67FD9A} $\alpha$ \\ \hline
\cellcolor[HTML]{67FD9A} $\alpha$ \\ \hline
\end{tabular}
\end{center}
which are compatible with $P_1$. The contribution from these shells is 
\[
15\left(3G({\underline{\!\times\!}\,}_2^2)O(t)\right)\frac{3t'N^{(1)}(t')}{15N(t')} = \frac{9G({\underline{\!\times\!}\,}_2^2)O(t)t'N^{(1)}(t')}{N(t')}.
\]
\item There are $3$ shells isomorphic to 
\begin{center}
\begin{tabular}{|c|}
\hline
$\times a$ \\ \hline
$\times a$ \\ \hline
\cellcolor[HTML]{FD6864}$a$ \\ \hline
\cellcolor[HTML]{FD6864}$a$ \\ \hline
\cellcolor[HTML]{67FD9A}$a$ \\ \hline
\cellcolor[HTML]{67FD9A}$a$ \\ \hline
\end{tabular}
\end{center}
which are compatible with $P_1$. The contribution from these shells is 
\[
15\left(3G({\underline{\!\times\!}\,}_2^2)O(t)\right) = 45G({\underline{\!\times\!}\,}_2^2)O(t).
\]
\item \ 
\begin{center}
\begin{tabular}{|c|}
\hline
$\times a$ \\ \hline
$\times a$ \\ \hline
$a$ \\ \hline
$a$ \\ \hline
$a$ \\ \hline
$a$ \\ \hline
\end{tabular}
\end{center}
The contribution from this shell is $15G({\underline{\!\times\!}\,}_4^2)O(t)$.
\item There are $3$ shells isomorphic to 
\begin{center}
\begin{tabular}{|c|c|}
\hline
$\times a$ & $\gamma$ \\ \hline
$\times a$ & $\gamma$ \\ \hline
$\alpha$ & $a$ \\ \hline
$\alpha$ & $a$ \\ \hline
$\beta$ & $a$ \\ \hline
$\beta$ & $a$ \\ \hline
\end{tabular}
\end{center}
which are compatible with $P_1$. The contribution from these shells is 
\[
15\left(3G({\underline{\!\times\!}\,}_2^2)G(\underline{4})O(t)\right)\frac{48(t')^2N^{(2)}(t')}{720N(t')} = \frac{3G({\underline{\!\times\!}\,}_2^2)G(\underline{4})O(t)(t')^2N^{(2)}(t')}{N(t')}.
\]
\item There are $6$ shells isomorphic to 
\begin{center}
\begin{tabular}{|c|c|}
\hline
$\times a$ & $\alpha$\\ \hline
$\times a$ & $\alpha$ \\ \hline
$\alpha$ & $a$ \\ \hline
$\alpha$ & $a$ \\ \hline
$\beta$ & $a$ \\ \hline
$\beta$ & $a$ \\ \hline
\end{tabular}
\end{center}
which are compatible with $P_1$. The contribution from these shells is 
\[
15\left(6G({\underline{\!\times\!}\,}_2^2)G(\underline{4})O(t)\right)\frac{3t'N^{(1)}(t')}{15N(t')} = \frac{18G({\underline{\!\times\!}\,}_2^2)G(\underline{4})O(t)t'N^{(1)}(t')}{N(t')}.
\]
\item There are $3$ shells isomorphic to 
\begin{center}
\begin{tabular}{|c|c|}
\hline
$\times a$ & $\beta$ \\ \hline
$\times a$ & $\beta$ \\ \hline
\cellcolor[HTML]{FD6864}$\alpha$ & $a$ \\ \hline
\cellcolor[HTML]{FD6864}$\alpha$ & $a$ \\ \hline
\cellcolor[HTML]{67FD9A}$\alpha$ & $a$ \\ \hline
\cellcolor[HTML]{67FD9A}$\alpha$ & $a$ \\ \hline
\end{tabular}
\end{center}
which are compatible with $P_1$. The contribution from these shells is 
\[
15\left(3G({\underline{\!\times\!}\,}_2^2)G(\underline{4})O(t)\right)\frac{3t'N^{(1)}(t')}{15N(t')} = \frac{9G({\underline{\!\times\!}\,}_2^2)O(t)G(\underline{4})t'N^{(1)}(t')}{N(t')}.
\]
\item There are $3$ shells isomorphic to 
\begin{center}
\begin{tabular}{|c|c|}
\hline
$\times a$ & $\alpha$ \\ \hline
$\times a$ & $\alpha$ \\ \hline
\cellcolor[HTML]{FD6864}$\alpha$ & $a$ \\ \hline
\cellcolor[HTML]{FD6864}$\alpha$ & $a$ \\ \hline
\cellcolor[HTML]{67FD9A}$\alpha$ & $a$ \\ \hline
\cellcolor[HTML]{67FD9A}$\alpha$ & $a$ \\ \hline
\end{tabular}
\end{center}
which are compatible with $P_1$. The contribution from these shells is 
\[
15\left(3G({\underline{\!\times\!}\,}_2^2)G(\underline{4})O(t)\right) = 45G({\underline{\!\times\!}\,}_2^2)G(\underline{4})O(t).
\]
\end{enumerate}
Observe that the total contribution from cases 1-4 is 
\[
15\left(3G({\underline{\!\times\!}\,}_2^2)O(t)\right)\left(\frac{48(t')^2N^{(2)}(t')}{720N(t')} + 3 \cdot \frac{3t'N^{(1)}(t')}{15N(t')} + 1\right).
\]
where $\left(\frac{48(t')^2N^{(2)}(t')}{720N(t')} + 3 \cdot \frac{3t'N^{(1)}(t')}{15N(t')} + 1\right)$ is the factor corresponding to the possible identifications for the partial shell 
\begin{center}
\begin{tabular}{|c|}
\hline
$\times a$ \\ \hline
$\times a$ \\ \hline
\cellcolor[HTML]{FD6864} \\ \hline
\cellcolor[HTML]{FD6864} \\ \hline
\cellcolor[HTML]{67FD9A} \\ \hline
\cellcolor[HTML]{67FD9A} \\ \hline
\end{tabular}
\end{center}
which has an $a_{12}^{-}$ fragment, a $u_{34}^{-}$ fragment, and a $u_{56}^{-}$ fragment.

Similarly, the total contribution from cases 6-9 is 
\[
15\left(3G({\underline{\!\times\!}\,}_2^2)G(\underline{4})O(t)\right)\left(\frac{48(t')^2N^{(2)}(t')}{720N(t')} + 3 \cdot \frac{3t'N^{(1)}(t')}{15N(t')} + 1\right).
\]
where $\left(\frac{48(t')^2N^{(2)}(t')}{720N(t')} + 3 \cdot \frac{3t'N^{(1)}(t')}{15N(t')} + 1\right)$ is the factor corresponding to the possible identifications for the partial shell 
\begin{center}
\begin{tabular}{|c|c|}
\hline
$\times a$ & \cellcolor[HTML]{8080FF} \\ \hline
$\times a$ & \cellcolor[HTML]{8080FF} \\ \hline
\cellcolor[HTML]{FD6864} & $a$ \\ \hline
\cellcolor[HTML]{FD6864} & $a$ \\ \hline
\cellcolor[HTML]{67FD9A} & $a$ \\ \hline
\cellcolor[HTML]{67FD9A} & $a$ \\ \hline
\end{tabular}
\end{center}
which has an $u_{12}^{-}$ fragment, a $u_{34}^{-}$ fragment, and a $u_{56}^{-}$ fragment. 
Putting everything together, the total contribution for $P_1$ is 
\[
G({\underline{\!\times\!}\,}_2^2)O(t)(1+G(\underline{4}))\left(\frac{3(t')^2N^{(2)}(t')}{N(t')} + \frac{27t'N^{(1)}(t')}{N(t')} + 45\right) + 15G({\underline{\!\times\!}\,}_4^2)O(t).
\]
\subsubsection{Shells Consistent with $P_2$}
The following shells are consistent with $P_2$:
\begin{enumerate}
\item There are $6$ shells isomorphic to 
\begin{center}
\begin{tabular}{|c|c|}
\hline
$\times a$ & $b$\\ \hline
$\times b$ & $a$ \\ \hline
$\alpha$ & $a$ \\ \hline
$\alpha$ & $a$ \\ \hline
$\beta$ & $b$ \\ \hline
$\beta$ & $b$ \\ \hline
\end{tabular}
\end{center}
which are compatible with $P_2$. Observe that these shells have negative sign so the contribution from these shells is 
\[
-15\left(6G({\underline{\!\times\!}\,}_2^2)G(\underline{3})O(t)\right)\frac{32(t')^2N^{(2)}(t')}{720N(t')} = -\frac{4G({\underline{\!\times\!}\,}_2^2)G(\underline{3})O(t)(t')^2N^{(2)}(t')}{N(t')}.
\]
\item There are $6$ shells isomorphic to 
\begin{center}
\begin{tabular}{|c|c|}
\hline
$\times a$ & $b$\\ \hline
$\times b$ & $a$ \\ \hline
\cellcolor[HTML]{FD6864}$\alpha$ & $a$ \\ \hline
\cellcolor[HTML]{FD6864}$\alpha$ & $a$ \\ \hline
\cellcolor[HTML]{67FD9A}$\alpha$ & $b$ \\ \hline
\cellcolor[HTML]{67FD9A}$\alpha$ & $b$ \\ \hline
\end{tabular}
\end{center}
which are compatible with $P_2$. Observe that these shells have negative sign so the contribution from these shells is 
\[
-15\left(6G({\underline{\!\times\!}\,}_2^2)G(\underline{3})O(t)\right)\frac{t'N^{(1)}(t')}{15N(t')} = -\frac{6G({\underline{\!\times\!}\,}_2^2)G(\underline{3})O(t)t'N^{(1)}(t')}{N(t')}.
\]
\item There are $12$ shells isomorphic to 
\begin{center}
\begin{tabular}{|c|c|}
\hline
$\times a$ & $b$\\ \hline
$\times b$ & $a$ \\ \hline
$\alpha$ & $a$ \\ \hline
$\alpha$ & $a$ \\ \hline
$a$ & $b$ \\ \hline
$a$ & $b$ \\ \hline
\end{tabular}
\end{center}
which are compatible with $P_2$. Observe that these shells have negative sign so the contribution from these shells is 
\[
-15\left(12G({\underline{\!\times\!}\,}_2^2)G(\underline{3})O(t)\right)\frac{3t'N^{(1)}(t')}{15N(t')} = -\frac{36G({\underline{\!\times\!}\,}_2^2)G(\underline{3})O(t)t'N^{(1)}(t')}{N(t')}.
\]
\item There are $6$ shells isomorphic to 
\begin{center}
\begin{tabular}{|c|c|}
\hline
$\times a$ & $b$\\ \hline
$\times b$ & $a$ \\ \hline
$b$ & $a$ \\ \hline
$b$ & $a$ \\ \hline
$a$ & $b$ \\ \hline
$a$ & $b$ \\ \hline
\end{tabular}
\end{center}
which are compatible with $P_2$. Observe that these shells have negative sign so the contribution from these shells is 
$-15\left(6G({\underline{\!\times\!}\,}_2^2)G(\underline{3})O(t)\right) = -90G({\underline{\!\times\!}\,}_2^2)G(\underline{3})O(t)$.
\item There are $12$ shells isomorphic to 
\begin{center}
\begin{tabular}{|c|c|}
\hline
$\times a$ & $b$\\ \hline
$\times b$ & $a$ \\ \hline
$\alpha$ & $a$ \\ \hline
$\beta$ & $a$ \\ \hline
$\alpha$ & $b$ \\ \hline
$\beta$ & $b$ \\ \hline
\end{tabular}
\end{center}
which are compatible with $P_2$. Observe that these shells have negative sign so the contribution from these shells is 
\[
-15\left(12G({\underline{\!\times\!}\,}_2^2)G(\underline{3})O(t)\right)\frac{8(t')^2N^{(2)}(t')}{720N(t')} = -\frac{2G({\underline{\!\times\!}\,}_2^2)G(\underline{3})O(t)(t')^2N^{(2)}(t')}{N(t')}.
\]
\item There are $12$ shells isomorphic to 
\begin{center}
\begin{tabular}{|c|c|}
\hline
$\times a$ & $b$\\ \hline
$\times b$ & $a$ \\ \hline
\cellcolor[HTML]{FD6864}$\alpha$ & $a$ \\ \hline
\cellcolor[HTML]{67FD9A}$\alpha$ & $a$ \\ \hline
\cellcolor[HTML]{FD6864}$\alpha$ & $b$ \\ \hline
\cellcolor[HTML]{67FD9A}$\alpha$ & $b$ \\ \hline
\end{tabular}
\end{center}
which are compatible with $P_2$. Observe that these shells have negative sign so the contribution from these shells is 
\[
-15\left(12G({\underline{\!\times\!}\,}_2^2)G(\underline{3})O(t)\right)\frac{t'N^{(1)}(t')}{15N(t')} = -\frac{12G({\underline{\!\times\!}\,}_2^2)G(\underline{3})O(t)t'N^{(1)}(t')}{N(t')}.
\]
\end{enumerate}
Observe that the total contribution from cases 1-4 is 
\[
-15\left(6G({\underline{\!\times\!}\,}_2^2)O(t)G(\underline{3})\right)\left(\frac{32(t')^2N^{(2)}(t')}{720N(t')} + \frac{t'N^{(1)}(t')}{15N(t')} + 2 \cdot \frac{3t'N^{(1)}(t')}{15N(t')} + 1\right).
\]
where $\left(\frac{32(t')^2N^{(2)}(t')}{720N(t')} + \frac{t'N^{(1)}(t')}{15N(t')} + 2 \cdot \frac{3t'N^{(1)}(t')}{15N(t')} + 1\right)$ is the factor corresponding to the possible identifications for the partial shell 
\begin{center}
\begin{tabular}{|c|c|}
\hline
$\times a$ & $b$\\ \hline
$\times b$ & $a$ \\ \hline
\cellcolor[HTML]{FD6864} & $a$ \\ \hline
\cellcolor[HTML]{FD6864} & $a$ \\ \hline
\cellcolor[HTML]{67FD9A} & $b$ \\ \hline
\cellcolor[HTML]{67FD9A} & $b$ \\ \hline
\end{tabular}
\end{center}
which has an $a_{56}^{+}$ fragment, a $b_{34}^{+}$ fragment, a $u_{34}^{-}$ fragment, and a a $u_{56}^{-}$ fragment.

the total contribution from cases 5-6 is 
\[
-15\left(12G({\underline{\!\times\!}\,}_2^2)O(t)G(\underline{3})\right)\left(\frac{8(t')^2N^{(2)}(t')}{720N(t')} + \frac{t'N^{(1)}(t')}{15N(t')}\right).
\]
where $\left(\frac{8(t')^2N^{(2)}(t')}{720N(t')} + \frac{t'N^{(1)}(t')}{15N(t')}\right)$ is the factor corresponding to the possible identifications for the partial shell 
\begin{center}
\begin{tabular}{|c|c|}
\hline
$\times a$ & $b$\\ \hline
$\times b$ & $a$ \\ \hline
\cellcolor[HTML]{FD6864} & $a$ \\ \hline
\cellcolor[HTML]{67FD9A} & $a$ \\ \hline
\cellcolor[HTML]{FD6864} & $b$ \\ \hline
\cellcolor[HTML]{67FD9A} & $b$ \\ \hline
\end{tabular}
\end{center}
which has an $a_{56}^{+}$ fragment, a $b_{34}^{+}$ fragment, a $u_{35}^{-}$ fragment, and a $u_{46}^{-}$ fragment.

Putting everything together, the total contribution for $P_2$ is 
\[
-G({\underline{\!\times\!}\,}_2^2)O(t)G(\underline{3})\left(\frac{6(t')^2N^{(2)}(t')}{N(t')} + \frac{54t'N^{(1)}(t')}{N(t')} + 90\right).
\]
\subsubsection{Shells Consistent with $P_3$}
The following shells are consistent with $P_3$:
\begin{enumerate}
\item There are $6$ shells isomorphic to
\begin{center}
\begin{tabular}{|c|c|}
\hline
$\times a$ & $b$\\ \hline
$a$ & $\times b$ \\ \hline
$a$ & $b$ \\ \hline
$a$ & $b$ \\ \hline
$\alpha$ & $\beta$ \\ \hline
$\alpha$ & $\beta$ \\ \hline
\end{tabular}
\end{center}
which are compatible with $P_3$. The contribution from these shells is 
\[
15\left(6G({\underline{\!\times\!}\,}_3^1)^{2}O(t)\right)\frac{48(t')^2N^{(2)}(t')}{720N(t')} = \frac{6G({\underline{\!\times\!}\,}_3^1)^{2}O(t)(t')^2N^{(2)}(t')}{N(t')}
\]
\item There are $12$ shells isomorphic to
\begin{center}
\begin{tabular}{|c|c|}
\hline
$\times a$ & $b$\\ \hline
$a$ & $\times b$ \\ \hline
$a$ & $b$ \\ \hline
$a$ & $b$ \\ \hline
$\alpha$ & $a$ \\ \hline
$\alpha$ & $a$ \\ \hline
\end{tabular}
\end{center}
which are compatible with $P_3$. The contribution from these shells is 
\[
15\left(12G({\underline{\!\times\!}\,}_3^1)^{2}O(t)\right)\frac{3t'N^{(1)}(t')}{15N(t')} = \frac{36G({\underline{\!\times\!}\,}_3^1)^{2}O(t)t'N^{(1)}(t')}{N(t')}
\]
\item There are $12$ shells isomorphic to
\begin{center}
\begin{tabular}{|c|c|}
\hline
$\times a$ & $b$\\ \hline
$a$ & $\times b$ \\ \hline
$a$ & $b$ \\ \hline
$a$ & $b$ \\ \hline
\cellcolor[HTML]{FD6864}$a$ & $\alpha$ \\ \hline
\cellcolor[HTML]{FD6864}$a$ & $\alpha$ \\ \hline
\end{tabular}
\end{center}
which are compatible with $P_3$. The contribution from these shells is 
\[
15\left(12G({\underline{\!\times\!}\,}_3^1)^{2}O(t)\right)\frac{3t'N^{(1)}(t')}{15N(t')} = \frac{36G({\underline{\!\times\!}\,}_3^1)^{2}O(t)t'N^{(1)}(t')}{N(t')}
\]
\item There are $6$ shells isomorphic to
\begin{center}
\begin{tabular}{|c|c|}
\hline
$\times a$ & $b$\\ \hline
$a$ & $\times b$ \\ \hline
$a$ & $b$ \\ \hline
$a$ & $b$ \\ \hline
$b$ & $a$ \\ \hline
$b$ & $a$ \\ \hline
\end{tabular}
\end{center}
which are compatible with $P_3$. The contribution from these shells is 
\[
15\left(6G({\underline{\!\times\!}\,}_3^1)^{2}O(t)\right) = 90G({\underline{\!\times\!}\,}_3^1)^{2}O(t).
\]
\item There are $6$ shells isomorphic to
\begin{center}
\begin{tabular}{|c|c|}
\hline
$\times a$ & $b$\\ \hline
$a$ & $\times b$ \\ \hline
$a$ & $b$ \\ \hline
$a$ & $b$ \\ \hline
\cellcolor[HTML]{FD6864}$a$ & \cellcolor[HTML]{FD6864}$b$ \\ \hline
\cellcolor[HTML]{FD6864}$a$ & \cellcolor[HTML]{FD6864}$b$ \\ \hline
\end{tabular}
\end{center}
which are compatible with $P_3$. The contribution from these shells is 
\[
15\left(6G({\underline{\!\times\!}\,}_3^1)^{2}O(t)\right) = 90G({\underline{\!\times\!}\,}_3^1)^{2}O(t).
\]
\item There are $24$ shells isomorphic to
\begin{center}
\begin{tabular}{|c|c|}
\hline
$\times a$ & $b$\\ \hline
$a$ & $\times b$ \\ \hline
$a$ & $b$ \\ \hline
$a$ & $\beta$ \\ \hline
$\alpha$ & $b$ \\ \hline
$\alpha$ & $\beta$ \\ \hline
\end{tabular}
\end{center}
which are compatible with $P_3$. The contribution from these shells is 
\[
15\left(24G({\underline{\!\times\!}\,}_3^1)^{2}O(t)\right)\frac{24(t')^2N^{(2)}(t')}{720N(t')} = \frac{12G({\underline{\!\times\!}\,}_3^1)^{2}O(t)(t')^2N^{(2)}(t')}{N(t')}
\]
\item There are $48$ shells isomorphic to
\begin{center}
\begin{tabular}{|c|c|}
\hline
$\times a$ & $b$\\ \hline
$a$ & $\times b$ \\ \hline
$a$ & $b$ \\ \hline
$a$ & $\alpha$ \\ \hline
\cellcolor[HTML]{FD6864}$a$ & $b$ \\ \hline
\cellcolor[HTML]{FD6864}$a$ & $\alpha$ \\ \hline
\end{tabular}
\end{center}
which are compatible with $P_3$. The contribution from these shells is 
\[
15\left(48G({\underline{\!\times\!}\,}_3^1)^{2}O(t)\right)\frac{3t'N^{(1)}(t')}{15N(t')} = \frac{144G({\underline{\!\times\!}\,}_3^1)^{2}O(t)t'N^{(1)}(t')}{N(t')}
\]
\item There are $24$ shells isomorphic to
\begin{center}
\begin{tabular}{|c|c|}
\hline
$\times a$ & $b$\\ \hline
$a$ & $\times b$ \\ \hline
$a$ & $b$ \\ \hline
$a$ & \cellcolor[HTML]{FD6864}$b$ \\ \hline
\cellcolor[HTML]{FD6864}$a$ & $b$ \\ \hline
\cellcolor[HTML]{FD6864}$a$ & \cellcolor[HTML]{FD6864}$b$ \\ \hline
\end{tabular}
\end{center}
which are compatible with $P_3$. The contribution from these shells is 
\[
15\left(24G({\underline{\!\times\!}\,}_3^1)^{2}O(t)\right) = 360G({\underline{\!\times\!}\,}_3^1)^{2}O(t).
\]
\item There are $6$ shells isomorphic to
\begin{center}
\begin{tabular}{|c|c|}
\hline
$\times a$ & $b$\\ \hline
$a$ & $\times b$ \\ \hline
$a$ & $\beta$ \\ \hline
$a$ & $\beta$ \\ \hline
$\alpha$ & $b$ \\ \hline
$\alpha$ & $b$ \\ \hline
\end{tabular}
\end{center}
which are compatible with $P_3$. The contribution from these shells is 
\[
15\left(6G({\underline{\!\times\!}\,}_3^1)^{2}O(t)\right)\frac{32(t')^2N^{(2)}(t')}{720N(t')} = \frac{4G({\underline{\!\times\!}\,}_3^1)^{2}O(t)(t')^2N^{(2)}(t')}{N(t')}
\]
\item There are $6$ shells isomorphic to
\begin{center}
\begin{tabular}{|c|c|}
\hline
$\times a$ & $b$\\ \hline
$a$ & $\times b$ \\ \hline
$a$ & $\alpha$ \\ \hline
$a$ & $\alpha$ \\ \hline
$\alpha$ & $b$ \\ \hline
$\alpha$ & $b$ \\ \hline
\end{tabular}
\end{center}
which are compatible with $P_3$. The contribution from these shells is 
\[
15\left(6G({\underline{\!\times\!}\,}_3^1)^{2}O(t)\right)\frac{t'N^{(1)}(t')}{15N(t')} = \frac{6G({\underline{\!\times\!}\,}_3^1)^{2}O(t)t'N^{(1)}(t')}{N(t')}
\]
\item There are $12$ shells isomorphic to
\begin{center}
\begin{tabular}{|c|c|}
\hline
$\times a$ & $b$\\ \hline
$a$ & $\times b$ \\ \hline
$a$ & $\alpha$ \\ \hline
$a$ & $\alpha$ \\ \hline
\cellcolor[HTML]{FD6864}$a$ & $b$ \\ \hline
\cellcolor[HTML]{FD6864}$a$ & $b$ \\ \hline
\end{tabular}
\end{center}
which are compatible with $P_3$. The contribution from these shells is 
\[
15\left(12G({\underline{\!\times\!}\,}_3^1)^{2}O(t)\right)\frac{3t'N^{(1)}(t')}{15N(t')} = \frac{36G({\underline{\!\times\!}\,}_3^1)^{2}O(t)t'N^{(1)}(t')}{N(t')}
\]
\item There are $6$ shells isomorphic to
\begin{center}
\begin{tabular}{|c|c|}
\hline
$\times a$ & $b$\\ \hline
$a$ & $\times b$ \\ \hline
$a$ & \cellcolor[HTML]{FD6864}$b$ \\ \hline
$a$ & \cellcolor[HTML]{FD6864}$b$ \\ \hline
\cellcolor[HTML]{FD6864}$a$ & $b$ \\ \hline
\cellcolor[HTML]{FD6864}$a$ & $b$ \\ \hline
\end{tabular}
\end{center}
which are compatible with $P_3$. The contribution from these shells is 
\[
15\left(6G({\underline{\!\times\!}\,}_3^1)^{2}O(t)\right) = 90G({\underline{\!\times\!}\,}_3^1)^{2}O(t).
\]
\item There are $24$ shells isomorphic to
\begin{center}
\begin{tabular}{|c|c|}
\hline
$\times a$ & $\beta$\\ \hline
$a$ & $\times b$ \\ \hline
$a$ & $b$ \\ \hline
$a$ & $b$ \\ \hline
$\alpha$ & $b$ \\ \hline
$\alpha$ & $\beta$ \\ \hline
\end{tabular}
\end{center}
which are compatible with $P_3$. The contribution from these shells is 
\[
15\left(24G({\underline{\!\times\!}\,}_3^1)^{2}O(t)\right)\frac{24(t')^2N^{(2)}(t')}{720N(t')} = \frac{12G({\underline{\!\times\!}\,}_3^1)^{2}O(t)(t')^2N^{(2)}(t')}{N(t')}
\]
\item There are $24$ shells isomorphic to
\begin{center}
\begin{tabular}{|c|c|}
\hline
$\times a$ & $\alpha$\\ \hline
$a$ & $\times b$ \\ \hline
$a$ & $b$ \\ \hline
$a$ & $b$ \\ \hline
\cellcolor[HTML]{FD6864}$a$ & $b$ \\ \hline
\cellcolor[HTML]{FD6864}$a$ & $\alpha$ \\ \hline
\end{tabular}
\end{center}
which are compatible with $P_3$. The contribution from these shells is 
\[
15\left(24G({\underline{\!\times\!}\,}_3^1)^{2}O(t)\right)\frac{3t'N^{(1)}(t')}{15N(t')} = \frac{72G({\underline{\!\times\!}\,}_3^1)^{2}O(t)t'N^{(1)}(t')}{N(t')}
\]
\item There are $24$ shells isomorphic to
\begin{center}
\begin{tabular}{|c|c|}
\hline
$\times a$ & \cellcolor[HTML]{FD6864}$b$\\ \hline
$a$ & $\times b$ \\ \hline
$a$ & $b$ \\ \hline
$a$ & $b$ \\ \hline
$\alpha$ & $b$ \\ \hline
$\alpha$ & \cellcolor[HTML]{FD6864}$b$ \\ \hline
\end{tabular}
\end{center}
which are compatible with $P_3$. The contribution from these shells is 
\[
15\left(24G({\underline{\!\times\!}\,}_3^1)^{2}O(t)\right)\frac{3t'N^{(1)}(t')}{15N(t')} = \frac{72G({\underline{\!\times\!}\,}_3^1)^{2}O(t)t'N^{(1)}(t')}{N(t')}
\]
\item There are $24$ shells isomorphic to
\begin{center}
\begin{tabular}{|c|c|}
\hline
$\times a$ & \cellcolor[HTML]{FD6864}$b$\\ \hline
$a$ & $\times b$ \\ \hline
$a$ & $b$ \\ \hline
$a$ & $b$ \\ \hline
\cellcolor[HTML]{FD6864}$a$ & $b$ \\ \hline
\cellcolor[HTML]{FD6864}$a$ & \cellcolor[HTML]{FD6864}$b$ \\ \hline
\end{tabular}
\end{center}
which are compatible with $P_3$. The contribution from these shells is 
\[
15\left(24G({\underline{\!\times\!}\,}_3^1)^{2}O(t)\right) = 360G({\underline{\!\times\!}\,}_3^1)^{2}O(t).
\]
\item There are $24$ shells isomorphic to
\begin{center}
\begin{tabular}{|c|c|}
\hline
$\times a$ & $\beta$\\ \hline
$a$ & $\times b$ \\ \hline
$a$ & $b$ \\ \hline
$a$ & $\beta$ \\ \hline
$\alpha$ & $b$ \\ \hline
$\alpha$ & $b$ \\ \hline
\end{tabular}
\end{center}
which are compatible with $P_3$. The contribution from these shells is 
\[
15\left(24G({\underline{\!\times\!}\,}_3^1)^{2}O(t)\right)\frac{32(t')^2N^{(2)}(t')}{720N(t')} = \frac{16G({\underline{\!\times\!}\,}_3^1)^{2}O(t)(t')^2N^{(2)}(t')}{N(t')}
\]
\item There are $24$ shells isomorphic to
\begin{center}
\begin{tabular}{|c|c|}
\hline
$\times a$ & $\alpha$\\ \hline
$a$ & $\times b$ \\ \hline
$a$ & $b$ \\ \hline
$a$ & $\alpha$ \\ \hline
$\alpha$ & $b$ \\ \hline
$\alpha$ & $b$ \\ \hline
\end{tabular}
\end{center}
which are compatible with $P_3$. The contribution from these shells is 
\[
15\left(24G({\underline{\!\times\!}\,}_3^1)^{2}O(t)\right)\frac{t'N^{(1)}(t')}{15N(t')} = \frac{24G({\underline{\!\times\!}\,}_3^1)^{2}O(t)t'N^{(1)}(t')}{N(t')}
\]
\item There are $24$ shells isomorphic to
\begin{center}
\begin{tabular}{|c|c|}
\hline
$\times a$ & $\alpha$\\ \hline
$a$ & $\times b$ \\ \hline
$a$ & $b$ \\ \hline
$a$ & $\alpha$ \\ \hline
\cellcolor[HTML]{FD6864}$a$ & $b$ \\ \hline
\cellcolor[HTML]{FD6864}$a$ & $b$ \\ \hline
\end{tabular}
\end{center}
which are compatible with $P_3$. The contribution from these shells is 
\[
15\left(24G({\underline{\!\times\!}\,}_3^1)^{2}O(t)\right)\frac{3t'N^{(1)}(t')}{15N(t')} = \frac{72G({\underline{\!\times\!}\,}_3^1)^{2}O(t)t'N^{(1)}(t')}{N(t')}
\]
\item There are $24$ shells isomorphic to
\begin{center}
\begin{tabular}{|c|c|}
\hline
$\times a$ & \cellcolor[HTML]{FD6864}$b$\\ \hline
$a$ & $\times b$ \\ \hline
$a$ & $b$ \\ \hline
$a$ & \cellcolor[HTML]{FD6864}$b$ \\ \hline
$\alpha$ & $b$ \\ \hline
$\alpha$ & $b$ \\ \hline
\end{tabular}
\end{center}
which are compatible with $P_3$. The contribution from these shells is 
\[
15\left(24G({\underline{\!\times\!}\,}_3^1)^{2}O(t)\right)\frac{3t'N^{(1)}(t')}{15N(t')} = \frac{72G({\underline{\!\times\!}\,}_3^1)^{2}O(t)t'N^{(1)}(t')}{N(t')}
\]
\item There are $24$ shells isomorphic to
\begin{center}
\begin{tabular}{|c|c|}
\hline
$\times a$ & \cellcolor[HTML]{FD6864}$b$\\ \hline
$a$ & $\times b$ \\ \hline
$a$ & $b$ \\ \hline
$a$ & \cellcolor[HTML]{FD6864}$b$ \\ \hline
\cellcolor[HTML]{FD6864}$a$ & $b$ \\ \hline
\cellcolor[HTML]{FD6864}$a$ & $b$ \\ \hline
\end{tabular}
\end{center}
which are compatible with $P_3$. The contribution from these shells is 
\[
15\left(24G({\underline{\!\times\!}\,}_3^1)^{2}O(t)\right) = 360G({\underline{\!\times\!}\,}_3^1)^{2}O(t).
\]
\item There are $4$ shells isomorphic to
\begin{center}
\begin{tabular}{|c|c|}
\hline
$\times a$ & $\beta$\\ \hline
$\alpha$ & $\times b$ \\ \hline
$a$ & $b$ \\ \hline
$a$ & $b$ \\ \hline
$a$ & $b$ \\ \hline
$\alpha$ & $\beta$ \\ \hline
\end{tabular}
\end{center}
which are compatible with $P_3$. The contribution from these shells is 
\[
15\left(4G({\underline{\!\times\!}\,}_3^1)^{2}O(t)\right)\frac{24(t')^2N^{(2)}(t')}{720N(t')} = \frac{2G({\underline{\!\times\!}\,}_3^1)^{2}O(t)(t')^2N^{(2)}(t')}{N(t')}.
\]
\item There are $8$ shells isomorphic to
\begin{center}
\begin{tabular}{|c|c|}
\hline
$\times a$ & $\alpha$\\ \hline
\cellcolor[HTML]{FD6864}$a$ & $\times b$ \\ \hline
$a$ & $b$ \\ \hline
$a$ & $b$ \\ \hline
$a$ & $b$ \\ \hline
\cellcolor[HTML]{FD6864}$a$ & $\alpha$ \\ \hline
\end{tabular}
\end{center}
which are compatible with $P_3$. The contribution from these shells is 
\[
15\left(8G({\underline{\!\times\!}\,}_3^1)^{2}O(t)\right)\frac{3t'N^{(1)}(t')}{15N(t')} = \frac{24G({\underline{\!\times\!}\,}_3^1)^{2}O(t)t'N^{(1)}(t')}{N(t')}.
\]
\item There are $4$ shells isomorphic to
\begin{center}
\begin{tabular}{|c|c|}
\hline
$\times a$ & \cellcolor[HTML]{FD6864}$b$\\ \hline
\cellcolor[HTML]{FD6864}$a$ & $\times b$ \\ \hline
$a$ & $b$ \\ \hline
$a$ & $b$ \\ \hline
$a$ & $b$ \\ \hline
\cellcolor[HTML]{FD6864}$a$ & \cellcolor[HTML]{FD6864}$b$ \\ \hline
\end{tabular}
\end{center}
which are compatible with $P_3$. The contribution from these shells is 
\[
15\left(4G({\underline{\!\times\!}\,}_3^1)^{2}O(t)\right) = 60G({\underline{\!\times\!}\,}_3^1)^{2}O(t).
\]
\item There are $12$ shells isomorphic to
\begin{center}
\begin{tabular}{|c|c|}
\hline
$\times a$ & $\beta$\\ \hline
$\alpha$ & $\times b$ \\ \hline
$a$ & $b$ \\ \hline
$a$ & $b$ \\ \hline
$a$ & $\beta$ \\ \hline
$\alpha$ & $b$ \\ \hline
\end{tabular}
\end{center}
which are compatible with $P_3$. The contribution from these shells is 
\[
15\left(12G({\underline{\!\times\!}\,}_3^1)^{2}O(t)\right)\frac{32(t')^2N^{(2)}(t')}{720N(t')} = \frac{8G({\underline{\!\times\!}\,}_3^1)^{2}O(t)(t')^2N^{(2)}(t')}{N(t')}
\]
\item There are $12$ shells isomorphic to
\begin{center}
\begin{tabular}{|c|c|}
\hline
$\times a$ & $\alpha$\\ \hline
$\alpha$ & $\times b$ \\ \hline
$a$ & $b$ \\ \hline
$a$ & $b$ \\ \hline
$a$ & $\alpha$ \\ \hline
$\alpha$ & $b$ \\ \hline
\end{tabular}
\end{center}
which are compatible with $P_3$. The contribution from these shells is 
\[
15\left(12G({\underline{\!\times\!}\,}_3^1)^{2}O(t)\right)\frac{t'N^{(1)}(t')}{15N(t')} = \frac{12G({\underline{\!\times\!}\,}_3^1)^{2}O(t)t'N^{(1)}(t')}{N(t')}.
\]
\item There are $24$ shells isomorphic to
\begin{center}
\begin{tabular}{|c|c|}
\hline
$\times a$ & $\alpha$\\ \hline
\cellcolor[HTML]{FD6864}$a$ & $\times b$ \\ \hline
$a$ & $b$ \\ \hline
$a$ & $b$ \\ \hline
$a$ & $\alpha$ \\ \hline
\cellcolor[HTML]{FD6864}$a$ & $b$ \\ \hline
\end{tabular}
\end{center}
which are compatible with $P_3$. The contribution from these shells is 
\[
15\left(24G({\underline{\!\times\!}\,}_3^1)^{2}O(t)\right)\frac{3t'N^{(1)}(t')}{15N(t')} = \frac{72G({\underline{\!\times\!}\,}_3^1)^{2}O(t)t'N^{(1)}(t')}{N(t')}.
\]
\item There are $12$ shells isomorphic to
\begin{center}
\begin{tabular}{|c|c|}
\hline
$\times a$ & \cellcolor[HTML]{FD6864}$b$\\ \hline
\cellcolor[HTML]{FD6864}$a$ & $\times b$ \\ \hline
$a$ & $b$ \\ \hline
$a$ & $b$ \\ \hline
$a$ & \cellcolor[HTML]{FD6864}$b$ \\ \hline
\cellcolor[HTML]{FD6864}$a$ & $b$ \\ \hline
\end{tabular}
\end{center}
which are compatible with $P_3$. The contribution from these shells is 
\[
15\left(12G({\underline{\!\times\!}\,}_3^1)^{2}O(t)\right) = 180G({\underline{\!\times\!}\,}_3^1)^{2}O(t).
\]
\item There are $4$ shells isomorphic to
\begin{center}
\begin{tabular}{|c|c|}
\hline
$\times a$ & $\beta$\\ \hline
$\alpha$ & $\times b$ \\ \hline
$b$ & $a$ \\ \hline
$b$ & $a$ \\ \hline
$b$ & $a$ \\ \hline
$\alpha$ & $\beta$ \\ \hline
\end{tabular}
\end{center}
which are compatible with $P_3$. These shells have negative sign so the contribution from these shells is 
\[
-15\left(4G({\underline{\!\times\!}\,}_3^1)^{2}O(t)\right)\frac{24(t')^2N^{(2)}(t')}{720N(t')} = -\frac{2G({\underline{\!\times\!}\,}_3^1)^{2}O(t)(t')^2N^{(2)}(t')}{N(t')}.
\]
\item There are $8$ shells isomorphic to
\begin{center}
\begin{tabular}{|c|c|}
\hline
$\times a$ & $\alpha$\\ \hline
$a$ & $\times b$ \\ \hline
$b$ & $a$ \\ \hline
$b$ & $a$ \\ \hline
$b$ & $a$ \\ \hline
$a$ & $\alpha$ \\ \hline
\end{tabular}
\end{center}
which are compatible with $P_3$. These shells have negative sign so the contribution from these shells is
\[
-15\left(8G({\underline{\!\times\!}\,}_3^1)^{2}O(t)\right)\frac{3t'N^{(1)}(t')}{15N(t')} = -\frac{24G({\underline{\!\times\!}\,}_3^1)^{2}O(t)t'N^{(1)}(t')}{N(t')}.
\]
\item There are $4$ shells isomorphic to
\begin{center}
\begin{tabular}{|c|c|}
\hline
$\times a$ & $b$\\ \hline
$a$ & $\times b$ \\ \hline
$b$ & $a$ \\ \hline
$b$ & $a$ \\ \hline
$b$ & $a$ \\ \hline
$a$ & $b$ \\ \hline
\end{tabular}
\end{center}
which are compatible with $P_3$. These shells have negative sign so the contribution from these shells is
\[
-15\left(4G({\underline{\!\times\!}\,}_3^1)^{2}O(t)\right) = -60G({\underline{\!\times\!}\,}_3^1)^{2}O(t).
\]
\item There are $12$ shells isomorphic to
\begin{center}
\begin{tabular}{|c|c|}
\hline
$\times a$ & $\beta$\\ \hline
$\alpha$ & $\times b$ \\ \hline
$b$ & $a$ \\ \hline
$b$ & $a$ \\ \hline
$b$ & $\beta$ \\ \hline
$\alpha$ & $a$ \\ \hline
\end{tabular}
\end{center}
which are compatible with $P_3$. These shells have positive sign so the contribution from these shells is 
\[
15\left(12G({\underline{\!\times\!}\,}_3^1)^{2}O(t)\right)\frac{8(t')^2N^{(2)}(t')}{720N(t')} = \frac{2G({\underline{\!\times\!}\,}_3^1)^{2}O(t)(t')^2N^{(2)}(t')}{N(t')}
\]
\item There are $12$ shells isomorphic to
\begin{center}
\begin{tabular}{|c|c|}
\hline
$\times a$ & $\alpha$\\ \hline
$\alpha$ & $\times b$ \\ \hline
$b$ & $a$ \\ \hline
$b$ & $a$ \\ \hline
$b$ & $\alpha$ \\ \hline
$\alpha$ & $a$ \\ \hline
\end{tabular}
\end{center}
which are compatible with $P_3$. The contribution from these shells is 
\[
15\left(12G({\underline{\!\times\!}\,}_3^1)^{2}O(t)\right)\frac{t'N^{(1)}(t')}{15N(t')} = \frac{12G({\underline{\!\times\!}\,}_3^1)^{2}O(t)t'N^{(1)}(t')}{N(t')}.
\]
\item There are $8$ shells isomorphic to
\begin{center}
\begin{tabular}{|c|c|}
\hline
$\times a$ & $\alpha$\\ \hline
$a$ & $\times b$ \\ \hline
$a$ & $b$ \\ \hline
$a$ & $b$ \\ \hline
$a$ & $b$ \\ \hline
$a$ & $\alpha$ \\ \hline
\end{tabular}
\end{center}
which are compatible with $P_3$. The contribution from these shells is 
\[
15\left(8G({\underline{\!\times\!}\,}_3^1)G({\underline{\!\times\!}\,}_5^1)O(t)\right)\frac{3t'N^{(1)}(t')}{15N(t')} = \frac{24G({\underline{\!\times\!}\,}_3^1)G({\underline{\!\times\!}\,}_5^1)O(t)t'N^{(1)}(t')}{N(t')}.
\]
\item There are $8$ shells isomorphic to
\begin{center}
\begin{tabular}{|c|c|}
\hline
$\times a$ & \cellcolor[HTML]{FD6864}$b$\\ \hline
$a$ & $\times b$ \\ \hline
$a$ & $b$ \\ \hline
$a$ & $b$ \\ \hline
$a$ & $b$ \\ \hline
$a$ & \cellcolor[HTML]{FD6864}$b$ \\ \hline
\end{tabular}
\end{center}
which are compatible with $P_3$. The contribution from these shells is 
\[
15\left(8G({\underline{\!\times\!}\,}_3^1)G({\underline{\!\times\!}\,}_5^1)O(t)\right) = 120G({\underline{\!\times\!}\,}_3^1)G({\underline{\!\times\!}\,}_5^1)O(t).
\]
\item There are $12$ shells isomorphic to
\begin{center}
\begin{tabular}{|c|c|}
\hline
$\times a$ & $b$\\ \hline
$a$ & $\times b$ \\ \hline
$a$ & $b$ \\ \hline
$a$ & $b$ \\ \hline
$a$ & $\alpha$ \\ \hline
$a$ & $\alpha$ \\ \hline
\end{tabular}
\end{center}
which are compatible with $P_3$. The contribution from these shells is 
\[
15\left(12G({\underline{\!\times\!}\,}_3^1)G({\underline{\!\times\!}\,}_5^1)O(t)\right)\frac{3t'N^{(1)}(t')}{15N(t')} = \frac{36G({\underline{\!\times\!}\,}_3^1)G({\underline{\!\times\!}\,}_5^1)O(t)t'N^{(1)}(t')}{N(t')}.
\]
\item There are $12$ shells isomorphic to
\begin{center}
\begin{tabular}{|c|c|}
\hline
$\times a$ & $b$\\ \hline
$a$ & $\times b$ \\ \hline
$a$ & $b$ \\ \hline
$a$ & $b$ \\ \hline
$a$ & \cellcolor[HTML]{FD6864}$b$ \\ \hline
$a$ & \cellcolor[HTML]{FD6864}$b$ \\ \hline
\end{tabular}
\end{center}
which are compatible with $P_3$. The contribution from these shells is 
\[
15\left(12G({\underline{\!\times\!}\,}_3^1)G({\underline{\!\times\!}\,}_5^1)O(t)\right) = 180G({\underline{\!\times\!}\,}_3^1)G({\underline{\!\times\!}\,}_5^1)O(t).
\]
\item \ 
\begin{center}
\begin{tabular}{|c|c|}
\hline
$\times a$ & $b$\\ \hline
$a$ & $\times b$ \\ \hline
$a$ & $b$ \\ \hline
$a$ & $b$ \\ \hline
$a$ & $b$ \\ \hline
$a$ & $b$ \\ \hline
\end{tabular}
\end{center}
The contribution from this shell is $15\left(G({\underline{\!\times\!}\,}_5^1)^{2}O(t)\right) = 15G({\underline{\!\times\!}\,}_5^1)^{2}O(t)$.
\end{enumerate}
To obtain the total contribution from $P_3$, we make the following observations:
\begin{enumerate}
\item The total contribution from cases 1-5 is 
\[
15\left(6G({\underline{\!\times\!}\,}_3^1)^{2}\right)O(t)\left(\frac{48(t')^2N^{(2)}(t')}{720N(t')} + 4\frac{3t'N^{(1)}(t')}{15N(t')} + 2\right)
\]
where the factor of $\frac{48(t')^2N^{(2)}(t')}{720N(t')} + 4\frac{3t'N^{(1)}(t')}{15N(t')} + 2$ comes from having an $a_{56}^{+}$ fragment, a $b_{56}^{+}$ fragment, a $u_{56}^{-}$ fragment, and a $u_{56}^{-}$ fragment.
\item The total contribution from cases 6-8 is 
\[
15\left(24G({\underline{\!\times\!}\,}_3^1)^{2}\right)O(t)\left(\frac{24(t')^2N^{(2)}(t')}{720N(t')} + 2\frac{3t'N^{(1)}(t')}{15N(t')} + 1\right)
\]
where the factor of $\frac{24(t')^2N^{(2)}(t')}{720N(t')} + 2\frac{3t'N^{(1)}(t')}{15N(t')} + 1$ comes from having an $a_{56}^{+}$ fragment, a $b_{46}^{+}$ fragment, a $u_{56}^{-}$ fragment, and a $u_{46}^{-}$ fragment.
\item The total contribution from cases 9-12 is 
\[
15\left(6G({\underline{\!\times\!}\,}_3^1)^{2}\right)O(t)\left(\frac{32(t')^2N^{(2)}(t')}{720N(t')} + 2\frac{3t'N^{(1)}(t')}{15N(t')} + \frac{t'N^{(1)}(t')}{15N(t')} + 1\right)
\]
where the factor of $\frac{32(t')^2N^{(2)}(t')}{720N(t')} + 2\frac{3t'N^{(1)}(t')}{15N(t')} + \frac{t'N^{(1)}(t')}{15N(t')} + 1$ comes from having an $a_{56}^{+}$ fragment, a $b_{34}^{+}$ fragment, a $u_{56}^{-}$ fragment, and a $u_{34}^{-}$ fragment.
\item The total contribution from cases 13-16 is 
\[
15\left(24G({\underline{\!\times\!}\,}_3^1)^{2}\right)O(t)\left(\frac{24(t')^2N^{(2)}(t')}{720N(t')} + 2\frac{3t'N^{(1)}(t')}{15N(t')} + 1\right)
\]
where the factor of $\frac{24(t')^2N^{(2)}(t')}{720N(t')} + 2\frac{3t'N^{(1)}(t')}{15N(t')} + 1$ comes from having an $a_{56}^{+}$ fragment, a $b_{16}^{+}$ fragment, a $u_{56}^{-}$ fragment, and a $u_{16}^{-}$ fragment.
\item The total contribution from cases 17-21 is 
\[
15\left(24G({\underline{\!\times\!}\,}_3^1)^{2}\right)O(t)\left(\frac{32(t')^2N^{(2)}(t')}{720N(t')} + 2\frac{3t'N^{(1)}(t')}{15N(t')} + \frac{t'N^{(1)}(t')}{15N(t')} + 1\right)
\]
where the factor of $\frac{32(t')^2N^{(2)}(t')}{720N(t')} + 2\frac{3t'N^{(1)}(t')}{15N(t')} + \frac{t'N^{(1)}(t')}{15N(t')} + 1$ comes from having an $a_{56}^{+}$ fragment, a $b_{14}^{+}$ fragment, a $u_{56}^{-}$ fragment, and a $u_{14}^{-}$ fragment.
\item The total contribution from cases 22-24 is 
\[
15\left(4G({\underline{\!\times\!}\,}_3^1)^{2}\right)O(t)\left(\frac{24(t')^2N^{(2)}(t')}{720N(t')} + 2\frac{3t'N^{(1)}(t')}{15N(t')} + 1\right)
\]
where the factor of $\frac{24(t')^2N^{(2)}(t')}{720N(t')} + 2\frac{3t'N^{(1)}(t')}{15N(t')} + 1$ comes from having an $a_{26}^{+}$ fragment, a $b_{16}^{+}$ fragment, a $u_{26}^{-}$ fragment, and a $u_{16}^{-}$ fragment.
\item The total contribution from cases 25-28 is 
\[
15\left(12G({\underline{\!\times\!}\,}_3^1)^{2}\right)O(t)\left(\frac{32(t')^2N^{(2)}(t')}{720N(t')} + 2\frac{3t'N^{(1)}(t')}{15N(t')} + \frac{t'N^{(1)}(t')}{15N(t')} + 1\right)
\]
where the factor of $\frac{32(t')^2N^{(2)}(t')}{720N(t')} + 2\frac{3t'N^{(1)}(t')}{15N(t')} + \frac{t'N^{(1)}(t')}{15N(t')} + 1$ comes from having an $a_{26}^{+}$ fragment, a $b_{15}^{+}$ fragment, a $u_{26}^{-}$ fragment, and a $u_{15}^{-}$ fragment.
\item The total contribution from cases 29-31 is 
\[
-15\left(4G({\underline{\!\times\!}\,}_3^1)^{2}\right)O(t)\left(\frac{24(t')^2N^{(2)}(t')}{720N(t')} + 2\frac{3t'N^{(1)}(t')}{15N(t')} + 1\right)
\]
where the factor of $\frac{24(t')^2N^{(2)}(t')}{720N(t')} + 2\frac{3t'N^{(1)}(t')}{15N(t')} + 1$ comes from having an $a_{26}^{+}$ fragment, a $b_{16}^{+}$ fragment, a $u_{26}^{-}$ fragment, and a $u_{16}^{-}$ fragment.
\item The total contribution from cases 32-33 is 
\[
15\left(12G({\underline{\!\times\!}\,}_3^1)^{2}\right)O(t)\left(\frac{8(t')^2N^{(2)}(t')}{720N(t')} + \frac{t'N^{(1)}(t')}{15N(t')}\right)
\]
where the factor of $\frac{8(t')^2N^{(2)}(t')}{720N(t')} + \frac{t'N^{(1)}(t')}{15N(t')}$ comes from having an $a_{25}^{+}$ fragment, a $b_{16}^{+}$ fragment, a $u_{26}^{-}$ fragment, and a $u_{15}^{-}$ fragment.
\item The total contribution from cases 34-35 is 
\[
15\left(8G({\underline{\!\times\!}\,}_3^1)G({\underline{\!\times\!}\,}_5^1)\right)O(t)\left(\frac{3t'N^{(1)}(t')}{15N(t')}+1\right)
\]
where the factor of $\frac{3t'N^{(1)}(t')}{15N(t')}+1$ comes from having a $b_{16}^{+}$ fragment and a $u_{16}^{-}$ fragment.
\item The total contribution from cases 36-37 is 
\[
15\left(12G({\underline{\!\times\!}\,}_3^1)G({\underline{\!\times\!}\,}_5^1)\right)O(t)\left(\frac{3t'N^{(1)}(t')}{15N(t')}+1\right)
\]
where the factor of $\frac{3t'N^{(1)}(t')}{15N(t')}+1$ comes from having a $b_{56}^{+}$ fragment and a $u_{56}^{-}$ fragment.
\end{enumerate}
Summing up these contributions, the total contribution from shells which are consistent with $P_3$ is 
\begin{align*}
&G({\underline{\!\times\!}\,}_3^1)^{2}O(t)\left(60\frac{(t')^2N^{(2)}(t')}{N(t')} + 666\frac{t'N^{(1)}(t')}{N(t')} + 1530\right)\\
&+G({\underline{\!\times\!}\,}_3^1)G({\underline{\!\times\!}\,}_5^1)O(t)\left(60\frac{t'N^{(1)}(t')}{N(t')} + 300\right) + 15G({\underline{\!\times\!}\,}_5^1)^{2}O(t)
\end{align*}
\printbibliography[heading=bibintoc]

\end{document}